\documentclass [11pt]{amsart}
\usepackage{amscd,amsmath,amsthm,amssymb,amsfonts,epsfig}
\usepackage{bbm}
\usepackage{graphicx}
\usepackage[utf8]{inputenc}
\usepackage{xcolor}
\usepackage{mathrsfs}  
\usepackage{esvect}
\usepackage{enumitem}
\usepackage{indentfirst}
\usepackage{caption}
\usepackage{subcaption}
\usepackage[export]{adjustbox}
\usepackage{tkz-euclide}
\usepackage{tcolorbox}
\usepackage{pgfplots}
\pgfplotsset{compat=newest}
\usepgfplotslibrary{fillbetween}
\usepackage{comment}
\usepackage{url}
\usepackage{hyperref}

\usepackage[a4paper,top=3cm,bottom=2cm,left=3cm,right=3cm,marginparwidth=1.75cm]{geometry}

\usetikzlibrary{intersections,positioning, calc}

\renewcommand{\bar}{\overline}
\newcommand{\R}{\mathbb{R}}

\newcommand{\D}{\nabla}
\newcommand{\pa}{\partial}
\newcommand{\Div}{\mathrm{div}}

\newcommand{\bn}{\mathbf{n}}
\newcommand{\bo}{\mathbf{0}}
\newcommand{\bp}{\mathbf{p}}

\newcommand{\bq}{\mathbf{q}}

\newcommand{\bv}{\mathbf{v}}

\newcommand{\bt}{\mathbf{t}}

\newcommand{\cK}{\mathcal{K}}
\newcommand{\cM}{\mathcal{M}}

\newcommand{\cD}{\mathcal{D}}
\newcommand{\cH}{\mathcal{H}}
\newcommand{\cG}{\mathcal{G}}

\newcommand{\ent}{\mathrm{Ent}}

\newcommand{\al}{\alpha}

\newcommand{\ga}{\gamma}
\newcommand{\Ga}{\Gamma}
\newcommand{\la}{\lambda}

\newcommand{\si}{\sigma}

\newcommand{\ka}{\kappa}
\newcommand{\vare}{\varepsilon}

\newcommand{\Om}{\Omega}

\newcommand{\Th}{\Theta}

\newtheorem{theorem}{Theorem}[section]
\newtheorem{corollary}[theorem]{Corollary}
\newtheorem{lemma}[theorem]{Lemma}
\newtheorem{proposition}[theorem]{Proposition}

\theoremstyle{remark}
\newtheorem{remark}[theorem]{Remark}

\theoremstyle{definition}
\newtheorem{definition}[theorem]{Definition}

\title{Uniqueness of tangent flows at infinity for finite-entropy shortening curves}

\author{Kyeongsu Choi}
\address{KC: School of Mathematics, Korea Institute for Advanced Study, 85 Hoegiro, Dongdaemun-gu, Seoul 02455, Republic of Korea}
\email{choiks@kias.re.kr}

\author{Dong-Hwi Seo}
\address{DS:Research Institute of Mathematics, Seoul National University, 1 Gwanak-ro, Gwanak-gu, Seoul 08826, Republic of Korea}
\email{donghwi.seo@snu.ac.kr}

\author{Wei-Bo Su}
\address{WS: Mathematics Division, National Center for Theoretical Sciences, Cosmology Building, No. 1, Sec. 4, Roosevelt Rd., Taipei City 106, Taiwan}
\email{weibo.su@ncts.ntu.edu.tw}

\author{Kai-Wei Zhao}
\address{KZ: Department of Mathematics, University of Notre Dame, Notre Dame, 46556, Indiana, USA}
\email{kzhao4@nd.edu }

\begin{document}

\begin{abstract}
In this paper, we prove that an ancient smooth curve-shortening flow with finite entropy embedded in $\mathbb{R}^2$ has a unique tangent flow at infinity. To this end, we show that its rescaled flows backwardly converge to a line with multiplicity $m\geq 3$ exponentially fast in any compact region, unless the flow is a shrinking circle, a static line, a paper clip, or a translating grim reaper. In addition, we figure out the exact numbers of tips, vertices, and inflection points of the curves at negative enough time. Moreover,  the exponential growth rate of graphical radius and the convergence of vertex regions to grim reaper curves will be shown. 
\end{abstract}

\maketitle 

\section{Introduction}

A family of planar curves $M_t$ is a solution to the \textit{curve-shortening flow} if there exists a smooth family of immersions $\ga: \mathbb{L}\times I \to \R^2$, where $I\subset \R$ is an interval and $\mathbb{L}  = \mathbb{R}$ or $\mathbb{S}^{1}$, such that $M_t = \ga(\mathbb{L},t)$ and satisfies the evolution equation
\begin{equation}\label{eq: CSF}
    \ga_t = \ga_{ss}.
\end{equation}
By the work of Gage-Hamilton \cite{GH86} and Grayson \cite{G87}, a closed flow embedded in $\mathbb{R}^2$ converges to a circle after normalization. Also, by Angenent \cite{An91}, a closed, immersed flow converges to an Abresch-Langer curve or a grim reaper curve after normalization. These curves of limiting shape are self-similar solutions. To be specific, shrinking circles, shrinking Abresch-Langer curves, and translating grim reapers are curve-shortening flows. Note that we call a flow $\mathcal{M}=\cup_{t\in I}(M_t,t)$ is \textit{ancient} if $(-\infty,t_0)\subset I$ for some $t_0$. Hence, the shrinking curves and translating curves are ancient flows.

\bigskip

There are non-self-similar ancient flows such as paper clips, ancient trombones \cite{AY21}, and truncated Yin-Yang spirals \cite{Ch22,AIOZ23}. However, the parabolic Liouville theorem allows us to classify ancient flows under natural conditions. For example, the shrinking circle, the translating grim reaper curves, and the paper clips are the only ancient convex flows, according to \cite{BLT20,DHS10, Wan11CSM}. In addition, there has been a recent outburst of research \cite{ADS19,ADS20,BC19,BC21,CDDHS22,CHH21,CHH22,DH24} on convex ancient solutions to its higher-dimensional analogy, the mean curvature flow. A key idea of recent research is that the rescaled convex flow 
\begin{align}
    \bar{M}_\tau = (-t)^{-\frac{1}{2}}M_t, \qquad \text{where}\quad \tau=-\log(-t),
\end{align}
converges to a sphere or a cylinder of multiplicity one as $\tau\to -\infty$. Hence, one can investigate its asymptotic behavior near the limit sphere or cylinder. 
In fact, if an ancient mean curvature flow backwardly converges to a shrinking flow, then the study of its asymptotic behavior near the limit shrinker yields classification results \cite{CCMS24,CCS23,CHH22,CHHW22,CM22}. However, previous research for ancient mean curvature flows was restricted to the case where the limit shrinker has a multiplicity one. 
In this paper, we will work on ancient curve-shortening flows embedded in $\mathbb{R}^2$ such that their rescaled flows backwardly converge to shrinkers with multiplicity, such as paper clips and ancient trombones.

\bigskip

To study ancient flows whose limits are shrinkers with multiplicity, we recall Colding-Minicozzi's \textit{Entropy} \cite{CM12}. See \eqref{eq:Entropy} for the definition. If $\cM = \cup_{t\in I}(M_t,t)$ moves by the curve-shortening flow, then $t \mapsto \ent[M_t]$ is non-increasing by Huisken's monotonicity formula \cite{Hu90}. See \cite[Lemma 1.11]{CM12}. Hence, we can also define \emph{Entropy} of the flow $\cM$ by
\begin{align}
    \ent[\cM] := \sup_{t\in I} \ent[M_t],
\end{align}
and we say that an ancient flow $\cM$ has \textit{finite-entropy} if $\ent[\cM]<+\infty$. For example, the entropy of a static line $\mathcal{L}$ and a shrinking circle $\mathcal{C}$ are, respectively,
\begin{align}
    &\ent[\mathcal{L}] = 1 \quad \text{and} \quad \ent[\mathcal{C}] = \sqrt{2\pi/e} \approx 1.47.
\end{align}
Since the line's entropy is $1$, the entropies of paper clips and translating grim reaper curves are $2$. Also, the entropy of an ancient trombone asymptotic to $m$-lines is $m \in \mathbb{N}$. However, the truncated Yin-Yang spirals \cite{Ch22, AIOZ23} have infinite entropy. 
\bigskip

Note that if the entropy of a shrinker is bounded and a rescaled ancient flow backwardly converges to the shrinker $\Gamma$ with multiplicity $m\in \mathbb{N}$, then the entropy of the flow is $m\cdot \text{Ent}(\Gamma)<+\infty$. 
On the other hand, we let $\mathcal{D}_{1/\lambda}(\mathcal{M}-X)$ denote the curve-shortening flow obtained by translating $X\in \mathbb{R}^2\times I$ to the spacetime origin and parabolically rescaling by $\lambda^{-1}>0$. 
See (\ref{eq: parabolic rescaling}) for more details. 
Then, for any fixed $X_0$, a sequence $\lambda_i\to +\infty$, and an ancient flow $\cM$ with finite entropy, the sequence $\mathcal{D}_{\lambda_i}(\mathcal{M}-X_0)$ has a subsequential limit $\overline{\cM}$, which is a self-shrinking flow of multiplicity $m\in \mathbb{N}$. The subsequential limit $\overline{\cM}$ is called a \textit{tangent flow at infinity}. Now, we provide the main results of this paper.

\begin{theorem}[unique tangent flow at infinity]\label{thm:main.unique}
An ancient finite-entropy smooth curve-shortening flow embedded in $\mathbb{R}^2$ has a unique tangent flow at infinity. Moreover, the tangent flow at infinity is a straight line with multiplicity $m\geq 2$, unless the flow is a shrinking circle or a static line.
\end{theorem}

\bigskip

The tangent flow (at infinity) is the analogy to the tangent cone (at infinity) of minimal surfaces and solutions to elliptic equations. The uniqueness of the tangent cone has been widely studied in various contexts of geometry and analysis, after the seminal uniqueness results of Allard--Almgren \cite{AA81} and Simon \cite{Sim83} for minimal surfaces. For the mean curvature flow, the uniqueness of the tangent flow has been proven at compact \cite{Sch14}, cylindrical \cite{CM15}, and asymptotically conical \cite{CS21} singularities. See also singular tangent flows \cite{LSS22, Sto23}, higher codimension \cite{LZ24}, Lagrangian mean curvature flow with multiplicity \cite{Nev07}. Note that the previous works, except for \cite{Nev07}, considered tangent flows of multiplicity one, while we deal with higher multiplicity limits. Neves' result \cite{Nev07} of the $1$-dimensional case gives the uniqueness of tangent flow with multiplicity under the assumption of a priori uniform bound on the Lagrangian angle, while we assume finite entropy instead of bounded angle.

 \bigskip

As stated in the uniqueness theorem \ref{thm:main.unique}, static lines, and shrinking circles are the only ancient finite-entropy flows whose backward limits are of multiplicity one. So, in this paper, we say that an ancient flow is \textit{nontrivial} if it is neither a static line nor a shrinking circle. Also, from now on, we let $\mathcal{M}=\cup_{t\in I}(M_t,t)$ denote a nontrivial smooth curve-shortening flow embedded in $\mathbb{R}^2$ for convenience, unless otherwise stated.

\bigskip
Now, we provide the graphical radius theorem, which is a quantitative version of the main theorem \ref{thm:main.unique}. Note that, due to the higher multiplicity issue, we can employ neither pseudo-locality \cite{INS19} nor White's local regularity \cite{Whi05}, which played crucial roles for extension and improvement arguments for the graphical radius estimates in \cite{CM15} and \cite{CS21}, respectively.

\begin{theorem}[graphical radius and convergence rate]\label{thm:main.conv.rate}
Let $\mathcal{M}$ be an ancient flow whose tangent flow at infinity is the line $\{x_2=0\}$ with multiplicity $m\geq 2$. Then, given $\varepsilon>0$, there exist some constants $\tau_0<0,\delta>0$ and functions 
\begin{equation}\label{eq:main.graph.radius}
  u^1,\cdots, u^m \in C^\infty(\cup_{\tau\leq \tau_0}(-e^{-\delta\tau},e^{-\delta\tau})\times \{\tau\})  
\end{equation}
such that $u^i(y,\tau)>u^{i+1}(y,\tau)$, $(y,u^i(y,\tau))\in \overline{M}_\tau:=e^{\frac{\tau}{2}}M_{-e^{-\tau}}$, and $|u^i_y(y,\tau)| \leq \varepsilon$ hold for all $i\geq 1$, $\tau \leq \tau_0$, and $y \in (-e^{-\delta\tau},e^{-\delta\tau})$. 

Moreover, given $R>0$, there exists some $\tau_R<0$ such that
\begin{equation}\label{eq:main.fast.decay}
   \max_{1\leq i\leq m} \|u^i(\cdot,\tau)\|_{C^2(B_R(0))}\leq e^{4\delta \tau}
\end{equation}
holds for $\tau \leq \tau_R$ and $i \in \{1,\cdots,m\}$.
\end{theorem}

\bigskip

To prove the graphical radius theorem \ref{thm:main.conv.rate}, we will count the exact numbers of several geometrically important types of points, such as knuckles, tips, vertices, and inflection points. See Section \ref{sec:terms} for terminology.

\begin{theorem}\label{thm:main.geometry}
Let $\mathcal{M}$ be an ancient non-compact \rm{[}\em resp. compact\,\rm{]}\em \,  flow with $\text{Ent}(\mathcal{M})=m\in \mathbb{N}$. Let $\mathcal{S}^{\text{tip}}_t$, $\mathcal{S}^{\text{knuc}}_t$, $\mathcal{S}^{\text{shp}}_t$, $\mathcal{S}^{\text{flat}}_t$, and $\mathcal{S}^{\text{infl}}_t$ denote the sets of tips, knuckles, sharp vertices, flat vertices, and inflection points of $M_t$, respectively. Then, there is negative enough time $T<0$ such that 
\begin{align}
    &     |\mathcal{S}^{\text{tip}}_t|=|\mathcal{S}^{\text{shp}}_t|=m-1\; [ resp.\, m]    , && |\mathcal{S}^{\text{flat}}_t|+|\mathcal{S}^{\text{infl}}_t|=m-2\; [ resp.\, m],
\end{align}
and  $|\mathcal{S}^{\text{knuc}}_t|=m$ hold for all $t\leq T$. 
\end{theorem}

\bigskip

Also, it is crucial to investigate the asymptotic behavior of the flow in sharp vertex regions, since the flow is close to straight lines in other regions. For example, we let $v(t)$ denote a continuous family of sharp vertices of $M_t$, and then we can say that $\mathbf{v} = \bigcup_{t \in I} \big(\bv(t), t\big)$ is a \textit{path of sharp vertices} of $\cM$. In particular, if $(-\infty,t_0)\subset I$, then we call it an ancient path. Then, the following theorem shows how the flow converges to a translating grim reaper in the parabolic neighborhood $P(\bv, (|\kappa(\bv)|\vare)^{-1})$ of $\bv(t):=(v(t),t)$. See (\ref{eq: parabolic ball}) and (\ref{eq: parabolic rescaling}) for the definitions of $P(X,r)$ and parabolic rescaling below.

\begin{theorem}[vertex asymptotics]\label{thm:main.vertex.asymp}
Let $\mathcal{M}$ be an ancient flow with an ancient path $\mathbf{v}$ of sharp vertices.
Then, given any $\varepsilon>0$, there is $T\ll -1$ such that for any $t\leq T$, the connected component of $\cD_{|\kappa(\bv(t))|}(\cM - \bv(t))\cap P(\bo, 1/\vare)$ containing the spacetime origin is $\vare$-close in $C^{\lfloor 1/\vare \rfloor}$ to a unit-speed-translating grim reaper curve whose zero-time-slice has its tip at the origin.
\end{theorem}

\bigskip

This vertex asymptotic theorem \ref{thm:main.vertex.asymp} can be obtained by the weak-flow version (Theorem~\ref{thm: low entropy flow}) of the following classification of low-entropy flows.

\begin{theorem}[low-entropy flow]\label{thm:main.low.entropy}
Let $\mathcal{M}$ be an ancient smooth curve-shortening flow embedded in $\mathbb{R}^2$ with $ \text{Ent}(\mathcal{M}) < 3$. Then, it is a static line, a shrinking circle, a paper clip, or a translating grim reaper.
\end{theorem}

\begin{remark}
    To wit, low-entropy flows are convex.
\end{remark}

\bigskip

This paper is organized as follows. In Section 2, we introduce the weak formulation of the curve-shortening flow and the geometric terminologies that will be used in the rest of the paper.
In particular, the ancient paths of critical points of distance and curvature are defined using Sturm's theory. 
In Section 3, we study the coarse structure of the backward limit of ancient finite-entropy embedded solutions. We show that the entropy for a nontrivial ancient flow must be an integer, which in turn determines the number of sheets, knuckles, and fingers of the very ancient time-slices. 
In Section 4, we focus on the low-entropy Brakke flows and classify the ancient embedded flows with entropy no more than 2, based on the strict convexity of the flows of entropy 2 (Theorem~\ref{thm: convex}).
In Section 5, we show that the sub-flow of fingers and tails has a localized Gaussian density ratio almost bounded by 2 and 1, respectively. Then the classification of weak low-entropy flows allows us to show the asymptotic behavior of the high-curvature region on fingers and the curvature decay estimate on tails.
In Section 6, we prove that the flow looks like a grim reaper around sharp vertices. 
With a better understanding of sharp vertices, we can show several geometric properties of fingers, including the improved lower bound for curvature at the sharp vertices, the maximal number of vertices, and the angle difference between two adjacent knuckles.
In Section 7, we show that the smallness of the graph functions in a fixed compact set propagates out and gives the lower bound of the graphical radius. Finally, in Section 8, we prove the uniqueness of tangent flow at infinity based on the graphical radius lower bound and an Allard--Almgren argument.

\
 \section{Preliminaries}
\subsection{Common notation}
Let $X_0 = (x_0 , t_0 ) \in \R^2\times \R$ and $r > 0$.
Denote the open ball in $\R^2$ of radius $r$ centered at $x_0 \in \R^2$ by 
\begin{align*}
    B_r(x_0) = B(x_0 , r) :=\{y\in \R^2\::\: \lvert y-x_0\rvert < r\}.
\end{align*} 
Denote the parabolic open ball in $\R^2 \times \R$ of radius $r$ centered at $X_0 = (x_0 , t_0)$ by
\begin{align}\label{eq: parabolic ball}
    P(X_0 , r) := B(x_0 ,r) \times (t_0 - r^2, t_0].
\end{align}
Let $\mathcal{D}_{r} (x,t) = (rx, r^2 t)$ denote the parabolic rescaling of $(x,t) \in \R^2\times \R$ by $r$.
Given a curve-shortening flow $\cM = \bigcup_{t\in I} M_t \times \{t\}$, let
\begin{align}\label{eq: parabolic rescaling}
    \cM_{X_0, r} := \mathcal{D}_{1/r} (\cM - X_0) = \bigcup_{t\in I} r^{-1}(M_t - x_0) \times \{r^{-2} (t - t_0)\}
\end{align}
denote the curve-shortening flow obtained by translating $X_0$ to the spacetime origin and parabolically rescaling by $r^{-1}$ whose time-slice at $s$ is $r^{-1}(M_{t_0 + r^2 s} - x_0)$.

\bigskip
\subsection{Weak flows}\label{sec: 2.2}
To analyze singularities or the limits of solutions, we need to define solutions in a non-smooth setting with the compactness theorem.
In geometric measure theory, a very nice class of singular curves in $\R^2$ is described by one-rectifiable Radon measures; see \cite{Sim84, Ilm1994ERP}. 
Recall that a one-rectifiable Radon measure $\mu$ in $\R^2$ is a Radon measure that has a one-dimensional tangent line of positive multiplicity at $\mu$-a.e. points (cf. \cite[11.6]{Sim84}). 
As in the smooth setting, the entropy of $\mu$ is defined as 
\begin{align}\label{eq:Entropy}
    \ent[\mu] = \sup_{x_0\in \R^2, \lambda> 0 } \int \tfrac{1}{\sqrt{4\pi \lambda}}e^{-\frac{\lvert x - x_0 \rvert^2}{4\lambda}}\ d\mu(x).
\end{align}
Note that it is the supremum of the \emph{Gaussian weighted length functional} $F_{x_0, \lambda}$ defined by
\begin{align}
    F_{x_0, \lambda}[\mu] = \frac{1}{\sqrt{4 \pi \lambda}}\int e^{-\frac{\lvert x - x_0 \rvert^2}{4\lambda}}\ d\mu(x)
\end{align}
over all centers $x_0 \in \R^2$ and all scales $\lambda>0$.

As in \cite{Bra84, Ilm1994ERP}, a one-dimensional \emph{Brakke flow} in $\R^2$ is a family of Radon measures $\cM = \{\mu_t\}_{t\in I}$ in $\R^2$ that is one-rectifiable for a.e. $t\in I$ and satisfies (\ref{eq: CSF}) in the weak sense, i.e.,
\begin{align}
    \overline{D}_t \int \varphi\ d\mu_t \leq \int \left(-\varphi \lvert \boldsymbol{\kappa} \rvert^2 + (\D \varphi)^{\perp}\cdot \boldsymbol{\kappa}\right)\ d\mu_t
\end{align}
for all nonnegative test functions $\varphi \in C_{c}^1(\R^2)$. Here, $\overline{D}_t$ denotes the limit superior of difference quotients, $\perp$ denotes the normal projection, which is defined $\mu_t$-almost everywhere, and $\boldsymbol{\kappa}$ denotes the generalized curvature vector of the associated varifold $V_{\mu_t}$, which is defined by the first variation formula and exists $\mu_t$-almost everywhere at almost every $t$. 
The integral on the right-hand side is defined as $-\infty$ whenever it does not make sense.
Brakke flows enjoy a nice existence and the compactness theorem \cite{Ilm1994ERP}.

The monotonicity formula plays an important role in the study of singularities. Let $\cM = \{\mu_t\}_{t\in I}$ be a Brakke flow with finite entropy, where the entropy is defined by
\begin{align*}
    \ent[\cM] = \sup_{t\in I} \ent[\mu_t].
\end{align*}
Following Huisken's calculation \cite{Hu90} with modification in the Brakke setting (cf. \cite[Lemma 7]{Ilm95}), for any spacetime point $X_0 = (x_0, t_0) \in \R^2\times \R$, for any $t< t_0$
\begin{align}
    \overline{D}_t \int \Phi_{X_0}(x,t)\ d\mu_t \leq -\int \left\lvert \boldsymbol{\kappa} + \frac{(x - x_0)^\perp}{2(t_0 - t)}\right\rvert^2 \Phi_{X_0}(x,t)\ d\mu_t\ .
\end{align}
Here, $\Phi_{X_0}(x, t) = \left( 4\pi (t_0 - t) \right)_+^{-\frac12} \exp\big(-\frac{\lvert x - x_0\rvert^2}{4 (t_0 - t)}\big)$ denotes the backward heat kernel where $(\cdot)_+ = \max\{ \cdot, 0\}$.
Note that the integral on the left-hand side of the monotonicity formula is exactly $F_{x_0, (t_0 - t)}[\mu_t]$.
A localized version will be stated and used in Section \ref{sec:5}. 
We define the \emph{Gaussian density ratio} centered at $X_0$ at scale $r>0$ by
\begin{align*}
    \Theta(\cM, X_0, r) = \int \Phi_{X_0}\ d\mu_{t_0-r^2}
\end{align*}
which is non-decreasing in $r > 0$. 
Thus, the limit
\begin{align*}
    \Theta(\cM, X_0) = \lim_{r\to 0^+} \Theta(\cM, X_0, r)
\end{align*}
exists for all $X_0$, and is called the \emph{Gaussian density} at $X_0$. The map 
\begin{align*}
    X \to \Theta(\cM, X)
\end{align*}
is upper-semicontinuous. 
If $\cM$ is ancient, the limit
\begin{align*}
    \Theta(\cM, \infty) := \lim_{r \to \infty} \Theta(\cM, X_0, r)
\end{align*}
exists for all $X_0$ and is equal to $\ent[\cM]$. 
By the rigidity part of the monotonicity formula, if $\Th(\cM,X_0) = \Theta(\cM, \infty)$, then the backward portion $\cM_{X_0}^- := (\cM - X_0)\cap \{X = (x,t): t < 0\}$ is self-similar, namely $\cD_\lambda \cM_{X_0}^- = \cM_{X_0}^-$ for all $\lambda > 0$.

The monotonicity formula allows us to control the Euclidean density ratio $r^{-1} \mu_t(B_r(x_0))$ for any $r>0$. 
Given $X_0$, using Ilmanen's compactness theorem (\cite[Theorem 7.1]{Ilm1994ERP}), blowup sequences $\cM_{X_0, \lambda_i} := \cD_{1/\lambda_i}(\cM - X_0)$ with $\lambda_i \to 0^+$ always have a subsequential limit $\hat\cM$, called \emph{tangent flow} at $X_0$.
Similarly, if $\cM$ is ancient, blowdown sequences $\cM_{X_0, \lambda_i}$ with $\lambda_i \to \infty$ always have a subsequential limit $\check\cM$, called \emph{tangent flow at infinity} as discussed in the Introduction.
Using the monotonicity formula again (cf. \cite[Lemma 8]{Ilm95}), $\hat\cM$ and $\check\cM$ are both backward self-similar.

All Brakke flows $\cM = \{\mu_t\}_{t\in I}$ throughout the present paper have the following additional good properties:
\begin{itemize}
    \item \emph{integral} (cf. \cite{Bra84, Ilm1994ERP}): $\mu_t$ is \emph{integer} one-rectifiable for almost all $t$.
    
    \item \emph{unit-regular} (cf. \cite{Whi05}): Every spacetime point with Gaussian density equal to one is a regular point.

    \item \emph{cyclic} (cf. \cite{Whi09}): For almost all $t$ the associated $\mathbb{Z}_2$ flat chain $[V_{\mu_t}]$ of the associate varifold is cyclic, namely, $\pa [V_{\mu_t}] = 0$.
\end{itemize}

It is worth noting that being integral, unit-regular, and cyclic is preserved under the limit of Brakke flows. 
For the closure of being integral and cyclic, see \cite[Theorem 3.2, 3.3, and 4.2]{Whi09}. The closure of being unit-regular follows from White's regularity theorem \cite{Whi05}.
In particular, Brakke flows starting at any closed embedded curve $M\subset \R^2$ constructed via Ilmanen's elliptic regularization procedure \cite{Ilm1994ERP} are integral, unit-regular, and cyclic.

For any closed set $K$ in $\R^2$, its \emph{level-set flow} $\{K_t\}_{t\geq 0}$ is defined to be the maximal family of closed sets starting at $K$ that satisfies the avoidance principle (cf. \cite[Section 10]{Ilm1994ERP}), whose spacetime track is denoted by $\cK$.
For any complete embedded curve $M\subset \R^2$, $\R^2\setminus M$ has exactly two connected components whose closures are indicated by $K, K'$, respectively. 
When $M\cong \mathbb{S}^1$, $K$ is always chosen to be compact and is called \emph{inside}, in contrast, $K'$ is called \emph{outside}; when $M\cong \R$, $K$ and $K'$ will be determined later and still called inside and outside with respect to a ``finger",  respectively.
Another two weak flows of the CSF starting at $M$ are the \emph{outer flow} $\{\pa K_t\}_{t\geq 0}$ and the \emph{inner flow} $\{\pa K_t'\}_{t\geq 0}$ \cite{Ilm1994ERP, HW18}. 
They can be viewed as \emph{boundary flows} of sets of locally finite perimeters (or Caccioppoli sets) introduced by de Giorgi (cf. \cite[Section 11]{Ilm1994ERP}). 
As long as the CSF $\{M_t\}_{t\in I}$ is smooth, these formulations are identical. 
The advantage of utilizing these flows in this paper is that we can keep track of the orientation more easily with the topological arguments in Section 4.

\bigskip

\subsection{Geometric terminology for planar curves}\label{sec:terms}

Let $\gamma:\mathbb{L}\to \mathbb{R}^2$ be a smooth immersion whose image $M=\gamma(\mathbb{L})$ is a complete embedded curve assigned with the following orientation. Denote by $\bt$ and $\bn$ the unit tangent and unit normal on $\Gamma$, respectively.
If $\mathbb{L} = \mathbb{S}^{1}$, then we choose the {\it counterclockwise} orientation for $\gamma$, namely, it coincides with the natural orientation of $\mathbb{S}^{1}$. Let $J:\R^2\to \R^2$ defined by $(x, y)\mapsto (-y, x)$. We choose the unit normal $\bn$ to be $\bn := J\bt$. 
Equivalently, $\bn$ points inward inside $K$ of $M$. 
The {\it signed curvature} is defined by
\begin{align}
     \kappa =\langle \gamma_{ss}, \bn\rangle = \langle\boldsymbol{\kappa}, \bn \rangle.
\end{align}
If $\mathbb{L} = \R$, we still choose $\bt$ and $\bn = J\bt$ such that $\bn$ points from $K'$ into $K$ as discussed in the last paragraph of Section 2.1.

Now we will define several geometric critical points of $M$ and a curved segment or a ray whose end point is a critical point. By \emph{curved segment}, we mean an image $\gamma([a, b])$; by \emph{curved ray}, we mean an image $\gamma((-\infty, a])$ or $\gamma([b, \infty))$ for some $-\infty < a < b < \infty$.
\begin{definition}[critical points of distance functions]\label{def: crit pt of dist}
Let $x_0\in \R^2$, we define \emph{knuckle}, \emph{tip}, \emph{finger}, \emph{sheet}, and \emph{tail} with respect to $x_0$ as follows: 
\begin{enumerate}
    \item $\bq_0 = \gamma(q_0)$ is a \emph{tip} with respect to $x_0$ if $\lvert \gamma - x_0\rvert$ reaches its local maximum at $q_0$.
    \item $\bp_0 = \gamma(p_0)$ is a \emph{knuckle} with respect to $x_0$ if $\lvert \gamma - x_0\rvert$ reaches its local minimum at $p_0$. 
    
    \item A compact curved segment $\Gamma\subset M$ is a \emph{finger} with respect to $x_0$ if its end points are isolated knuckles with respect to $x_0$ and there is no knuckle with respect to $x_0$ in the interior of $\Gamma$. 
    
    \item A curved ray $\Gamma\subset M$ is a \emph{tail} with respect to $x_0$ if its unique end point is a knuckle with respect to $x_0$ and there is no knuckle with respect to $x_0$ in the interior of $\Gamma$. 
\end{enumerate}
\end{definition}
\

\begin{definition}[critical points of curvature $\lvert \kappa \rvert$]\label{def: vertex}
We define \emph{inflection point}, \emph{vertex}, and $\emph{edge}$ of $M$ as follows:
\begin{enumerate}
    \item An \emph{inflection point} of the curve $\Gamma$ is a zero of curvature.
    
    \item A \emph{vertex} $\bv = \gamma(v)$ of the curve $\Gamma$ is a critical point of curvature, i.e., $\kappa_s(v)=0$. We say that the vertex $p\in \Gamma$ is \emph{sharp} ( resp. \emph{flat} ) if $\lvert \kappa \rvert$ reaches its local maximum (resp. minimum) at $p$.

    \item A curved segment or ray $\Sigma$ of $M$ is an \textit{edge} of $M$ if $\partial\Sigma$ consists of sharp vertices and there is no sharp vertex in the interior of $\Sigma$.
\end{enumerate}
\end{definition}

\bigskip
Furthermore, we need the following definition for Sturm's theory.
\begin{definition}[paths of geometric critical points]
    For a curve-shortening flow $\mathcal{M}$, the spacetime track $\bp = \bigcup_{t\in I}(p(t), t)\in \mathcal{M}$ is called a \emph{path} of tips or knuckles with respect to $x_0$, or vertices if $p(t)$ is continuous and $p(t)$ is a tip or a knuckle with respect to $x_0$, or a vertex of $M_{t}$ for all $t\in I$, respectively.
\end{definition}

\bigskip

Lastly, \emph{graphical radius} and \emph{sheets} are defined in Definition \ref{def:radius.sheet} after we show the rough convergence theorem \ref{thm: rough convergence}. \emph{Sub-flow of a finger} will be defined in Definition~\ref{def: sub-flow}. In addition, the $\varepsilon$-\emph{grim reaper} flows and the \emph{bumpy} curves are defined in Definition \ref{def: esp grim reaper} and \ref{def: bumpy}, respectively. Furthermore, the $\varepsilon$-\textit{trombone time} is defined in Definition \ref{def:trombone}.

\bigskip

\subsection{Sturm's Theory}\label{sec: 2.4}
For a function $u: [A_0, A_1] \to \R$, we will call the set $u^{-1}(0)$ of zeros the \emph{nodal set} of $u$. Any connected component of $[A_0, A_1] \setminus u^{-1}(0)$ is called a \emph{nodal domain}. A point $x_0\in (A_0, A_1)$ is called a \emph{multiple zero} of $u$ if $u(x_0) = u_x (x_0) = 0$.
\begin{proposition}[Sturm's theory \cite{Ang:1988:ZSS}]\label{prop: zeroset}
Let $u: [A_0, A_1]\times [0, T] \to \mathbb{R} $ be a smooth nontrivial solution to the parabolic equation $u_t=au_{xx}+bu_x + cu$, where $a,b,c\in C^{\infty}([A_0, A_1]\times [0,T])$ and $a>0$. Suppose $u$ satisfies Dirichlet, Neumann, the periodic boundary condition, or the non-vanishing boundary condition: $u(A_i,t) \neq 0$ holds for all $t\in [0,T]$ and $i = 1,2$. Let $z(t)$ denote the number of zeros of $u(\cdot, t)$ in $[A_0, A_1]$ with $t\in [0, T]$. Then 
\begin{enumerate}
    \item[$(a)$] for all $t\in [0, T]$, $z(t)$ is finite,
    \item[$(b)$] if $(x_0, t_0)$ is a multiple zero of $u$, i.e., $u$ and $u_x$ both vanish at $(x_0, t_0)$, then for all $t_1< t_0 < t_2$ we have $z(t_1) > z(t_2)$. More precisely, there exists a neighborhood $N = [x_0 - \vare , x_0 + \vare] \times [t_0 - \delta, t_0 + \delta]$ such that 
    \begin{enumerate}
        \item[$(i)$] $u(x_0 \pm \vare, t) \neq 0$ for $\lvert t - t_0 \rvert \leq \delta$,
        \item[$(ii)$] $u(\cdot, t + \delta)$ has at most one zero in the interval $[x_0 - \vare , x_0 + \vare]$,
        \item[$(iii)$] $u(\cdot, t - \delta)$ has at least two zeros in the interval $[x_0 - \vare , x_0 + \vare]$.
    \end{enumerate}   
\end{enumerate} 
\end{proposition}
\

\begin{corollary}\label{cor: almost no multiple zero}
Let $u$ be assumed as in the above proposition. Then the set 
\begin{equation*}
    \{t\in [0, T]: u(\cdot, t) \mbox{ has a multiple zero}\}
\end{equation*} 
is finite.
\end{corollary}

As a consequence of Proposition \ref{prop: zeroset} and the implicit function theorem, the zero set $Z = \{(x,t)\in [A_0, A_1] \times [0, T]: u(x,t) = 0\}$ is a closed set with a \emph{forest structure} $R\sqcup S$, in which $R = \{(x,t): u(x,t)= 0, u_x(x,t) \neq 0\}$ is a union of smooth paths (edges) and $S = \{(x,t): u(x,t)= 0, u_x(x,t) = 0\}$ is the discrete set of multiple zeros (nodes). See Figure \ref{fig: forest} for an illustration. Additionally, each node has at least two branches, and the total number of edges increases in $-t$-direction. Since $u$ has opposite signs in two adjacent spacetime nodal domains, a path of zeros terminates at a node (multiple zero) if and only if the node has an even number of branches. 

\begin{figure}[h]
\centering
\includegraphics[width = 0.5\linewidth]{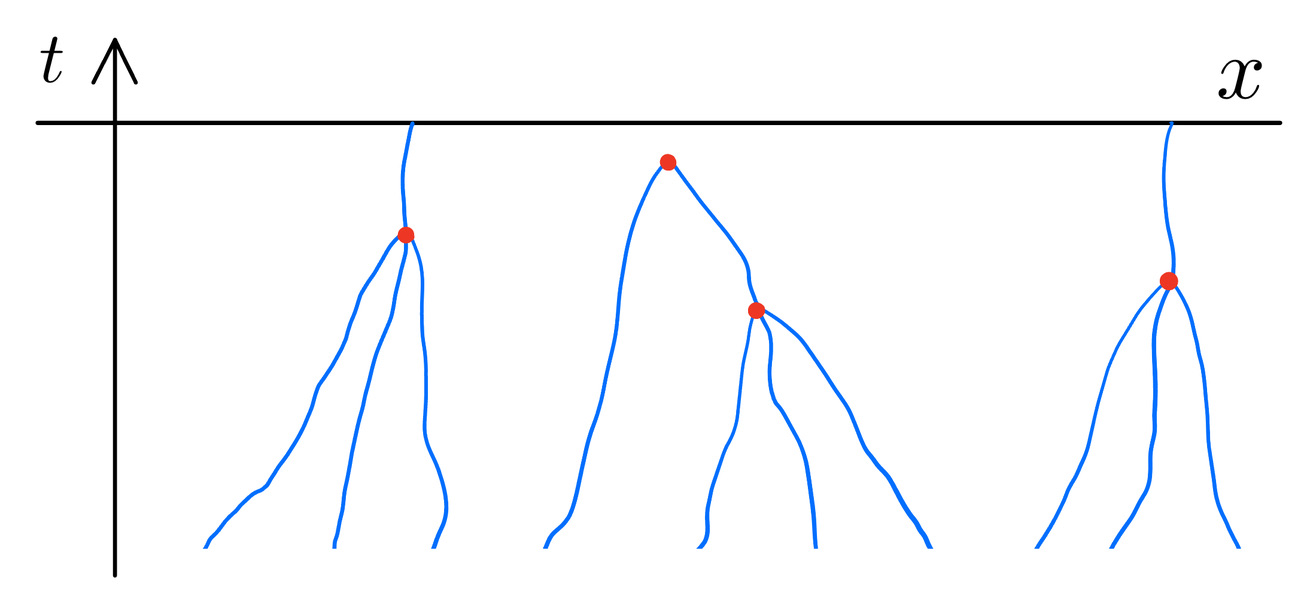}
\caption{The nodal set of a solution $u$ to a parabolic equation in which blue curves are paths of simple zeros and red dots are multiple zeros.}
\label{fig: forest}
\end{figure}

\bigskip

\subsection{Ancient paths of knuckles and tips}\label{subsec: knuckle/tip}
Let $\cM$ be a solution to the CSF parametrized by $\ga$. Given $x_0\in \R^2$, define the squared-distance function by $\varphi^{(x_0)}(s,t) = |\ga(s,t) - x_0|^2 + 2t$. Here, $s$ is the arc-length parameter. When $x_0 = 0$, simply denote $\varphi^{(x_0)}$ by $\varphi$. A direct computation gives
\begin{align}\label{eq: derivatives of distance function}
    \tfrac{\partial}{\partial s}\varphi^{(x_0)} = 2 \langle \ga(s, t) - x_0, \bt(s, t) \rangle, \quad \tfrac{\partial^2}{\partial s^2} \varphi^{(x_0)} = 2 + 2\langle \ga(s,t) - x_0, \boldsymbol{\kappa}(s,t)\rangle.
\end{align}
Thus, $\varphi^{(x_0)}$ satisfies the heat equation:
\begin{align*}
    \tfrac{\partial}{\partial t}\varphi^{(x_0)}=\tfrac{\partial^2}{\partial s^2} \varphi^{(x_0)}.
\end{align*}

\begin{lemma}\label{lem: center of distance}
   $\varphi^{(x_0)}(\cdot, t_0)$ attains a critical point at $s_0$ if and only if $x_0$ is on the normal line of $\ga$ at $\ga(s_0,t_0)$, that is, for some $\al \in \R$
\begin{equation}\label{eq: center}
    x_0 = \ga(s_0,t_0) - \al \bn(s_0,t_0).
\end{equation}
Furthermore, if $\ka(s_0,t_0) = 0$, then $\varphi^{(x_0)}(\cdot, t_0)$ always attains a local minimum at $s_0$; if $\ka(s_0,t_0) \neq 0$ and we substitute $\alpha = \beta / \ka(s_0, t_0)$ in $(\ref{eq: center})$, then
\begin{enumerate}
    \item[$(a)$] when $\beta > 1$, $\varphi^{(x_0)}(\cdot, t_0)$ attains a local maximum at $s_0$, 
    \item[$(b)$] when $\beta < 1$, $\varphi^{(x_0)}(\cdot, t_0)$ attains a local minimum at $s_0$,
    \item[$(c)$] when $\beta = 1$, $x_0$ is the center of the osculating circle tangent to $\ga$ at $\ga(s_0,t_0)$, that is,  $\varphi^{(x_0)}_s(s_0,t_0) = \varphi^{(x_0)}_{ss}(s_0,t_0) = 0$.
\end{enumerate}
\end{lemma}

\begin{proof}
    The lemma follows immediately from the first and second derivative tests. 
\end{proof}

\
 
\begin{proposition}[ancient paths of knuckles/tips]\label{prop: ancient local-min-path/local-max-path}
Suppose that for some $t_0\in I$ such that $p_0$ and $q_0$ are, respectively, a local minimum point, a local maximum point of $\varphi^{(x_0)}(\cdot, t_0)$. Then the following statements hold.
\begin{enumerate}
    \item[$(a)$] There exist continuous ancient paths $\bp = \bigcup_{t\leq t_0} (p(t), t), \bq = \bigcup_{t\leq t_0} (q(t), t)\subset \mathbb{L}\times (-\infty, t_0]$ such that $p(t_0) = p_0$, $q(t_0) = q_0$ and $p(t), q(t)$ are a local minimum point and a local maximum point of $\varphi^{(x_0)}(\cdot, t)$, respectively, for all $t\leq t_0$. Furthermore, for all $t_1 \leq t_2 \leq t_0$,
    \begin{equation}\label{eq: monotone of local extremum}
        \varphi^{(x_0)}(\bp(t_1)) \leq \varphi^{(x_0)}(\bp(t_2)) \quad \text{and} \quad  \varphi^{(x_0)}(\bq(t_1)) \geq  \varphi^{(x_0)}(\bq(t_2)).
    \end{equation}

    \item[$(b)$] The sign of curvature $\ka(\bq(t))$ remains the same for all $t\leq t_0$.

    \item[$(c)$] $(-t)^{-1}|\ga(\bp(t))|^2$ is increasing in $t$.
    
    \item[$(d)$] 
        \begin{equation}\label{eq: limsup of minimum}
            \limsup_{t\to -\infty} \,(-t)^{-1}|\ga(\bp(t))|^2 \leq  2
        \end{equation}
        and
        \begin{equation}\label{eq: liminf of maximum}
            \liminf_{t\to -\infty} \,(-t)^{-1}|\ga(\bq(t))|^2 \geq  2.
        \end{equation}
\end{enumerate} 
\end{proposition}

\begin{proof}
(a) Let $u = \varphi^{(x_0)}_s$. Using the commutator identity, $\tfrac{\pa}{\pa t} \tfrac{\pa}{\pa s} = \tfrac{\pa}{\pa s} \tfrac{\pa}{\pa t} + \kappa^2 \tfrac{\pa}{\pa s}$, $u$ satisfies the parabolic equation
\begin{align*}
    u_t = u_{ss} + \kappa^2 u.
\end{align*}
Suppose that $\cM$ is a flow of closed curves. By Sturm's theory in Section 2.1 with periodic boundary condition, there exists a continuous path $\mathbf{p} = \cup_{t\leq t_0} (\bp(t)) \subset Z$ such that $p(t_0) = p_0$ and $\varphi^{(x_0)}_{ss} = u_s > 0$ in $\mathbf{p}\cap R$. This implies that $\mathbf{p}\cap R$ is a path of local spatial minimum points of $\varphi^{(x_0)}$. 
From the structure of spacetime nodal domains in Section 2.4, $\mathbf{p} \cap S$ is still a path of local spatial minimum points of $\varphi^{(x_0)}$. (\ref{eq: monotone of local extremum}) follows from the minimum principle. The statements regarding the ancient path of local maximum points $q(t), t\le t_0$ hold analogously. 
Now suppose that $\cM$ is a flow of complete non-compact curves. Note that for all $t\leq t_0$, $\displaystyle{\lim_{s\to \pm\infty}\varphi^{(x_0)}(s,t) = \infty}$. According to Sard's theorem, for any $t$ and any $A_0 > 0 $, there exists $A > A_0$ such that $\varphi_s^{(x_0)}(\pm A, t) \neq 0$. Then the Sturm theory with non-vanishing boundary condition applies to $u$ in the spacetime region $[-A, A] \times [t - \vare ,t]$ for some $\vare > 0$ for any $t$ and arbitrarily large $A$.

(b) Observe that for a.e. $t$, $\varphi_{s}^{(x_0)}(\bq(t)) = 0$ and $\varphi_{ss}^{(x_0)}(\bq(t)) < 0$. By (\ref{eq: derivatives of distance function}), this means that $1 < \lvert \ga(\bq(t))  - x_0\rvert \,\lvert \kappa(\bq(t)) \rvert$. Thus, $\kappa(\bq(t)) \neq 0$ for all $t$ and $\kappa(\bq(t))$ cannot change the sign.

(c) For $t_1 < t_2 < t_0$, by (\ref{eq: monotone of local extremum})
\begin{align*}
        (-t_1)^{-1}|\gamma(\bp(t_1))-x_0|^2 &=(-t_1)^{-1} \varphi^{(x_0)}(\bp(t_1)) + 2 
        \leq (-t_1)^{-1}\varphi^{(x_0)}(\bp(t_2)) + 2\\
        &\leq (-t_2)^{-1}\varphi^{(x_0)}(\bp(t_2)) + 2
        = (-t_2)^{-1}|\gamma(\bp(t_2))-x_0|^2.
\end{align*}

(d) Using (\ref{eq: monotone of local extremum}), we find
\begin{equation*}
 (-t)^{-1}|\gamma(\bp(t))-x_0|^2=(-t)^{-1} \varphi^{(x_0)}(\bp(t))+2 \leq (-t)^{-1}\varphi^{(x_0)}(\bp(t_0)) + 2,
\end{equation*}
and 
\begin{equation*}
    (-t)^{-1}|\gamma(\bq(t))-x_0|^2 =(-t)^{-1}\varphi^{(x_0)}(\bq(t)) + 2 \geq (-t)^{-1}\varphi^{(x_0)}(\bq(t_0)) + 2.
\end{equation*}
By taking $\limsup$ and $\liminf$ respectively, we get (\ref{eq: limsup of minimum}) and (\ref{eq: liminf of maximum}).

\end{proof}

\

\subsection{Ancient paths of vertices}
   We recall that the curvature  of  the CSF satisfies
    \begin{align}\label{eq: evolution of curvature}
        \tfrac{\partial}{\partial t}\kappa = \tfrac{\partial^2}{\partial s^2}\kappa + \kappa^3.
    \end{align}
    Equivalently, the curvature $\bar{\kappa}$ of the rescaled solution $\bar{M}_{\tau}$ satisfies 
    \begin{align}\label{eq: evolution of rescaled curvature}
        \tfrac{\partial}{\partial \tau} \bar{\kappa} = \tfrac{\partial^2}{\partial s^2}\bar{
        \kappa} -\tfrac{1}{2}\bar{\kappa}+\bar{\kappa}^3.
    \end{align}

\medskip

\begin{proposition}[ancient sharp-vertex-path]\label{prop: ancient sharp-vertex path}
    Suppose that $\cM$ is a nontrivial flow. Then, given a sharp vertex $v_0$ of $M_{t_0}$ with $t_0\leq T$, there exists an ancient sharp-vertex-path $\bv = \bigcup_{t\leq t_0} (v(t), t)$  such that $v(t_0) = v_0$.
\end{proposition}

\begin{proof}
    Let $u = \kappa_s$. Differentiating (\ref{eq: evolution of curvature}) in $s$ and using 
    $\tfrac{\pa}{\pa t} \tfrac{\pa}{\pa s} = \tfrac{\pa}{\pa s} \tfrac{\pa}{\pa t} + \kappa^2 \tfrac{\pa}{\pa s}$, $u$ satisfies
    \begin{align}\label{eq: evolution of k_s}
        u_t = u_{ss} + 4\kappa^2 u.
    \end{align}
    Using the fact that $\cM$ is nontrivial, we can apply Proposition \ref{prop: zeroset} to the function $u$ as Proposition \ref{prop: ancient local-min-path/local-max-path}
\end{proof}
\

\begin{lemma}[rough curvature lower bound of a sharp vertex]\label{lem: curvature lower bound of vertex}
    For a nontrivial rescaled flow $\bar{M}_{\tau}$, assume that the curvature satisfies $\left|\overline{\kappa}(v(\tau_0), \tau_0)\right|\ge K$ at some sharp vertex $v(\tau_0)$ for some $0< K < 1/\sqrt{2}$. Then, we have
    \begin{align} \label{ineq:compa1}
        |\overline{\kappa}(v(\tau), \tau)| > K \quad \text{for all }\tau< \tau_0
    \end{align} 
    along the vertex path $v(\tau)$. Furthermore, 
    \begin{align} \label{ineq:compa2}
        \liminf_{\tau \to -\infty} |\overline{\kappa}(v(\tau), \tau)| \geq \tfrac{1}{\sqrt{2}}
    \end{align}
\end{lemma}

\begin{proof}
    For simplicity, let $\bar{\chi}(\tau) = \bar{\kappa}(v(\tau), \tau)$ for all $\tau \leq \tau_0$. 
    Assume that $\bar{\chi}(\tau_0) > 0$.
    We claim that if $\bar{\chi}(\tau_0) \geq 1/\sqrt{2}$, then $\bar{\chi}(\tau) \geq 1/\sqrt{2}$ for all $\tau \leq \tau_0$ and then the lemma holds trivially.
    By (\ref{eq: evolution of rescaled curvature}) and the maximality at a sharp vertex, whenever $0< \bar{\chi}(\tau) < 1/\sqrt{2}$, 
    \begin{align}\label{eq: curv l.b.1}
        \tfrac{d}{d \tau} \bar{\chi}(\tau) \leq -\tfrac12 \bar{\chi}(\tau) + \bar{\chi}(\tau)^3 < 0.
    \end{align}
    Suppose the contrary, there exists $\tau_1 < \tau_0$ so that $\bar{\chi}(\tau_1) < 1/\sqrt{2}$. 
    By the mean value theorem, there exists $\tau'\in (\tau_1, \tau_0)$ such that $\bar{\chi}(\tau') < 1/\sqrt{2}$ and $\tfrac{d \bar{\chi}}{d \tau} (\tau') > 0$, contradicting (\ref{eq: curv l.b.1}). 
    
    Suppose $ 0 < K \leq \bar{\chi}(\tau_0) < 1/\sqrt{2}$. 
    Let $\tau_* = \inf\{a \leq \tau_0\::\: \forall \tau \in [a, \tau_0], 0< \bar{\chi}(\tau) < 1/\sqrt{2}\}$.
    From (\ref{eq: curv l.b.1}), $\bar{\chi}(\tau)$ is strictly decreasing in $(\tau_*, \tau_0]$. Thus, $\bar{\chi}(\tau) > K$ on $(\tau_*, \tau_0)$. 
    If $\tau_* > -\infty$, then $\bar{\chi}(\tau_*) = 1/\sqrt{2}$, the situation in $(-\infty, \tau_*]$ reduces to the first paragraph, and we are done. If $\tau_* = -\infty$, then for $\tau \in (-\infty, \tau_0)$, using $K< \bar{\chi}(\tau) < 1/\sqrt{2}$ along with (\ref{eq: curv l.b.1}),
    \begin{align*}
        \tfrac{d}{d \tau} \bar{\chi} \leq -\tfrac12 \bar{\chi} + \bar{\chi}^3 = \bar{\chi} \big(\bar{\chi} - \tfrac{1}{\sqrt{2}}\big) \big(\bar{\chi} + \tfrac{1}{\sqrt{2}}\big) \leq \tfrac{K}{\sqrt{2}} \big(\bar{\chi} - \tfrac{1}{\sqrt{2}}\big) < 0,
    \end{align*}
    and therefore 
    \begin{align*}
        0\leq \big(\tfrac{1}{\sqrt{2}} - \bar{\chi}(\tau) \big) \leq \big(\tfrac{1}{\sqrt{2}} - \bar{\chi}(\tau_0) \big) \exp\big( \tfrac{K}{\sqrt{2}}(\tau - \tau_0)\big).
    \end{align*}
    In particular, $\big(\tfrac{1}{\sqrt{2}} - \bar{\chi}(\tau) \big)$ converges to 0 as $\tau$ tends to $-\infty$. 
    This completes the proof for $\bar{\chi}(\tau_0) > 0$. The assertions for $\bar{\chi}(\tau_0) < 0$ hold analogously since $-\tfrac12 x + x^3$ is odd.
\end{proof}

\bigskip

\section{Rough backward convergence of rescaled flow}
In this section, we prove a rough backward convergence result for nontrivial ancient embedded rescaled flows to a line up to a rotation, whose multiplicity is uniquely determined by its entropy. This will in turn give rise to the sheeting structure and determine the number of fingers and knuckles for ``very old" time-slices.
\subsection{Rough convergence theorem and graphical radius}
Let $\mathbb{L} = \mathbb{R}$ or $\mathbb{S}^{1}$, and let $\overline{M}_{\tau} = \overline{\gamma}(\mathbb{L}, \tau)$ be an ancient embedded solution to the rescaled curve-shortening flow $(\partial_{\tau}\mathbf{x})^{\perp} = \boldsymbol{\kappa} + \frac{\mathbf{x}^{\perp}}{2}$ in $\mathbb{R}^{2}$. For any curve $M\subset \R^2$, we consider its Gaussian weighted length centered at the origin at scale $1$:
\begin{align*}
    F(M) := F_{0, 1}(M) = \frac{1}{\sqrt{4\pi}}\int_{M}e^{-\frac{|\mathbf{x}|^{2}}{4}}\:d\cH^1(\mathbf{x}).
\end{align*}
It follows from Huisken's monotonicity formula that $F(\overline{M}_{\tau})$ is non-increasing in $\tau$.

\begin{lemma}\label{lem:compactness of time-slices}
    Suppose that $\{\overline{M}_{j}\}$ is a sequence of connected, smooth, embedded curves in $\mathbb{R}^{2}$ such that for all $j$, $F(\overline{M}_{j})<\Lambda$ for some $\Lambda>0$ and
    \begin{align}\label{eq:L^2 curvature goes to 0}
        \lim_{j\to\infty}\int_{\overline{M}_{j}}\big|\boldsymbol{\kappa} + \frac{\mathbf{x}^{\perp}}{2}\big|^{2}e^{-\frac{|\mathbf{x}|^{2}}{4}}\:d\cH^1(\mathbf{x}) = 0.
    \end{align}
    Then after passing to a subsequence, $\overline{M}_{j}$ converges in $C^{1, \alpha}_{\rm loc}$, $\alpha\in(0, 1/2)$, to a line $\ell$ passing through the origin or a circle $\mathbb{S}^{1}_{r}$ with radius $r=\sqrt{2}$, possibly with multiplicity. 
\end{lemma}
\begin{proof}
    By assumption (\ref{eq:L^2 curvature goes to 0}), for any $R>0$ and $j$ sufficiently large,
    \begin{align*}
    e^{-\frac{R^{2}}{4}}\int_{\overline{M}_{j}\cap B_{R}}|\boldsymbol{\kappa}|^{2}\:d\mathcal{H}^{1}(\mathbf{x}) &\leq \int_{\overline{M}_{j}\cap B_{R}}|\boldsymbol{\kappa}|^{2}\:e^{-\frac{|\mathbf{x}|^{2}}{4}}\:d\cH^1(\mathbf{x})\\
    & \leq\int_{\overline{M}_{j}\cap B_{R}}\textstyle\left(\left|\boldsymbol{\kappa} + \frac{\mathbf{x}^{\perp}}{2}\right|^{2} + \left|\frac{\mathbf{x}^{\perp}}{2}\right|^{2}\right) e^{-\frac{|\mathbf{x}|^{2}}{4}}\:d\cH^1(\mathbf{x})\\
    &\leq \int_{\overline{M}_{j}\cap B_{R}}\big|\textstyle{\boldsymbol{\kappa}+ \frac{\mathbf{x}^{\perp}}{2}}\big|^{2}\:e^{-\frac{|\mathbf{x}|^{2}}{4}}\:d\cH^1(\mathbf{x}) + 
    \frac{\sqrt{\pi}}{2}R^{2}\Lambda\leq 10R^{2}\Lambda.
\end{align*}
Together with the mass bound $F(\overline{M}_{j})<\Lambda$ and the Sobolev embedding, it follows that $\overline{M}_{j}$ converges in $C^{1, \alpha}_{\rm loc}$, for some $\alpha\in(0, 1/2)$, to a limit curve $\overline{M}_{\infty}$ which satisfies
\begin{align*}
    \boldsymbol{\kappa} + \frac{\mathbf{x}^{\perp}}{2} = 0
\end{align*}
in the distribution sense. By elliptic regularity, $\overline{M}_{\infty}$ is a smooth shrinker.

 Next, we claim that $\overline{M}_{\infty}$ cannot have transverse self-intersection points. Suppose the contrary. Let $p\in\overline{M}_{\infty}$ be a transverse self-intersection point. 
There exist $L_1, L_2\in \mathrm{Gr_1}(\R^2)$ with $L_1\neq L_2$ such that $L_1 \cup L_2 \subset T_p \overline{M}_\infty$. 
Since the curvature of $\overline{M}_\infty$ at $p$ is bounded by $\lvert p\rvert$, 
by $C^{1,\alpha}$ convergence, for all small $\vare > 0$ there exists $\delta > 0$ depending only on $\lvert p \rvert$ such that for all sufficiently large $j$, $\overline{M}_j\cap B_\delta(p)$ contains two components which can be expressed respectively as a graph of $u_j^i$ over $p + L_i$ with $\delta^{-1} \lvert u_j^i\rvert +  \lvert (u_j^i)'\rvert  < \vare $ for $i = 1,2$.
It follows that these two components intersect each other, contradicting the embeddedness of $\overline{M}_j$.

By Abresch--Langer's classification of shrinkers \cite{AL86}, $\overline{M}_{\infty}$ must be a circle with radius $\sqrt{2}$ or a line passing through the origin. This finishes the proof.
\end{proof}

\

\begin{lemma}\label{lem: trivial case}
    Suppose that $\mathcal{M}$ is an ancient smooth embedded flow such that its rescaled flow $\overline{M}_{\tau}$ converges to $\mathbb{S}^{1}_{\sqrt{2}}$ as $\tau\to-\infty$ in $C^{\infty}$-sense. Then $\mathcal{M}$ is a family of shrinking circles.
\end{lemma}
\begin{proof}
    By Grayson's Theorem \cite{G87}, any closed, embedded curve will shrink to a round point in finite time. The result now follows from Huisken's monotonicity formula.
\end{proof}

\ 


The main result in this section is the following backward convergence of nontrivial ancient flows.

\begin{theorem}[rough convergence]\label{thm: rough convergence}
Let $\cM$ be an ancient nontrivial flow with entropy $\Lambda:= \ent[\cM]<\infty$. Then $\Lambda = m\in\mathbb{N}$, and the following property holds for its rescaled flow $\overline{M}_\tau$:

Given $R \gg 1$ and $\varepsilon>0$, there are $T = T(\varepsilon, R)\ll -1$ and a smooth rotation function $S:(-\infty,T]\to SO(2)$ such that $S(\tau) \overline{M}_\tau \cap B_R(0)$ is a disjoint union of $m$ graphs, each of which is $\varepsilon$-close in $C^{\lfloor\frac{1}{\varepsilon}\rfloor}$-sense to the coordinate axis $\ell_{0} = \{(y,0): |y|<R\}$.
\end{theorem}

\begin{proof}
The finite entropy assumption and the Huisken monotonicity formula imply that 
\begin{align*}
    \int_{-\infty}^{-1}\int_{\overline{M}_{\tau}}{\textstyle\left|\boldsymbol{\kappa} + \frac{\mathbf{x}^{\perp}}{2}\right|^{2}}e^{-\frac{|\mathbf{x}|^{2}}{4}}\:d\cH^1(\mathbf{x})\:d\tau = \sum_{k=1}^{\infty}\int_{-(k+1)}^{-k}\int_{\overline{M}_{\tau}}{\textstyle\left|\boldsymbol{\kappa} + \frac{\mathbf{x}^{\perp}}{2}\right|^{2}}e^{-\frac{|\mathbf{x}|^{2}}{4}}\:d\cH^1(\mathbf{x})\:d\tau<\infty.
\end{align*}
For each $k\in\mathbb{N}$, let $\tau_{k}\in[-(k+1), -k]$ be the time such that 
\begin{align*}
    \int_{\overline{M}_{\tau_{k}}}{\textstyle\left|\boldsymbol{\kappa} + \frac{\mathbf{x}^{\perp}}{2}\right|^{2}}e^{-\frac{|\mathbf{x}|^{2}}{4}}\:d\cH^1(\mathbf{x})= \inf_{\tau\in[-(k+1), -k]}\int_{\overline{M}_{\tau}}{\textstyle\left|\boldsymbol{\kappa}+ \frac{\mathbf{x}^{\perp}}{2}\right|^{2}}e^{-\frac{|\mathbf{x}|^{2}}{4}}\:d\cH^1(\mathbf{x}).
\end{align*}
Thus, we must have
\begin{align*}
    \lim_{\tau_{k}\to-\infty}\int_{\overline{M}_{\tau_{k}}}{\textstyle\left|\boldsymbol{\kappa} + \frac{\mathbf{x}^{\perp}}{2}\right|^{2}}e^{-\frac{|\mathbf{x}|^{2}}{4}}\:d\cH^1(\mathbf{x}) = 0.
\end{align*}
By Lemma~\ref{lem:compactness of time-slices}, after passing to a subsequence, $\overline{M}_{\tau_{k}}$ converges in $C^{1, \alpha}_{\rm loc}$ to a limiting curve $\overline{M}_{-\infty}$, which is a line $\ell$ passing through the origin or a circle $\mathbb{S}^{1}_{\sqrt{2}}$, possibly with finite multiplicity. 

We can rule out the shrinking circles by the assumption that the flow is nontrivial.

\

\noindent{\bf Claim~1:}\quad $\overline{M}_{-\infty}$ cannot be a finite cover of $\mathbb{S}^{1}_{\sqrt{2}}$.
\begin{proof}[Proof of Claim~1.] 
Suppose that $\overline{M}_{-\infty}$ is an $m$-cover of $\mathbb{S}^{1}_{\sqrt{2}}$, $m\in\mathbb{N}$. By the $C^{1, \alpha}_{loc}$-convergence, $\overline{M}_{\tau_{j}}$ is a union of $C^{1, \alpha}$ graphs on the circle. Since $\overline{M}_{\tau_{j}}$ is embedded, the graphs are disjoint, which contradicts the connectedness. Thus, we must have $m=1$. By Lemma~\ref{lem: trivial case}, $\mathcal{M}$ is a family of shrinking circles, which contradicts the assumption that the flow is nontrivial.
\end{proof}

Hence, any subsequential limit of $\overline{M}_{\tau_{j}}$ is a line passing through the origin, possibly with multiplicity.

\

\noindent{\bf Claim~2:}\quad The multiplicity $m\in\mathbb{N}$ is independent of the subsequences.
\begin{proof}[Proof of Claim~2.]
    Suppose that along subsequences $\tau_{j}$, $\tau'_{j}$ we have
\begin{align*}
    \overline{M}_{\tau_{j}}\underset{C^{1, \alpha}_{loc}}{\longrightarrow} m\ell\quad\mbox{and}\quad \overline{M}_{\tau'_{j}}\underset{C^{1, \alpha}_{loc}}{\longrightarrow} m'\ell',
\end{align*}
we will show $m = m'$. Given any $\varepsilon>0$, pick a large $k_0$ such that $F(\overline{M}_{\tau_{k_{0}}})>m-\varepsilon$. Then for every $\tau'_{j}$ with $\tau'_{j}>\tau_{k_{0}}$, by monotonicity we have
\begin{align*}
    m'\geq F(\overline{M}_{\tau'_{j}})\geq F(\overline{M}_{\tau_{k_{0}}})> m-\varepsilon.
\end{align*}
Since $\varepsilon>0$ is arbitrary, $m'\geq m$. Similarly, $m\geq m'$, so $m = m'$. 
\end{proof}

We can now characterize the rough shape of the time slices $\overline{M}_{\tau_{j}}$ for $\tau_{j}\ll -1$.

\

\noindent{\bf Claim~3:}\quad For every $\varepsilon>0$ and $R>0$, there exists $J\in\mathbb{N}$ such that whenever $j\geq J$, $\overline{M}_{\tau_{j}}\cap B_{R}(0)$ is a union of $m$ graphs of functions $u^{1}_{j}, \cdots, u^{m}_{j}$ over a line $\ell$ passing through the origin, with $\max_{1\leq i\leq m}\|u^{i}_{j}\|_{C^{1, \alpha}(B_{R})}<\varepsilon$.
\begin{proof}[Proof of Claim~3.]
Suppose the contrary; then we can find a subsequence of $\tau_{j}$ along which the assertion does not hold. By Lemma~\ref{lem:compactness of time-slices} and the Claims above, there is a further subsequence that converges in $C^{1, \alpha}$-sense to a multiplicity $m$ line passing through the origin, a contradiction.
    
\end{proof}

It follows from Claim~3 that for a given $R\gg 1$, for $j$ large enough, the curve $\overline{M}_{\tau_{j}}$ can be decomposed into a disjoint union of graphs in $B_{R}$ and the arcs outside $B_{R}$ that contain the {\it tips}, namely, the local maxima of $|\mathbf{x}|^{2}$ on $\overline{M}_{\tau_{j}}$. The next claim shows that if $\tau_{j}\ll-1$, these tips will stay outside of $B_{R}$ as time goes backwards from $\tau_{j}$ to $\tau_{j-1}$.

\

\noindent{\bf Claim~4:}\quad Given $R>0$, there exists $T = T(R)\ll -1$ such that if $\tau\leq T$, the local maximum of $|\mathbf{x}|^{2}$ on $\overline{M}_{\tau}$ cannot be attained in $B_{R}(0)$.
\begin{proof}[Proof of Claim~4.]
    By Lemma~\ref{lem:compactness of time-slices} and Claim~3, we can find a sequence of rotations $S_{j}\in SO(2)$ such that $S_{j}\overline{M}_{\tau_{j}}$ converges to the axis $\ell_{0} = \{(y, 0)\::\:y\in\mathbb{R}\}$ with multiplicity $m$ in $C^{1, \alpha}_{loc}$-sense. 
    Let $\bq = \bigcup_{t\leq t_0} (q(t), t)$ be an ancient path of tips of $\cM$ defined in Section~\ref{subsec: knuckle/tip}. Converting to rescaled flow, $\overline{\gamma}(q(\tau), \tau)$ is a local maximum of the squared distance $|\mathbf{x}|^{2}$ on $\overline{M}_{\tau}$. Denote $r(\tau) := |\overline{\gamma}(q(\tau), \tau)|^{2}$. Observe that $\overline{\gamma}(q(\tau), \tau) = \overline{\gamma}(q(\tau), \tau)^{\perp}$. From this fact and the above convergence, we know that $r(\tau_{j}) := r(\overline{\gamma}(q(\tau_{j}), \tau_{j}))\to\infty$ as $\tau_{j}\to-\infty$. By the rescaled flow equation, we have
\begin{align*}
    \frac{d}{d\tau}r(\tau) = -2 \lvert \kappa \rvert \sqrt{r} + r \leq r.
\end{align*}
It follows that for any $\tau\in[\tau_{j+1}, \tau_{j}]\subseteq[-(j+2), -j]$,
\begin{align*}
    r(\tau)\geq e^{-2}\cdot r(\tau_{j}).
\end{align*}
Since $r(\tau_{j})\to\infty$, we conclude that $\lim_{\tau\to-\infty}r(\tau) = \infty$.
\end{proof}


For each $k$, define a curve-shortening flow $M^{k}_{\tau}$ by
\begin{align*}
    M^{k}_{t} := (-t)^{1/2}\overline{M}_{\tau_{k}-\log(-t)},\quad t\in [-1, -e^{\tau_k}).
\end{align*}
Then $M^{k}_{-1} = \overline{M}_{\tau_{k}}$. By Claim~3, for $k$ sufficiently large, $M^{k}_{-1}\cap B_{R}$ is a union of graphs of functions $u^{1}_{k}, \cdots, u^{m}_{k}$ on $\ell_{0}$ with small $C^{1, \alpha}$-norms. Given any connected component $\Sigma^{k}_{-1}$ of $M^{k}_{-1}\cap B_{R}$, Claim~4 implies that there is a well-defined curve-shortening flow $\{\Sigma^{k}_{t}\}_{t\in[-1, -e^{\tau_k - \tau_{k-1}}]}$ in $B_{R}$. Hence, the pseudo-locality theorem \cite[Theorem~1.5]{INS19} implies that, for $k$ sufficiently large, $\{S_{k}\overline{M}_{\tau}\cap B_{R}\}_{\tau_{k}\leq\tau\leq\tau_{k-1}}$ is a family of $m$ smooth graphs over $\ell_{0}$.

By Proposition~\ref{prop: ancient local-min-path/local-max-path}, we can choose an ancient path of knuckles $\{p(t) \in B_{\sqrt{-t}R}\}_{t\leq T(R)}$ of $\cM$. Converting to the rescaled flow, define an angle function $\sigma:(-\infty, \overline{\tau}(R)]\to \mathbb{S}^{1} \subset \R^2$ by
\begin{align}
    \sigma(\tau) = \pa_s \overline{\gamma}(p(\tau), \tau)\in \mathbb{S}^{1}.
\end{align}
Then $\sigma(\tau)$ depends smoothly on $\tau$ since $\overline{M}_{\tau}$ does. Identifying $\mathbb{S}^{1}$ with $SO(2)$ we may view $\sigma(\tau)^{-1}$ as a family of rotation matrices $S(\tau)\in SO(2)$.

\ 

\noindent{\bf Claim~5:}\quad For any $\varepsilon>0$, there is $T = T(\varepsilon, R)\gg 1$ such that whenever $\tau<-T$, $S(\tau)\overline{M}_{\tau}\cap B_{R}$ is $\varepsilon$-close to $\ell_{0}$ in the smooth sense.
\begin{proof}[Proof of Claim~5.]
    From the above discussion, we can find a large $K = K(\varepsilon, R)\in\mathbb{N}$ such that whenever $k\geq K$, $S_{k}\overline{M}_{\tau}\cap B_{R}$ is $\frac{\varepsilon}{2}$-close in $C^{1}$-sense to $\ell_{0}$, for $\tau\in[\tau_{k}, \tau_{k-1}]$. Pick $T\geq \max\{-\tau_{K-1}, -\overline{\tau}(R)\}$. Then for any $\tau\leq -T$, we can find $k\geq K$ such that $\tau\in[\tau_{k}, \tau_{k-1}]$. It follows that $d_{S^{1}}(\sigma(\tau), S^{-1}_{k})<\varepsilon/2$, where $d_{S^{1}}(\cdot, \cdot)$ is the arc-length distance on $\mathbb{S}^{1}$, and $S_{k}^{-1}$ is regarded as a point in $\mathbb{S}^{1}$ via the identification $SO(2) \simeq \mathbb{S}^{1}$. This implies that $S(\tau)\overline{M}_{\tau}\cap B_{R} = \sigma(\tau)^{-1}\overline{M}_{\tau}\cap B_{R}$ is $\varepsilon$-close to $\ell_{0}$ in $C^{1}$-sense. The higher order closeness then follows from parabolic regularity.
\end{proof}

Finally, we show that $\ent[\cM] = m$. Observe that for any $(x_0, t_0) \in \R^2 \times \R$, 
\begin{align*}
    S(\tau) \tfrac{1}{\sqrt{t_0 - t}}(M_t - x_0) = S(\tau)\sqrt{\tfrac{-t}{t_0 - t}}\left[(\tfrac{1}{\sqrt{-t}}M_{t}) - \tfrac{x_0}{\sqrt{-t}}\right] \to m\ell_{0} \quad \mbox{as}\; \tau = -\log(-t)\to - \infty,
\end{align*}
in $C^{\infty}_{loc}$-sense. Since the Gaussian weighted length of $\ell_{0}$ is 1, by monotonicity formula we conclude $\ent[\cM] = m$.

\end{proof}

\bigskip

\subsection{Applications}
We assume that $\cM$ is nontrivial and has entropy $m \in\mathbb{N}$.

\begin{definition}[graphical radius and sheets]\label{def:radius.sheet}
Given a smooth ancient embedded nontrivial curve-shortening flow $\cM$ with entropy $m$, let $\overline{M}_\tau$ denote the rescaled flow with respect to the spacetime origin. As a consequence of Theorem~\ref{thm: rough convergence}, there exists a smooth function $\rho(\tau) \geq 0$, called a \emph{graphical radius}, such that
\begin{align*}
    \lim_{\tau \to -\infty} \rho(\tau) = \infty, \quad  \rho'(\tau) \leq 0,
\end{align*}
and such that $S(\tau) \overline{M}_{\tau} \cap B_{2\rho(\tau)}(0)$ is a disjoint union of $m$ graphs each of which is $(2\rho)^{-2}$-close in $C^{\lfloor\rho\rfloor}$-sense to the line segment $(-2\rho(\tau), 2\rho(\tau))\times \{0\}$; namely,

\begin{align*}
    S(\tau) \overline{M}_{\tau} \cap B_{2\rho(\tau)}(0) \subset \bigsqcup_{i = 1}^m \big\{ (x, u_\tau^i (x))\::\: x\in (-2\rho(\tau), 2\rho(\tau)) \big\}
\end{align*}

where 
\begin{align*}
    \| u_\tau^i \|_{C^{\lfloor\rho\rfloor}\big( (-2\rho(\tau), 2\rho(\tau)) \big)} \leq (2\rho(\tau))^{-2}.
\end{align*}
Each graph $\big\{ (x, u_\tau^i (x))\::\: x\in (-2\rho(\tau), 2\rho(\tau)) \big\}$ is called a \emph{sheet} of $\bar{M}_\tau$.
\end{definition}

\

\begin{lemma}\label{lem: uniquenss of knuckle on sheet}
    Let $x_0\in \R^2$. If $t\ll -1$ is such that $\rho(\tau) > 10$ and $|\bar{x}_0| := \tfrac{\lvert x_0\rvert}{\sqrt{-t}} < \tfrac{1}{10}$, then each sheet of $\bar{M}_\tau \cap B_2(0)$ has one and exactly one knuckle with respect to $\bar{x}_0$.
\end{lemma}
\begin{proof}
    The smallness of the $C^0$ norm of sheets in $B_2(0)$ and $\lvert x_0 \rvert < \tfrac{1}{10}$ implies that each sheet has at least one knuckle with respect to $\bar{x}_0$. 
    The smallness of the $C^2$ norm of sheets in $B_2(0)$ and (\ref{eq: derivatives of distance function}) implies that $\varphi_{ss}^{(x_0)}(\cdot, t) > 0$ and hence $\varphi_s^{(x_0)}(\cdot, t)$ has only one zero on each sheet of $\bar{M}_\tau \cap B_2(0)$.
\end{proof}

 \   

\begin{lemma} \label{lem: lim of minimum}
    Let $x_0\in \R^2$ and let $\mathbf{p} = \cup_{t\leq t_0} (\bp(t))$ be an ancient path of knuckles as in Proposition~\ref{prop: ancient local-min-path/local-max-path}. Then, we have 
        \begin{equation}\label{eq: lim of minimum}
            \limsup_{t\to -\infty} \,(-t)^{-1}\lvert\ga(\mathbf{p}(t))\rvert^2 = 0.
        \end{equation}
\end{lemma}

\begin{proof}
    By (\ref{eq: limsup of minimum}) in Proposition \ref{prop: ancient local-min-path/local-max-path}, $\tfrac{1}{\sqrt{-t}} \ga(\mathbf{p}(t)) \in B_2(0)$ for $t\ll -1$. For $t\ll -1$ and $\rho(\tau)\gg 1$, Lemma \ref{lem: uniquenss of knuckle on sheet} implies that each sheet of $\bar{M}_\tau \cap B_2(0)$ has one and exactly one knuckle with respect to the origin. Thus,  the rough convergence theorem \ref{thm: rough convergence} implies the desired result.
\end{proof}

\

\begin{corollary}\label{cor: vertex is far}
    Let $\mathbf{v}$ be an ancient sharp-vertex-path as in Proposition~\ref{prop: ancient sharp-vertex path}. Then there exists $T \ll -1$ such that $\frac{1}{\sqrt{-t}}\lvert \ga\big(\bv(t)\big) \rvert \geq \rho\big(\tau)$ for $t\leq T$. In particular,
    \begin{align*}
        \lim_{t\to -\infty} \frac{1}{\sqrt{-t}} \lvert \ga\big(\bv(t)\big) \rvert = \infty.
    \end{align*}
\end{corollary}
\begin{proof}
    This follows from Lemma~\ref{lem: curvature lower bound of vertex} and the curvature estimate of the rough convergence theorem \ref{thm: rough convergence} in the graphical region.
\end{proof}

\
\begin{proposition}[max number of knuckles]\label{prop: finite number of minimum points} For any $x_0 \in \R^2$, $t\in I$, $M_t$ has at most $m$ knuckles with respect to $x_0$, that is, $\varphi^{(x_0)}$ has at most $m$ local minimum points.
\end{proposition}

\begin{proof}
For some $t_0\in I$ and $x_0\in \R^2$, suppose that there exist $(m+1)$ distinct minimum points of the function $\varphi^{(x_0)}(\cdot, t_0)$, say $p_0^1,\ldots, p_0^{m+1}$.
By Proposition \ref{prop: ancient local-min-path/local-max-path}, there exist ancient paths of knuckles $\bp^i$ of $(p_0^i, t_0)$ for $k = 1,\ldots, m+1$, and such that $\bp^i(t) \neq \bp^j(t)$ for $i\neq j$, $t\leq t_0$. 
By (\ref{eq: limsup of minimum}) in Proposition \ref{prop: ancient local-min-path/local-max-path}, each knuckle $\bar{\gamma}(\bp^{i}(t))$ comes from within $B_2(0)$ for $t \ll -1$. 
Thus, for $t\ll -1$, $\bar{M}_\tau$ has at least $m+1$ knuckles with respect to $\bar{x}_0$ contained in $B_2(0)$. 
It follows from Lemma \ref{lem: uniquenss of knuckle on sheet} and the rough convergence theorem \ref{thm: rough convergence} that for $t\ll -1$ there are exactly $m$ knuckles with respect to $\bar{x}_0$, a contradiction.
\end{proof}

\
 \begin{corollary}[number and location of knuckles]\label{cor: location of knuckles} With the same assumption of Lemma \ref{lem: uniquenss of knuckle on sheet}, $\bar{M}_\tau$ has exactly $m$ knuckles with respect to $x_0$. In addition, every knuckle is contained in a sheet in $\bar{M}_\tau\cap B_2(0)$ $\bar{M}_\tau\cap B_2(0)$.
 \end{corollary}
 \begin{proof}
     This follows from Lemma \ref{lem: uniquenss of knuckle on sheet} and Proposition \ref{prop: finite number of minimum points}.
 \end{proof}

 \

\begin{corollary}[number of fingers]\label{cor: number of fingers}
    Under the assumptions of Proposition \ref{prop: finite number of minimum points}, if $\mathbb{L} = \R$ (resp. $\mathbb{L} = \mathbb{S}^1$), then $M_t$ has at most $m-1$ (resp. $m$) fingers with respect to $x_0$.
\end{corollary}

\begin{proof}
    Every segment bounded by a pair of adjacent knuckles defines a unique finger. The preimage of $k$ knuckles decompose $\R$ into $k-1$ intervals and decompose $\mathbb{S}^1$ into $k$ intervals, respectively. Therefore, the assertion follows.
\end{proof}
\

A critical point of $\varphi^{(x_0)}$ is said to be \emph{multiple} if $\varphi^{(x_0)}_{ss}$ vanishes and \emph{simple} if otherwise.
\begin{corollary}[saturation time]\label{cor: saturation time}
    Under the assumptions of Proposition \ref{prop: finite number of minimum points}, there exists the \emph{saturation time} $t_*(x_0)$ with respect to $x_0$ such that if $t < t_*(x_0)$ then $M_t$ has exactly $m$ knuckles with respect to $x_0$ and every critical point of $\varphi^{(x_0)}(\cdot,t)$ is simple. 

\end{corollary}

\begin{proof}
    Applying Proposition~\ref{prop: zeroset} (b) to $\varphi_s^{(x_0)}$, every multiple zero causes the total number of zeros to increase as time goes backward.
    Along with Proposition \ref{cor: location of knuckles}, the set 
    \begin{align*}
        \{t\in I: \varphi^{(x_0)}(\cdot, t) \mbox{ has a multiple critical point}\}
    \end{align*}
    is finite.
    Take    
    \begin{align*}
    t_*(x_0) = \min \{t\in I: \varphi^{(x_0)}(\cdot, t) \mbox{ has a multiple critical point}\} > -\infty.
    \end{align*}
\end{proof}

 \bigskip

\section{Ancient low-entropy flows}
In this section, we complete the classification (Theorem~\ref{thm: low entropy flow}) of low-entropy ($\ent(\cM) \leq 2$) flows in the Brakke setting in Section \ref{sec: 2.2}. We first study the partial regularity of low-entropy flows. Next, we prove that all low-entropy weak flows are smooth, embedded except for crossing two lines, and convex. Theorem~\ref{thm: low entropy flow} then follows from the previous classification results.

\subsection{Partial regularity of low-entropy flows}
\begin{definition}
    Let $C\subset\mathbb{R}^{2}$ be a stationary cone. Then we call the spacetime track $C\times\mathbb{R}$ (resp. $C\times(-\infty, 0]$) a \emph{static} (resp. \emph{quasi-static}) solution to the Brakke flow.
\end{definition}
\
\begin{lemma}\label{lem: minimal cone}
    Let $\mu$ be an integral one-rectifiable Radon measure in $\R^2$ with $\ent[\mu] \leq 2$. If the associated varifold $V_\mu$ is a cyclic minimal cone, then $\mu = m \cH^1 \lfloor L$ for some straight line $L$ and $m \in \{1, 2\}$ or $\mu = \cH^1 \lfloor (L_1 \cup L_2)$ for some distinct straight lines $L_1, L_2$.
\end{lemma}

\begin{proof}
    Since $V_{\mu}$ is a cyclic cone, it is an even number of straight rays joining at 0. The entropy bound implies that there are two or four such rays. If there are two rays, then since $V_{\mu}$ is stationary, they form a straight line $L$. If there are four rays with directional unit vectors, say $\mathbf{v}_1, \mathbf{v}_2, \mathbf{v}_3, \mathbf{v}_4$ arranged counterclockwise, since $V_{\mu}$ is stationary, we have $\mathbf{v}_1 + \mathbf{v}_2 + \mathbf{v}_3 + \mathbf{v}_4 = 0$. 
    Observe that the angle bisectors $\mathbf{v}_1 + \mathbf{v}_2 = - (\mathbf{v}_3 + \mathbf{v}_4)$ and $\mathbf{v}_2 + \mathbf{v}_3 = - (\mathbf{v}_4 + \mathbf{v}_1)$. A simple calculation gives $\mathbf{v}_1 + \mathbf{v}_3 = \mathbf{v}_2 + \mathbf{v}_4 = 0$. Thus, in this case, $V_{\mu}$ consists of two straight lines or a straight line with multiplicity two, which completes the proof.
\end{proof}

\
\begin{proposition}\label{prop: low entropy selfsimilar flow}
    Suppose $\cM$ is a unit-regular, cyclic, integral, \emph{self-similar} Brakke flow with $\ent[\cM] \leq 2$. Then $\cM$ must be one of the following flows:
    \begin{itemize}
        \item If $\ent[\cM] = 1$, then $\cM$ is a quasi-static line with multiplicity one;
        \item If $\ent[\cM] = \sqrt{\tfrac{2\pi}{e}}\approx 1.52$, then $\cM$ is a shrinking circle;
        \item If $\ent[\cM] = 2$, then $\cM$ is a quasi-static line with multiplicity two or two crossing quasi-static lines.
    \end{itemize}
\end{proposition}

\begin{proof}
    Let $x\in \mathrm{spt}(M_{-1})$ and let $X = (x, -1)$. Any tangent flow $\cM_X$ at $X$ is a quasi-static solution over some stationary cone $C$. Since $\ent[C] = \Th(\cM, X) \leq 2$, by Lemma \ref{lem: minimal cone}, $\ent[C] = \Th(\cM, X)$ is 1 or 2. 
    
    Suppose that there is a point $X = (x, -1)$ such that $\Th(\cM, X) = 2$. The entropy bound implies $\Th(\cM, X, r) = 2$ for $r>0$. By the rigidity part of Huisken's monotonicity formula, $\cM - X$ is a quasi-static solution over $C$. 
    This immediately implies that $\cM$ is either a quasi-static line with multiplicity 2 or two quasi-static crossing lines in which case $x = 0$. 

    Suppose that, for any point $x\in \mathrm{spt}(M_{-1})$, $\Th(\cM, (x, -1)) = 1$. Then, by Lemma \ref{lem: minimal cone} and Allard's regularity theorem \cite{Allard72}, $M_{-1}$ is smooth. Furthermore, $M_{-1}$ is embedded; otherwise, any point of self-intersection has Euclidean density at least 2 and hence $(x, -1)$ has Gaussian density 2, which is a contradiction. By the classification of smooth embedded self-similar solutions by Abresch-Langer \cite{AL86}, $M_{-1}$ is either a line or $\mathbb{S}_{\sqrt{2}}^1(0)$.
\end{proof}
\
\begin{proposition}[partial regularity]\label{thm: partial regularity}
    Let $\cM$ be an ancient unit-regular, cyclic, integral Brakke flow with $\ent[\cM] \leq 2$. Then there is either (1) no singularity, or (2) a round singularity, or (3) there is a point with density 2, in which case $\cM$ is a union of two quasi-static lines passing through a common point.
\end{proposition}

\begin{proof}
    Decompose $\mathrm{spt}(\cM)$ as $\mathrm{reg}(\cM) \sqcup \mathrm{sing}(\cM)$ where $X = (x,t)\in \mathrm{reg}(\cM)$ if and only if $M_t\cap B_r(x)$ is an embedded smooth manifold for some $r>0$. Suppose $X_0 = (x_0, t_0)$ is a singular point of $\cM$. Then $\cM$ being unit-regularity implies $\Th(\cM, X_0) > 1$. 
    By monotonicity formula, $\Th(\cM, X_0 , \infty) > 1$. 
    Any tangent flow at infinity $\check{\cM}$ is unit-regular, cyclic, integral, and self-similar, since all these properties are preserved limits of Brakke flows.
    By the classification in Proposition~\ref{prop: low entropy selfsimilar flow}, $\check{\cM}$ is either a shrinking circle, a quasi-static line with multiplicity 2, or two crossing quasi-static lines. 
    Furthermore, applying the same reasoning to any sub-flow of a connected component of $\mathrm{spt}(\cM)$, the entropy bound implies that $\mathrm{spt}(\cM)$ is connected. The same classification reasoning applies to any tangent flow $\hat{\cM}$ at $X_0$. 

    If $\Th(\cM, X) = 2$ for some $X$, then by the rigidity part of monotonicity formula and Proposition~\ref{prop: low entropy selfsimilar flow}, $\cM$ is a union of two quasi-static lines passing through $X$. 
    Suppose $\Th(\cM, X) < 2$ for all $X\in \cM$, then $\cM$ has only round singularities by the classification of the tangent flow. 
    By Allard's regularity theorem, for any small $\vare > 0$, there exists $r > 0$ such that $\cM \cap \{ B_{2r}(x_0) \times (t_0 -4r^2, t_0 - \vare r^2) \}$ is $\vare$-close to shrinking circle $\sqrt{-2t}\,\mathbb{S}^1$ in $C^\infty$ sense. 
    Furthermore, $X_0\not \in \bar{\mathrm{reg}(\cM) \cap \big(\R^2\times \{t > t_0\}\big)}$ by Brakke's clearing-out lemma \cite[6.3]{Bra84} (cf. also \cite[Proposition 4.23]{Eck2004RTM}). 
    All other singular points have the same asymptotic behavior, so $\mathrm{sing}(\cM)$ is discrete. It follows from the connectedness of $\cM$ and the structure of $\mathrm{sing}(\cM)$ that $M_t$ is smooth, embedded, and connected for all $t < t_0$ and therefore $\mathrm{sing}(\cM) = \{X_0\}$.
\end{proof}

\
\begin{lemma}\label{lem: exclusion of crossing lines}
    Two crossing lines are not a limit of a sequence of embedded smooth CSFs in Brakke flow sense on $P(\mathbf{0}, 2)$.
\end{lemma}

\begin{proof}
    Let $L_1 \neq L_2$ denote the crossing two lines meeting at $0$. Suppose that there is a sequence of embedded smooth CSF $\cM^i$ such that $\cM^i \rightharpoonup (L_1 \cup L_2) \times \R$ in the Brakke flow sense on $P(\mathbf{0}, 2)$. 
    Given any $\vare, \delta > 0$ small with $\vare \ll \delta$, it follows from White's regularity theorem that for sufficiently large $i$, $M_t^i \cap (B_2(0) - B_\delta(0))$ is $\vare$-close to $(L_1 \cup L_2) \cap (B_2(0) - B_\delta(0))$ for all $t\in (-4, 0)$ in $C^2$-sense. This implies that for all $t\in (-4, 0)$, $M_t^i \cap B_2(0)$ has two connected components denoted by $\Ga_t^{i, 1}$ and $\Ga_t^{i, 2}$. Let $\cG^{i, 1}, \cG^{i, 2}$ be the spacetime profiles of such flows, respectively. 
    By the proximity between $\cM^i$ and $L_1\cup L_2$ in the Radon measure sense, for large $i$, $\Th(\cG^{i,j}, X, r) \leq 1 + \vare_W$ for any $X\in P(\mathbf{0},1)$, $0 < r <1$, $j = 1,2$, where $\vare_W > 0$ is the critical constant for White's regularity theorem. Thus, $\Ga_t^{i, 1}$ and $\Ga_t^{i, 2}$ have uniform curvature bounds and converge to $L_1, L_2$, respectively, in the sense of $C^2$ as $i \to \infty$. But this implies that $\Ga_t^{i, 1}$ and $\Ga_t^{i, 2}$ intersect, a contradiction to the embeddedness of $M^i$.
\end{proof}

\
\subsection{Convexity of low-entropy flows}
In the following statements, we assume that $\cM$ has $\ent[\cM]\leq 2$ and is not crossing quasi-static lines. By Proposition~\ref{thm: partial regularity}, we know that $\cM$ is smooth up to the first singular time. 
\begin{definition}\label{def: finger region}
    Let $x_0\in \R^2$, and let $\Gamma(x_0, t)$ denote a finger of $M_t$ with respect to $x_0$ with knuckles $\{p_1,p_2\}$, which could be identical if $M_t$ is closed. Then, the \textit{finger region} $\Omega_\Gamma (x_0, t)$ is the connected bounded open subset of $\R^2$ enclosed by $ \pa \bar{B}_r(x_0) \cup 
 \Ga(x_0, t)$ where $r=\max \{\lvert p_1 - x_0 \rvert , \lvert p_2 - x_0 \rvert \}$. 
 See Figure \ref{fig: finger region} for an illustration.
\end{definition}
Note that, by the definition of the finger, the distance between the tip and $x_0$ is greater than $r$ and hence the finger region is nonempty. 

\begin{figure}[h]
\centering
\includegraphics[width = 0.45\linewidth]{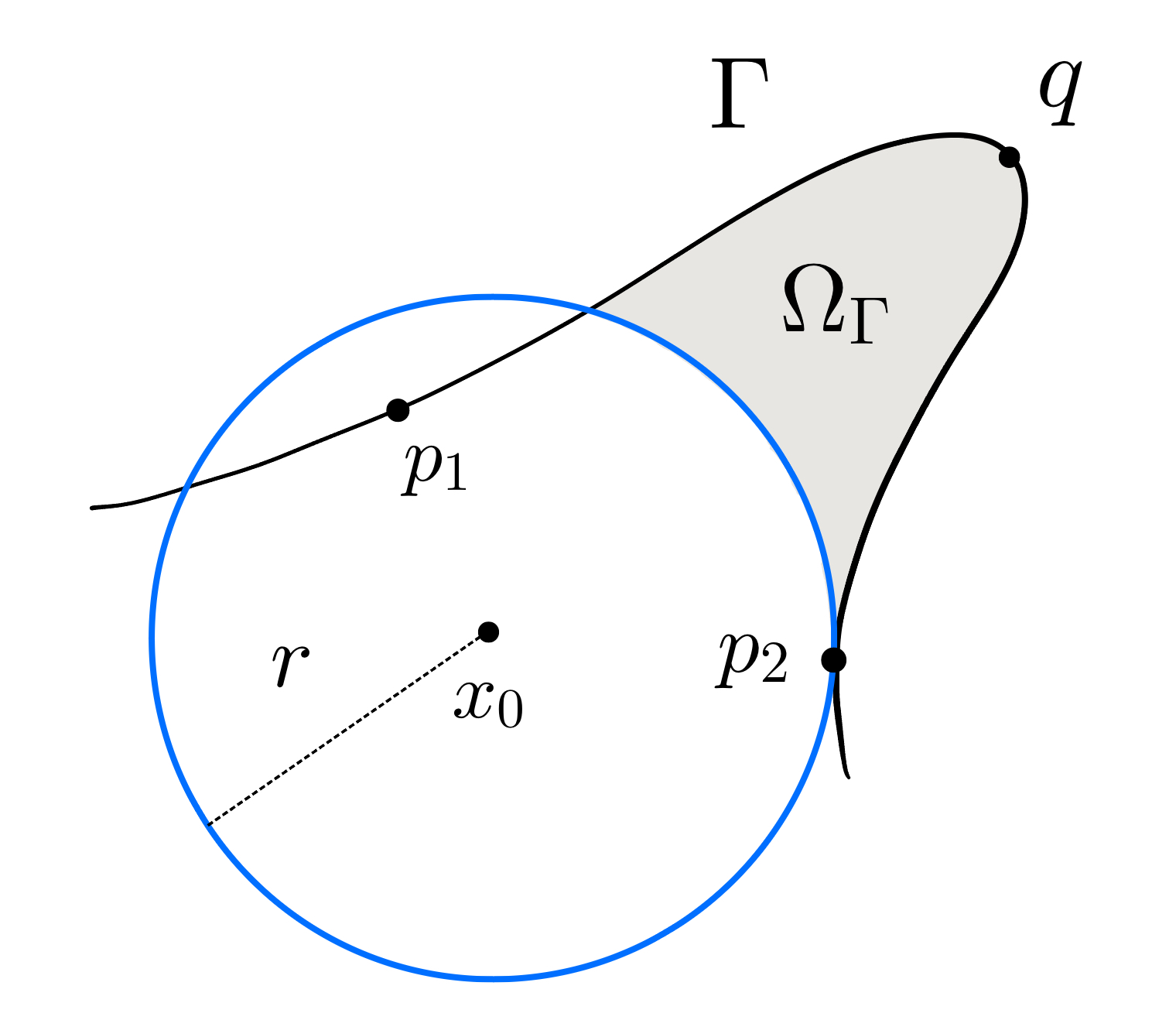}
\caption{Finger $\Ga$ with respect to $x_0$ with knuckles $p_1, p_2$, tip $q$, and (shadowed) finger region $\Om_\Ga$.}
\label{fig: finger region}
\end{figure}

 \bigskip

In Section 2.2, we introduced the level-set flows of inside $K_t$ and outside $K_t'$ of $M_t$. In case $M_t$ is closed, we take $K_t$ to be the compact domain bounded by $M_t$; in case $M_t$ is non-compact, we take $K_t$ such that $K_t$ contains the (unique) connected open region bounded by the two sheets of $M_t$ in the $B(0, \sqrt{-t}\rho(\tau))$ as $t \ll -1$. 
See Figure~\ref{fig:inside} for an illustration.
We choose the orientation of arc-length parametrization $\ga$ such that the unit normal $\bn = J\bt$ points from $K_t'$ into $K_t$ as in Section~\ref{sec:terms}. 
Define the function $\varphi: \R^2 \times \R \times I \to \R$ by $\varphi(x_0; s, t) = \lvert \ga(s,t) - x_0 \rvert^2 + 2 t$.

\begin{figure}
    \centering
    \includegraphics[width=0.6\linewidth]{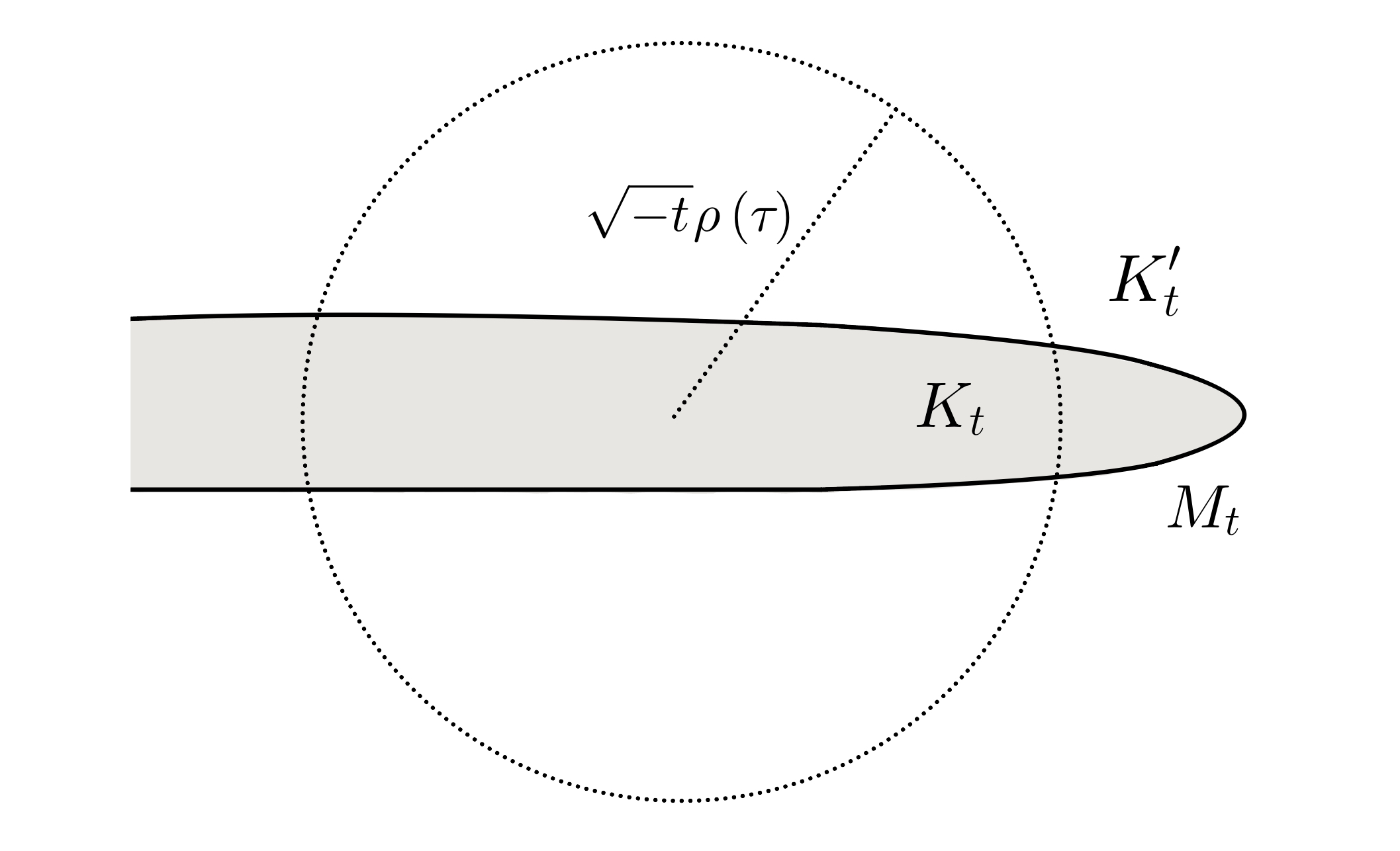}
    \caption{Inside $K_t$ and outside $K_t'$ for a finger of non-compact $M_t$.}
    \label{fig:inside}
\end{figure}

\begin{proposition}\label{prop: ancient finger region}
    Suppose that a finger $\Gamma(x_0, t_0)$ of $M_{t_0}$ with respect to $x_0$ has two distinct knuckles and satisfies $\Omega_{\Gamma(x_0, t_0)}\subset K_{t_0}$ $[\text{resp.\;} \Omega_{\Gamma(x_0, t_0)}\subset K'_{t_0}]$. Then, there is a continuous one-parameter family of curve segments $\{\Gamma(x_0, t)\}_{t\leq t_0}$ such that each $\Gamma(x_0, t)$ is a finger of $M_t$ that satisfies $\Omega_{\Gamma(x_0, t)}\subset K_t$ $[\text{resp.\;} \Omega_{\Gamma(x_0, t)}\subset  K_t']$.
\end{proposition}

\begin{proof}
    It follows from the assumption of distinct knuckles at $t_0$ and Proposition \ref{prop: ancient local-min-path/local-max-path} that there exist ancient paths of knuckles $p_1(x_0; t)\neq p_2(x_0; t)$ and tips $q(x_0; t)$ with respect to $x_0$ for $t \leq t_0$. 
    By Corollary \ref{cor: saturation time}, $p_1(x_0; t), p_2(x_0; t)$ are the only knuckles for $t \leq t_0$. 
    Then the curve segment on $M_t$ bounded by two \emph{adjacent} knuckles $p_1(x_0; t), p_2(x_0; t)$ and containing $q(x_0; t)$ defines a finger of $M_t$ with respect to $x_0$, denoted by $\Ga(x_0, t)$. 
    By the continuity and the non-extinction of finger regions, $\Om_\Ga(x_0, t) \subset K_t$ [resp. $\Omega_{\Gamma(x_0, t)}\subset  K_t'$] for $t< t_0$.
\end{proof}
 
\bigskip

\begin{lemma}\label{lem: inside/outside}
    If $\Om_\Ga(x_0, t) \subset K_t$ $[$resp. $\Om_\Ga(x_0, t) \subset K_t'$$]$, then the tip of $\Ga(x_0, t)$ has curvature $\ka > 0$ $[$resp. $\ka < 0$$]$.
\end{lemma}

\begin{proof}
    Recall that $\varphi_{ss}(x_0, \cdot, t) = 2 \big(1 + \ka \langle \ga - x_0, \bn\rangle\big) < 0$ at the tip of $\Ga(x_0, t)$. This implies that $\ka \langle \ga - x_0, \bn\rangle < 0$ at the tip. 
    Recalling the orientation in Section \ref{sec:terms}, if $\Om_\Ga(x_0, t) \subset K_t$ $[$resp. $\Om_\Ga(x_0, t) \subset K_t'$$]$, then $\langle \ga - x_0, \bn\rangle < 0$ $[$resp. $> 0$$]$ at the tip and hence $\ka > 0$ $[$resp. $\ka < 0$ $]$ at the tip.
\end{proof}

\bigskip

\begin{theorem}\label{thm: convex}
Suppose $\cM$ is an ancient smooth complete embedded nontrivial solution to the CSF with $\ent(\cM) = 2$. Then $M_t$ is strictly convex for all $t\in I$.
\end{theorem}

\begin{proof}
    Suppose $\ka(q_0, t_0) < 0$ for some $(q_0, t_0)$ with respect to the orientation chosen in the setting. 
    By Lemma \ref{lem: center of distance}, we can pick a center $x_0$ such that $\varphi(x_0; s, t_0)$ reaches the local maximum at $s = q_0$. 
    Let $\Ga(x_0, t_0)$ denote the finger of $M_{t_0}$ with respect to $x_0$ with knuckles $p_1(x_0; t_0), p_2(x_0; t_0)$ and tip $q(x_0; t_0) = q_0$. 
    Furthermore, by choosing $x_0$ close to the center of the osculating circle of $\Ga(x_0, t_0)$ at $q_0$, we can assume $p_1(x_0; t_0) \neq p_2(x_0; t_0)$. 
    By Lemma \ref{lem: inside/outside}, $\ka(q_0, t_0) < 0$ implies that the finger region $\Om_{\Ga}(x_0, t_0) \subset K_{t_0}'$. 
    See Figure~\ref{fig:inside out} for an illustration.
    By Proposition \ref{prop: ancient finger region}, there exists an ancient continuous family of fingers $\{\Ga(x_0, t)\}_{t\leq t_0}$ of $M_{t}$ with respect to $x_0$ with knuckles $p_1(x_0; t)\neq p_2(x_0; t)$ and a tip $q(x_0; t)$ such that $\Om_\Ga(x_0, t) \subset K_t'$.

    \begin{figure}
        \centering
        \includegraphics[width=0.6\linewidth]{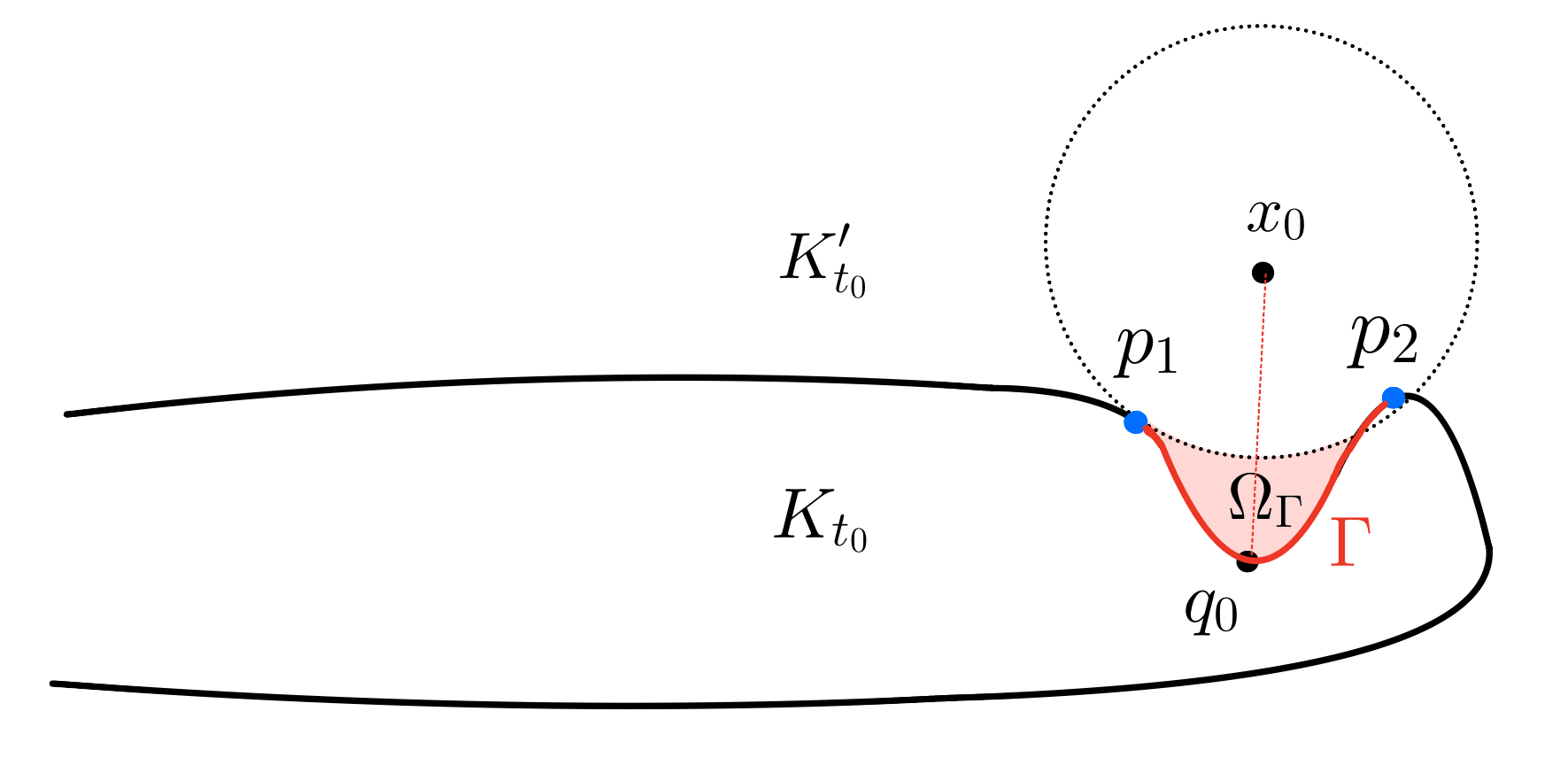}
        \caption{Finger region $\Om_\Gamma$ in the outside $K_{t_0}'$ owing to $\kappa(q_0,t_0)<0$.}
        \label{fig:inside out}
    \end{figure}

    Let $T \ll -1$ such that $\lvert x_0 \rvert \leq \sqrt{-T}$ and $\big((-T)^{-\frac12} M_T\big) \cap B_2(0)$ is $10^{-1}$-close to 2-lines. 
    It follows from Corollary \ref{cor: saturation time} that $p_1(x_0; T), p_2(x_0; T), q(x_0; T)$ are all simple critical points of $\varphi(x_0; \cdot, t)$. 
    Then, by the implicit function theorem, when deforming the reference point continuously along the ray $\{\si x_0: \si \geq 0\}$, the critical points $p_1(\si x_0; T), p_2(\si x_0; T), q(\si x_0; T)$ of $\varphi(\si x_0, \cdot, T)$ depend continuously on $\si$ in some neighborhood $S\subset [0, \infty)$ around $1$. 
    Note that for all $\si \in S\cap [0,1]$, $\ga\big(p_1(\si x_0; T), T\big), \ga\big(p_2(\si x_0; T), T\big)$ stay in distinct sheets of $M_T\cap B(0,2\sqrt{-T})$ and $\ga\big(q(\si x_0; T)\big)$ stay outside of $B(0,2\sqrt{-T})$. 
    In particular, these three critical points stay apart for $\si \in S \cap [0,1]$ and we can reiterate the above argument. 
    Thus, the standard continuity argument implies $S\cap [0, 1] = [0,1]$. 
    The continuous paths of $p_1(\si x_0; T), p_2(\si x_0; T), q(\si x_0; T)$ define an isotopic deformation of finger regions $\{\Om_\Ga(\si x_0, T)\}_{0\leq \si \leq 1}$. 
    Therefore, $\Om_\Ga(x_0, T) \subset K_T'$ implies that $\Om_\Ga(\si x_0, T) \subset K_T'$ for $0\leq \si \leq 1$. 
    But, $\Om_\Ga(0, T) \subset K'_T$ contradicts the choice of the continuous family $\{K_t\}$ in the setting. 
    Finally, with $\ka \geq 0$, the maximum principle implies that $\ka > 0$ for nontrivial solutions to the CSF.
\end{proof}

\bigskip
\subsection{Classfication of low-entropy flows}

\begin{theorem}[classification of low-entropy weak flows]\label{thm: low entropy flow}
    Let $\cM$ be an ancient unit-regular, cyclic, integral Brakke flow with $\ent[\cM] \leq 2$. Then $\cM$ must be one of the following:
    \begin{itemize}
        \item a quasi-static line with multiplicity one or two, or
        \item two crossing quasi-static lines, or
        \item a paper clip $($the Angenent oval$)$, or
        \item a shrinking circle, or
        \item a grim reaper.
    \end{itemize}
\end{theorem}

\begin{proof}
    Considering the entropy condition, every trivial solution is one of a quasi-static line with multiplicity 1 or 2, two crossing quasi-static lines, and a shrinking circle.
    By Proposition~\ref{thm: partial regularity}, all nontrivial solutions are smooth and embedded up to the first singular time. Theorem \ref{thm: convex} implies that all such solutions are convex. Since we have the classification of ancient smooth embedded convex solutions (see \cite[Theorem 1.1]{BLT20}), this finishes the proof.
\end{proof}

\bigskip

Now we can conclude the classification of the ancient smooth embedded curve-shortening flow with entropy $\ent(\cM) < 3$.

\begin{proof}[Proof of Theorem \ref{thm:main.low.entropy}]
    Let $\mathcal{M}$ be an ancient \emph{smooth embedded} curve-shortening flow in $\mathbb{R}^2$ with $ \text{Ent}(\mathcal{M}) < 3$. By the rough convergence theorem~\ref{thm: rough convergence}, if $\cM$ is nontrivial, then $\ent(\cM) \leq 2$. It follows from Theorem~\ref{thm: low entropy flow}, that $\cM$ must be one of a static line, a shrinking circle, a paper clip, and a grim reaper due to the embeddeness of $\mathcal{M}$.
\end{proof}

\bigskip

\section{Asymptotic behavior of fingers and tails}\label{sec:5}
In this section, we will show that, by a suitable localization procedure, the fingers are modeled on entropy-two flows and the tails are modeled on entropy-one flows.

\subsection{Localized Gaussian density ratio of fingers and tails}

In order to study the local behavior of tips and tails outside of graphical regions, we will isolate a finger or a tail from other parts of the flow using the localizing cut-off function that was first considered by Brakke \cite{Bra84} (see also \cite[Proposition 4.17]{Eck2004RTM}). 

Let $\bar{X} = (\bar{x}, \bar{t}) \in \mathbb{R}^{2,1}$, $R > 0$ fixed, for $x\in \R^2$, $t\leq \bar{t}$, define the cutoff function
\begin{align*}
    \psi_{\bar{X}, R}(x, t) = \Big( 1 - \frac{\lvert x - \bar{x} \rvert^2 - 2 (\bar{t} - t) }{R^2} \Big)_+^3,
\end{align*}
where $(\cdot)_+ = \max\{ \cdot, 0\}$. Note that, for $t \leq \bar{t}$, the support of $\psi_{\bar{X}, R}(\cdot, t)$ is
\begin{align*}
    \mathrm{spt}\ \psi_{\bar{X}, R}(\cdot, t) =  B\big( \bar{x}, \sqrt{ R^2 + 2(\bar{t}- t)} \big),
\end{align*}
and note also that $ \psi_{\bar{X}, R}(\cdot, t) \geq 1 $ if and only if $x \in B\big(\bar{x}, \sqrt{2(\bar{t} - t)} \big)$
with maximum
\begin{align*}
    \max_x\ \psi_{\bar{X}, R}(x, t) = \psi_{\bar{X}, R}(\bar{x}, t) = \big( 1 + \tfrac{2(\bar{t} - t)}{R^2} \big)^3.
\end{align*}
As we can see, the cutoff function $\psi$ gets more distorted as $t\ll \bar{t}$. 

Suppose that $0< r < R$ and suppose that $\cM$ is a proper CSF in the intersection of the support of $\psi_{\bar{X}, R}$ and $\R^2 \times [\bar{t}-r^2, \bar{t})$ to avoid any boundary term in weighted monotonicity formula below. For $\si \in (0, r)$, we define the \emph{$R$-localized Gaussian density ratio} of $\cM$ around $\bar{X}$ at scale $\si$ by
\begin{equation*}
    \Th^{R}(\cM, \bar{X}, \si) = \int_{M_{ \bar{t} - \si^2}} \psi_{\bar{X}, R} \Phi_{\bar{X}}\ d\cH^1,
\end{equation*}
where $\Phi_{\bar{X}}(x, t) = \left( 4\pi (\bar{t} - t) \right)_+^{-\frac12} \exp\big(-\frac{\lvert x - \bar{x}\rvert^2}{4 (\bar{t} - t)}\big)$ denotes the backward heat kernel.
Then the weighted monotonicity formula for $\bar{t} - r ^2 < t < \bar{t}$
\begin{align}
    \frac{d}{dt} &\int_{M_t} \psi_{\bar{X}, R} \Phi_{\bar{X}}\ d\cH^1 \label{eq: weighted monotonicity formula}\\
    &= -\int_{M_t}  \left\{ 24\left( 1 - \tfrac{\lvert x - \bar{x} \rvert^2 - 2 ( \bar{t} - t) }{R^2}  \right)_+\left\lvert \tfrac{(x - \bar{x})^\top}{R^2}\right\rvert^2 + \psi_{\bar{X}, R} \left\lvert \boldsymbol{\kappa} + \tfrac{(x - \bar{x})^\perp}{2(\bar{t} - t)}\right\rvert^2 \right\}\Phi_{\bar{X}}(x, t)\ d\cH^1.\notag
\end{align}
implies that $\Th^R(\cG, \bar{X}, \si)$ is non-decreasing in $\si \in (0, r )$.

\ 

\begin{definition}\label{def: sub-flow}
    Let $\cM$ be assumed as in the rough convergence theorem \ref{thm: rough convergence} and let $\cG = \bigcup_{t\leq t_*} \Ga_t^{\cG} \times \{t\}$ denote the \emph{sub-flow} of a finger or a tail of $\cM$ with respect to the origin, where $t_*$ is the saturation time in Corollary \ref{cor: saturation time}. 
\end{definition}

 Now, let $X_0 = (x_0, t_0) \in \cG$ such that 
\begin{align}\label{eq: setting point}
    t_0\ll -1, \quad \rho_0 = \rho(\tau_0) \gg 0, \quad \lvert x_0 \rvert > 2\sqrt{-t_0} \, \rho_0.
\end{align}
Let $\cG_{X_0} := \cG - X_0$ be the spacetime translation of $\cG$ by $X_0$. Set 
\begin{align}\label{eq: setting radii}
    r = \lvert x_0 \rvert / \sqrt{\rho_0} > 2\sqrt{- t_0} \sqrt{\rho_0}, \quad R = \tfrac{\lvert x_0 \rvert}{2} = \tfrac{1}{2} \sqrt{\rho_0} r.
\end{align}

\ 

\begin{lemma}\label{lem: proper flow}
    Let $\cG_{X_0}, r, R$ be as above. For any $\bar{X} = (\bar{x}, \bar{t}) \in P(\mathbf{0}, r )$, $\cG_{X_0}\cap \mathrm{spt}\big(\psi_{\bar{X}, R} \big) \cap \big(\R^2 \times ( -r^2, 0]\big)$ is a proper flow in $\mathrm{spt}\big(\psi_{\bar{X}, R}\big) \cap \big( \R^2 \times (-r^2, 0] \big)$.
\end{lemma}

\begin{proof}
    By construction, $\cG_{X_0}$ is smoothly embedded. It suffices to show that for any $\bar{t} - r ^2 < t \leq \bar{t}$, $\pa (\Ga_t^{\cG} - x_0) \subset \R^2 \setminus \mathrm{spt}\big(\psi_{\bar{X}, R}(\cdot, t)\big)$. Let $\pa \Ga_t^{\cG} = \{p_1(t), p_2(t)\}$. According to our setting, for $i = 1,2$,
    \begin{align*}
        d(0, p_i(t) - x_0 ) = d(x_0, 0) + d(0, p_i(t)) =  \lvert x_0 \rvert + O(\rho_0^{-2}).
    \end{align*}
    Here $d$ denotes the Euclidean distance. Whereas
    \begin{align*}
        \mathrm{spt}\big(\psi_{\bar{X}, R}(\cdot, t)\big) = B \Big(\bar{x}, \sqrt{ R^2 + 2(\bar{t}- t)}\Big) \subset B \Big(0, \sqrt{R^2 + 2r ^2}+ \lvert r  \rvert \Big),
    \end{align*}
    and $\sqrt{R^2 + 2r ^2}+ \lvert r  \rvert \leq \frac12 \lvert x_0 \rvert (1 + O(\rho_0^{-1/2})) < d(0, p_i(t) - x_0 )$ provided $\rho_0$ is large.
\end{proof}

\bigskip
\begin{proposition}[Gaussian density ratio upper bound, unrescaled version]\label{prop: unrescaled GD upper bound}
     Let $\cG_{X_0}, r , R$ be assumed as (\ref{eq: setting radii}). There exists a numerical constant $C> 0$ with the following property. If $\cG$ is the sub-flow of a finger, then for any $\bar{X} \in P(\mathbf{0}, r)$
    \begin{align}
        \Th^R( \cG_{X_0}, \bar{X}, r ) \leq 2 + C\rho_0^{-1};
    \end{align}
    if $\cG$ is the sub-flow of a tail, then for any $\bar{X} \in P(\mathbf{0}, r)$
    \begin{align}
        \Th^R( \cG_{X_0}, \bar{X}, r ) \leq 1 + C\rho_0^{-1}.
    \end{align}
\end{proposition}

\begin{proof}
    By rescaling the variable,
    \begin{align*}
        \Th^R( \cG_{X_0}, \bar{X}, r )
        = \int_{\frac{1}{r}(\Ga_{t_0 + \bar{t} - r ^2} - x_0 - \bar{x})}\big(1 - \tfrac{r^2}{R^2} (\lvert x\rvert^2 - 2)\big)_+^3 \tfrac{1}{\sqrt{4\pi }} e^{-\frac{\lvert x \rvert^2}{4}}d\cH^1.
    \end{align*}
    For brevity, let $t' = t_0 + \bar{t} -r^2$. 
    Note that the support of the last integral is
    \begin{align}\label{eq: support}
        &\tfrac{1}{r}\Big\{(\Ga_{t'} - x_0 - \bar{x}) \cap B(0, \sqrt{R^2 +2r^2})\Big\}\\
        = &\tfrac{\sqrt{-t'}}{r} \Big\{ \tfrac{1}{\sqrt{-t'}} \Ga_{t'} \cap 
        B\Big(\tfrac{x_0 + \bar{x} }{\sqrt{-t'}}, \sqrt{\tfrac{R^2 + 2r^2}{-t'}}\Big) - \tfrac{x_0 + \bar{x} }{\sqrt{-t'}}  \Big\}.\notag
    \end{align}
    Using (\ref{eq: setting radii}) and $\bar{X}\in P(\mathbf{0}, r)$, there exists a numerical constant $C>0$
    \begin{align}
        1 \leq \tfrac{\sqrt{-t'}}{r} &\leq  \sqrt{2} + \tfrac{C}{\rho_0},\label{eq: est rR 1}\\
        \tfrac{\lvert x_0 + \bar{x}\rvert }{\sqrt{-t'}} &\leq \sqrt{\rho_0} + 1,\label{eq: est rR 2}\\
        \sqrt{\tfrac{R^2 + 2r^2}{r^2}} &\leq \tfrac{1}{2}\sqrt{\rho_0}+ C\tfrac{1}{\sqrt{\rho_0}}.\label{eq: est rR 3}
    \end{align}
    If $\rho_0$ is sufficiently large, then from (\ref{eq: est rR 1}), (\ref{eq: est rR 2}), (\ref{eq: est rR 3}),
    \begin{align*}
        \tfrac{1}{\sqrt{-t'}} \Ga_{t'} \cap B \Big( \tfrac{x_0 + \bar{x} }{\sqrt{-t'}}, \sqrt{\tfrac{R^2 + 2r^2}{-t'}}\Big) \subset \tfrac{1}{\sqrt{-t'}} \Ga_{t'} \cap B_{2\rho_0}(0),
    \end{align*}
    which is a disjoint union of one (resp. two) curve(s) with curvature and the minimum of $\lvert x\rvert$ bounded by $(2\rho_0)^{-2}$ if $\cG$ is the sub-flow of a tail (resp. a finger). 
    After rescaling using (\ref{eq: est rR 1}) and applying Lemma \ref{lem: proper flow}, the support (\ref{eq: support}) is a union of curves without boundary in $B(0, \tfrac{1}{r}\sqrt{R^2 + r^2})$ having curvature and the minimum of $\lvert x\rvert$ bounded from above by $\rho_0^{-2}$. 
    Lastly, using Lemma~\ref{lem: est of loc.G.D.R} together with (\ref{eq: est rR 3}) and counting the number of components of the support, we get the desired estimates.
\end{proof}

We end this subsection by stating the rescaled version of localized Gaussian density ratio estimate.
\begin{proposition}[Gaussian density ratio upper bound, rescaled version]\label{prop: rescaled GD upper bound}
     Let $\cG_{X_0}, r , R$ be assumed as (\ref{eq: setting radii}) and let $\lambda > 0$. If $\cG$ is the sub-flow of a finger, then for any $\bar{X}\in P(\mathbf{0}, r/ \la)$,
    \begin{align}\label{eq: rescaled density estimate 1}
        \Th^{R/\la}( \cG_{X_0,\la}, \bar{X}, r/ \la ) \leq 2 + O(\rho_0^{-1});
    \end{align}
    if $\cG$ is a sub-flow of a tail, then for any $\bar{X} \in P(\mathbf{0}, r/ \la)$
    \begin{align}\label{eq: rescaled density estimate 2}
        \Th^{R/ \la}( \cG_{X_0, \la}, \bar{X}, r/ \la ) \leq 1 + O(\rho_0^{-1}).
    \end{align}
\end{proposition}

\

\subsection{Asymptotic behavior of fingers and tails}

Let $\cM$ be an ancient flow and $X_0 = (x_0, t_0)\in \cM$. Define the \emph{regularity scale} at $X_0$ to be 
\begin{align}\label{eq: regularity scale}
     \lambda^\cM(X_0) = \sup\;\{r>0: \cM \text{ is a proper smooth flow in $P(X_0, r)$ with $\lvert \kappa \rvert \leq r^{-1}$}\}.
\end{align}

\begin{theorem}[asymptotic behavior of a large curvature region]\label{thm: asymp of finger}
    Let $\cG \subset \cM$ be the sub-flow of a finger and let $\{X_i = (x_i,t_i)\}_{i\in \mathbb{N}}\subset \cG$ be a sequence such that $t_i\to -\infty$, $\lvert x_i \rvert  > 2\sqrt{-t_i} \rho_i $, and $\liminf \sqrt{-t_i}\lambda_i^{-1} = \al > 0$ where $\rho_i = \rho(\tau_i)$ is the graphical radius of $\mathcal{M}$ and $\lambda_i$ is the regularity scale of $\cG$ at $X_i$. Then $\cG_{X_i,\lambda_i}$ converges subsequentially to a grim reaper with unit speed in the $C^{\infty}_{\text{loc}}$-topology.
\end{theorem}

\begin{proof}  
    For each $i\in \mathbb{N}$, set $r_i=|x_i|/\sqrt{\rho_i}>2\sqrt{-t_i}\sqrt{\rho_i}$ and $R_i = \lvert x_i \rvert/2$ as in \eqref{eq: setting radii}. By assumption,
    \begin{align}\label{eq: range of rescaled cutoff}
        \liminf_i R_i/\la_i \geq \liminf_i r_i/\la_i \geq \liminf_i 2 \sqrt{\rho_i} \sqrt{-t_i}/\la_i \geq 2\al \liminf_i  \sqrt{\rho_i} = \infty.
    \end{align} 
    By Proposition \ref{prop: rescaled GD upper bound} for $r=r_i, \lambda=\lambda_i$ for all $i$, and the weighted monotonicity formula, for any fixed $\sigma_0 > 1$, 
    \begin{align}\label{eq: density ratio bound}
        \limsup_{i} \sup_{\bar{X}\in P(\mathbf{0}, \sigma_0),\: 0< \si < \sigma_0} \Th^{R_i/ \la_i}(\cG_{X_i, \la_i}, \bar{X}, \si) \leq 2.
    \end{align}
    Since the limit of cutoff functions $\lim_i \psi_{0, R_i/\la_i}(X) = \psi_{0, \infty}(X) = 1$ for any $X$, (\ref{eq: density ratio bound}) implies that the Euclidean density ratio of scale less than $\sigma_0$ around any point in $B_{\sigma_0}(0)$ is uniformly bounded. By Ilmanen's compactness theorem \cite[7.1]{Ilm1994ERP}, $\cG_{X_i, \la_i}$ converges subsequentially to an ancient flow $\widetilde{\cG}$ on $P(\mathbf{0}, \sigma_0)$ in the sense of Brakke flow. By diagonal argument and relabeling the index, we assume that $\cG_{X_i, \la_i}$ converges to $\widetilde{\cG}$ on any compact sets as Brakke flows. By the definition of regularity scale, we can also assume that $\widetilde{\cG} \cap P(\mathbf{0}, 1)$ is a proper flow satisfying $\mathbf{0} \in \widetilde{\cG}$ and
    \begin{align}\label{eq: curvature bound in unit ball}
        \sup\:  \{\ \lvert \kappa(X)\rvert \::\: X\in \widetilde{\cG} \cap P(\mathbf{0}, 1)\}= 1.
    \end{align}
    Furthermore, (\ref{eq: density ratio bound}) together with Fatou's lemma implies that the $\ent[\widetilde{\cG}] \leq 2$.

    Since $\lim_i R_i/ \la_i = \infty$, for any $\sigma_0 > 1$, for all large $i$, $\cG_{X_i, \la_i}\cap P(\mathbf{0}, \sigma_0)$ is proper. Passing to the weak limit, $\widetilde{\cG}$ is cyclic, unit-regular, and integral. It is clear that $\widetilde{\Ga}_0$ is not a finite Radon measure. Combining these facts together with (\ref{eq: curvature bound in unit ball}), Theorem \ref{thm: low entropy flow} implies that $\widetilde{\cG}$ is a grim reaper. Finally, by standard parabolic theory, such convergence holds in $C^{\lfloor \sigma_0 \rfloor}$-topology in $P(\mathbf{0}, \sigma_0)$ for any $\sigma_0 > 1$.
\end{proof}

\bigskip

\begin{theorem}[curvature decay of a tail]\label{thm: curv decay of tail}
    Let $\cG \subset \cM$ be a sub-flow of a tail. There exist a time $T = T(\cM) \ll -1$ and a constant $C = C(\cM)$ such that for any $X = (x,t) \in \cG$ with $t\leq T$,
    \begin{align}\label{eq: curv decay of tail 1}
        \lvert \kappa(X) \rvert \leq \tfrac{C \sqrt{\rho}}{\lvert x \rvert},
    \end{align}
    and
    \begin{align}\label{eq: curv decay of tail 2}
        \lvert \kappa(X) \rvert \leq \tfrac{C}{\sqrt{-t} \sqrt{\rho}}.
    \end{align}
\end{theorem}

\begin{proof}
    By the definition of graphical radius, for all sufficiently large $-t$, if $\lvert x \rvert \leq 2 \sqrt{-t}\rho(\tau)$, then $\sqrt{-t}\lvert \kappa(X)\rvert \leq \tfrac14 \rho^{-2}$ and hence $\lvert x \rvert \cdot\lvert \kappa(X) \rvert \leq \rho^{-1}$. 
    Thus, suppose the opposite of (\ref{eq: curv decay of tail 1}); then there exists a sequence $\{X_i = (x_i, t_i)\}\subset \cG$ with $\rho_i = \rho(\tau_i)$ and $r_i=|x_i|/\sqrt{\rho_i}$ satisfying $\lim_i t_i = -\infty$, $\lvert x_i \rvert > 2 \sqrt{-t_i}\rho_i$, and $\lim_i \la_i^{-1}r_i = \lim_i \la_i^{-1}\lvert x_i \rvert /\sqrt{\rho_i} = \infty$. Following an argument similar to the one of Theorem \ref{thm: asymp of finger}, $\cG_{X_i, \la_i}$ converges subsequentially to a cyclic unit-regular integral ancient Brakke flow $\widetilde{\cG}$. Since $\cG$ is a sub-flow of a tail, Proposition \ref{prop: rescaled GD upper bound} implies that $\ent[\widetilde{\cG}] \leq 1$. By Theorem \ref{thm: low entropy flow}, $\widetilde{\cG}$ is a quasi-static line, which contradicts the curvature condition (\ref{eq: curvature bound in unit ball}). 
    
    Next we take a larger $-T$ such that for $t \leq T$, $\tfrac14 \rho^{-2} \leq C/\sqrt{\rho}$ and hence (\ref{eq: curv decay of tail 2}) holds true if $\lvert x \rvert \leq 2 \sqrt{-t}\rho(\tau)$. If $\lvert x \rvert > 2 \sqrt{-t}\rho(\tau)$, then (\ref{eq: curv decay of tail 1}) implies (\ref{eq: curv decay of tail 2}).
\end{proof}

\bigskip

\section{Vertices}

Given any ancient flow (or sub-flow) $\mathcal{M} = \bigcup_{t\in I}M_t \times \{t\}$, $X\in\mathbb{R}^{2}\times I$, and $r>0$, recall that $\mathcal{M}_{X,r} := \mathcal{D}_{1/r}(\mathcal{M} - X)$, where $\mathcal{D}_{1/r}$ is the parabolic rescaling by $1/r$. Recall that an {\it ancient path}  $\mathbf{p} = \bigcup_{t\in I}\mathbf{p}(t)$ on $\mathcal{M}$ is a piecewise smooth curve in $\mathbb{R}^{2}\times I$ such that $\mathbf{p}(t)\in M_{t}$ for all $t\in I$. Here, for brevity, we abuse the notations and write $\mathbf{p}(t) = (p(t), t)$ for a  spacetime curve.

\begin{definition}[$\varepsilon$-grim reaper]\label{def: esp grim reaper}
Given $\varepsilon>0 $ and an ancient flow (or sub-flow) $\mathcal{M}$, we say $X$  lies  at the tip of an $\varepsilon$-grim reaper at scale $r$ if $\mathcal{M}_{X,r}$ is $\varepsilon$-close in $B(0,1/\varepsilon)\times [-1/\varepsilon^2, 0]$ in $C^{[1/\varepsilon]}$-sense to a translating grim reaper with width $\pi$ and tip on the spacetime origin.

Moreover, we say $\mathcal{M}$ is $\varepsilon$-grim reaper-like along an ancient path $\mathbf{p}=\cup_{t\in I}\mathbf{p}(t) \subset\mathcal{M}$ (at curvature scale), if every $\bp(t)$ lies  at the tip of an $\varepsilon$-grim reaper at scale $|\kappa(\bp(t))|\neq 0$.
\end{definition}

\ 

Recall the definition of regularity scale (\ref{eq: regularity scale}). We now show that the ancient curve resembles the translating grim reaper around its sharp vertex paths.

\begin{theorem}[$\Rightarrow$ Theorem~\ref{thm:main.vertex.asymp}]\label{thm: vertex path epsilon grim reaper}
Let $\cG \subset \cM$ be the sub-flow of a finger and $\mathbf{v}$ is an ancient path of 
sharp vertices on $\mathcal{G}$. Then, given any $\varepsilon>0$, there is $T\ll -1$ such that for any $t\leq T$, $\mathbf{v}(t)$ lies at the tip of an $\varepsilon$-grim reaper at scale $\lambda_\bv(t)$, where $\lambda_\bv(t)$ is the regularity scale of $\cG$ at $\mathbf{v}(t)$. 
\end{theorem}

\begin{proof}
    Suppose not. It follows from Lemma \ref{lem: curvature lower bound of vertex} that we can find a sequence of spacetime points along $\mathbf{v}(t)$ that satisfies the hypotheses of Theorem \ref{thm: asymp of finger}. But the result of Theorem \ref{thm: asymp of finger} leads to a contradiction.
\end{proof}

\bigskip

Theorem~\ref{thm: vertex path epsilon grim reaper} implies that for an ancient path of sharp vertices $\bv$ on a sub-flow $\mathcal{G}$ of a finger, there exists a positive \emph{increasing} function $\omega: (-\infty, T] \rightarrow \R_+$ such that $\lim_{t\to -\infty}\omega(t) = 0$ and $\mathbf{v}(t)$ lies at the tip of $\omega(t)$-grim reaper at scale $\lambda_\bv(t)$. 
From the structure of $\omega(t)$-grim reaper, in $M_t\cap B(\mathbf{v}(t), \lambda_\bv(t)/\omega(t) )$ we have
\begin{align}\label{eq: k of GR}
    \kappa = \lambda^{-1}_\bv \cos \theta + \eta^{(0)},
\end{align}
where $\theta(s, t) = \int_{v(t)}^{s} \kappa(\sigma, t)\ d\sigma$ and $\eta^{(0)}$ satisfies the estimate
\begin{align*}
    \sup_{M_t\cap B(v(t), \lambda_\bv(t)/\omega(t) )}\ \lambda_\bv \lvert \eta^{(0)} \rvert + \lambda_\bv^2 \lvert \eta_s^{(0)} \rvert + \lambda_\bv^3 \lvert \eta_{ss}^{(0)} \rvert \leq 2\omega(t).
\end{align*}
Let $d(t)$ denote the Euclidean distance from the tip $\mathbf{q}(t)$. Then whenever $\lambda_\bv /\omega > d \geq 200\pi \lambda_\bv$ with $\omega$ sufficiently small, we have a curvature estimate
\begin{align}\label{eq: curv decay of GR}
    \lambda_\bv^{-1}e^{-\tfrac{99d}{100\lambda_\bv}} -2\omega \lambda_{\bv}^{-1} \leq \lvert \kappa \rvert \leq \lambda_\bv^{-1}e^{-\tfrac{d}{\lambda_\bv}} + 2\omega \lambda_{\bv}^{-1}.
\end{align}
A direct computation using (\ref{eq: k of GR}) shows 
\begin{align}
    \kappa_s &= -\lambda_\bv^{-2} \sin \theta \cos \theta + \tilde{\eta}^{(1)} = -\kappa\sqrt{\lambda_\bv^{-2} - \kappa^2} + \eta^{(1)}, \label{eq: ks of GR}\\
    \kappa_{ss} &= -\lambda_\bv^{-3} [2\cos^3 \theta  - \cos \theta] + \tilde{\eta}^{(2)} = -2\kappa^3 + \lambda_\bv^{-2} \kappa + \eta^{(2)}. \label{eq: kss of GR}
\end{align}
Here, $\tilde{\eta}^{(1)}$, $\eta^{(1)}$, $\tilde{\eta}^{(2)}$, and $\eta^{(2)}$ are error terms involving $\eta^{(0)}$ and its derivative up to second order.  Furthermore, there exists a numerical constant $c_0>0$ such that in $M_t\cap B(v(t), \lambda_\bv(t) /\omega(t) )$

\begin{align}\label{eq: err of k}
    \lvert \eta^{(1)}\rvert \leq c_0\omega \lambda_\bv^{-2}\ ,\  \lvert \eta^{(2)}\rvert \leq c_0\omega \lambda_\bv^{-3}.
\end{align}

\bigskip

\begin{lemma}\label{lem: unique vertex}
Let $\cG \subset \cM$ be a sub-flow of a finger and $\mathbf{v}$ is an ancient path of 
sharp vertices on $\mathcal{G}$. Suppose that $\bv_0 := \mathbf{v}(t_{0})$ lies at the tip of a $\omega_0 := \omega(t_{0})$-grim reaper at scale $\lambda_0 := \lambda_{\mathbf{v}(t_{0})}$ with $t_0 \ll -1$. Then, for each $t\in [t_0 + \tfrac14\lambda_0^2 \log \omega_0, t_0]$, $\Ga^\cG_t$ has a single sharp vertex in $B(v_0,  -\tfrac14\lambda_0 \log \omega_0)$.
\end{lemma}

\begin{proof} 
    In $\cG \cap P(\bv_0, \lambda_0/\sqrt{\omega_0})$, $\lvert \kappa(\bv(t))\lambda_0 - 1\rvert \leq 2\omega_0$ and, from (\ref{eq: kss of GR}) and (\ref{eq: err of k}), $\kappa_{ss}\neq 0$ whenever $\lvert \kappa \rvert \geq \lambda_0^{-1} / 2$.
    Thus, from (\ref{eq: ks of GR}) and (\ref{eq: err of k}), $v(t)$ is the unique vertex in any neighborhood of $v(t)$ with $\lvert \kappa \rvert \geq \lambda_0^{-1} / 2$ for any $t_0 - \lambda_0^2/\omega_0 < t\leq t_0$.
    Next observe that if $\omega_0$ is sufficiently small by taking $t_0 \ll -1$, then from (\ref{eq: ks of GR}) and (\ref{eq: err of k}), in $\cG \cap P(\bv_0, \lambda_0/\sqrt{\omega_0})$ we have $\kappa_{s} \neq 0$ whenever $\lambda_0^{-1} /2 > \lvert \kappa\rvert \geq \lambda_0^{-1} \sqrt{\omega_0}\gg c_0 \lambda_0^{-1} \omega_0$. 
    Combining the results with both conditions on curvature together with (\ref{eq: curv decay of GR}), for any $t_0 - \lambda_0^2/\omega_0 < t\leq t_0$, $v(t)$ is the unique vertex in $B(v(t), - \tfrac12 \lambda_0 \log \omega_0 + 4 \lambda_0 \omega_0)$. 
    Finally, we observe that for $t\in [t_0 + \tfrac14 \lambda_0^2 \log \omega_0, t_0]$, $v(t) \in B(v_0, - \tfrac14 \lambda_0 \log \omega_0 + 2 \lambda_0 \omega_0) \subset B(v(t), -\tfrac12 \lambda_0 \log \omega_0 + 4 \lambda_0 \omega_0)$ since the speed of vertex is roughly $\lambda_0^{-1}$. Therefore, the uniqueness of sharp vertex holds in the desired spacetime neighborhood of $\bv_0$.
\end{proof}

\bigskip

\begin{proposition}[improved curvature lower bound of a sharp vertex]\label{prop: impro.vert.l.b}
    Let $c_0$ be the numerical constant in $(\ref{eq: err of k})$. Suppose the rescaled flow $\bar{\cM}$ is $\omega$-grim reaper-like along an ancient path of the sharp vertices $\{\bar{\mathbf{v}}(\tau)\}_{\tau\le \tau_0}$. For any small $\omega_0 > 0 $, there exists $T\ll -1$ such that for all $\tau \leq T$
    \begin{align}\label{eq: impro.vert.l.b}
        \lvert \bar{\kappa}(\bar{\bv}(\tau))\rvert = \sqrt{-t}\lvert \kappa(\bv(t)) \rvert \geq \tfrac{1}{2\sqrt{c_0\omega_0}}.
    \end{align}
    In particular,
    \begin{align}\label{eq: impro.vert.l.b.0}
        \lim_{\tau\to -\infty} \lvert \bar{\kappa}(\bar{\bv}(\tau)) \rvert =  \infty.
    \end{align}
\end{proposition}

\begin{proof}
    Choose $T$ such that $\omega(\tau) < \omega_0$ for all $\tau<T$. Let $\bar{\chi}(\tau) =  \bar{\kappa}(\bar{\bv}(\tau))$. Substituting (\ref{eq: kss of GR}) into the evolution equation (\ref{eq: evolution of rescaled curvature}) yields 
    \begin{align*}
        \tfrac{\pa}{\pa \tau} \bar{\kappa} = -\bar{\kappa}^3 + \bar{\chi}^2 \bar{\kappa} - \tfrac12 \bar{\kappa}+ \bar{\eta}^{(2)} =: Q
    \end{align*}
    around the tip $\bar{\bv}(\tau)$. In particular, $Q(\bar{\bv}(\tau)) = -\tfrac12 \bar{\chi} + \bar{\eta}^{(2)}$.
    Since (\ref{eq: err of k}) is scaling invariant, we still have $\rvert\bar{\eta}^{(2)}\rvert \leq c_0 \omega_0 \bar{\chi}^3$. Thus, if $\bar{\chi}(\tau) > 0$,
    \begin{align}\label{eq: impro.vert.l.b.1}
        Q(\bar{\bv}(\tau)) \leq -\bar{\chi}[ \tfrac12 - c_0 \omega_0\bar{\chi}^2]; 
    \end{align}
    if $\bar{\chi}(\tau) < 0$,
    \begin{align}\label{eq: impro.vert.l.b.2}
        Q(\bar{\bv}(\tau)) \geq -\bar{\chi}[ \tfrac12 - c_0 \omega_0\bar{\chi}^2].
    \end{align}
    Note that the roots of the polynomial on the right-hand sides of (\ref{eq: impro.vert.l.b.1}) and (\ref{eq: impro.vert.l.b.2}) are $\bar{\chi} = 0, \pm 1/\sqrt{2c_0 \omega_0}$.
    Applying the same reasoning as Lemma~\ref{lem: curvature lower bound of vertex} yields the improved lower bound (\ref{eq: impro.vert.l.b}).
\end{proof}

\bigskip
Following the idea in \cite{CHN13}, we can make the classification of entropy-two self-similar flow quantitative.
\begin{lemma}[quantitative rigidity of entropy-two self-similar flow]\label{lem: quantitative rigidity}
    For any $\vare > 0$, there exists $0 < \delta = \delta(\vare) < \tfrac14$ with the following significance. 
    For any $R > 10$, if $\cG$ is a proper smooth embedded CSF in $P(\mathbf{0}, R)$ satisfying that for some $r \in (0,  \delta R)$
    \begin{align*}
        2 - \delta \leq \Theta^R(\cG, \mathbf{0}, r) \leq \Theta^R(\cG, \mathbf{0}, 2r) \leq 2 + \delta,
    \end{align*}
    then $\cG_{0, r}$ is $\vare$-close in $C^2$ to a line with multiplicity 2 in $B(0, 2/\vare) \times [-4,-1]$.
\end{lemma}

\begin{proof}
    Let $\vare>0$. Suppose the contrary, then there exist a sequence $R^i>10$ and a proper smooth embedded flow $\cG^i$ which is defined in $P(\mathbf{0}, R_i)$ satisfying
    \begin{align*}
        2 - 1/i \leq \Theta^{R_i}(\cG^i, \mathbf{0}, r_i) \leq \Theta^{R_i}(\cG^i, \mathbf{0}, 2r_i) \leq 2 + 1/i,
    \end{align*}
    for $r_i\in (0, R_i/i)$, but $\tilde{\cG}^{i} := \cG^i_{0, r_i}$ is not $\vare$-close in $C^2$ to any line with multiplicity 2 in $B(0, 2/\vare) \times [-4,-1]$.
    From the scaling property of the localized Gaussian density ratio, we have
    \begin{align}\label{eq: quantitative rigidity 1}
        2 - 1/i \leq \Theta^{R_i/r_i}(\tilde{\cG^i}, \mathbf{0}, 1) \leq \Theta^{R_i/r_i}(\tilde{\cG^i}, \mathbf{0}, 2) \leq 2 + 1/i.
    \end{align}
    By Ilmanen's compacteness theorem for Brakke flows \cite{Ilm1994ERP}, the weighted monotonicity (\ref{eq: weighted monotonicity formula}), (\ref{eq: quantitative rigidity 1}), the fact that $\lim_i R_i/r_i = \infty$, and Proposition~\ref{prop: low entropy selfsimilar flow}, up to a subsequence, $\tilde{\cG^i}$ converges to a line with multiplicity 2 or quasi-static crossing two lines in $B(0, 2/\vare)\times [-4, -1]$. 
    The latter case is excluded by Lemma~\ref{lem: exclusion of crossing lines} together with the embeddedness assumption.
    Convergence to a multiplicity-two line in Hausdorff distance implies that $\tilde{\cG}^{i}$ has two connected components (sub-flows) in $B(0, 2/\vare)\times [-4, -1]$ for $i$ large. 
    Let $\tilde{\cG}^{i,1}$, $\tilde{\cG}^{i,2}$ denote these two sub-flows of $\tilde{\cG}^i$ for all large $i$. 
    Note that for $j = 1,2$, $\Theta(\tilde{\cG}^{i,j}, X) \leq 1 + 1/i$ for all $X \in B(0, 2/\vare)\times [-4, -1]$. 
    By White's regularity theorem \cite{Whi05}, up to passing to a further subsequence, both sub-flows $\tilde{\cG}^{i,1}, \tilde{\cG}^{i,2}$ converge in $C^2$ to the same quasi-static line in $B(0, 2/\vare)\times [-4, -1]$, which contradicts our assumption.
\end{proof}

\

\begin{proposition}\label{prop: curv.est.of.finger}
Let $\cG \subset \cM$ be a sub-flow of a finger with a sharp vertex path $\mathbf{v}$. For any $\varepsilon >0$, there exists $T\ll -1$ with the following significance:  
Suppose that $\bv_0 := \mathbf{v}(t_{0})\in \cG$ lies at the tip of a $\omega_0 := \omega(t_{0})$-grim reaper at scale $\lambda_0 := \lambda_{\mathbf{v}(t_{0})}$ for $t_0 < T$. 
Then $\Ga^\cG_t\cap B(v_0, \sqrt{(t_0 - t)}/\vare)$ contains exactly two components that are $\vare$-close in $C^2$ to a line with multiplicity 2 for $ t\in [t_2 , t_1]$ 
where $t_1 = t_0 + \tfrac14\lambda_0^2 \log \omega_0$, $t_2 = t_0 - \tfrac{\lvert v_0\rvert^2}{4\rho_0}$, and $\rho_0 = \rho(\tau_0)$ is the graphical radius at $t_0$. 
In particular, for $ t\in [ \rho_0  t_0 , t_1]$
\begin{align}
    \sup_{\Gamma_{t}^\cG \cap B(v_0, \sqrt{(t_0 - t)}/\vare)} \ \lvert \kappa \rvert \leq \frac{\vare}{\sqrt{t_0 - t}} \Big( \leq \tfrac{2\vare}{\lambda_0 \sqrt{-\log \omega_0}} \Big).
\end{align}
\end{proposition}

\begin{proof}
Now we choose $\sigma = \tfrac12\lambda_0 \sqrt{-\log \omega_0}$ so that $t_1 = t_0 - \sigma^2$ and take $r = \lvert v_0 \rvert /\sqrt{\rho_0}$, $R = \lvert v_0 \rvert /2$ as (\ref{eq: setting radii}). We claim that 
\begin{align}\label{eq: curv.est.of.finger 1}
    2 - \delta < \Theta^R(\cG, \bv_0, \sigma) < \Theta^R(\cG, \bv_0, r) < 2 + \delta
\end{align}
whenever $t_0 \ll -1$, where $\delta = \delta(\vare)$ is determined in Lemma~\ref{lem: quantitative rigidity}.
At $t_1 = t_0 + \tfrac14 \lambda_0^2 \log \omega_0$, by the $\omega_0$-grim-reaper structure around $\bv_0$, $\lvert v(t_1) - v_0 \rvert = - \tfrac14 \lambda_0 \log \omega_0 + O(\omega_0 \lambda_0)$. 
From estimate $(\ref{eq: curv decay of GR})$, $\Ga_{t_1} \cap B(v_0, - \tfrac18 \lambda_0 \log \omega_0)$ has curvature $\lvert \kappa \rvert \leq 2\lambda_0^{-1} \sqrt[8]{\omega_0}$. 
After rescaling by $1/\sigma$, $\tfrac{1}{\sigma}\Ga_{t_1} \cap B(v_0, \tfrac14 \sqrt{-\log \omega_0})$ has curvature
\begin{align}\label{eq: curv.est.of.finger 2}
    \lvert \kappa \rvert \leq \sqrt[8]{\omega_0} \sqrt{-\log \omega_0} 
    = - \tfrac14 \sqrt[8]{\omega_0} \log \omega_0 \cdot \big(\tfrac14 \sqrt{-\log \omega_0}\big)^{-1}
\end{align}
and has
\begin{align}\label{eq: curv.est.of.finger 3}
    \min \ \lvert x \rvert \leq \tfrac{8}{\sqrt{-\log \omega_0}}.
\end{align}
By Proposition~\ref{prop: impro.vert.l.b}, $\lambda_0 \leq 2\sqrt{-c_0 \omega_0 t_0}$. Thus, if $0 < \omega_0 \ll 1$ and $\rho_0 \gg 1$ by taking $t_0 \ll -1$, then $\sigma \leq \sqrt{c_0 \omega_0 \log \omega_0 t_0} \ll 2\sqrt{-\rho_0 t_0} < r$.
Furthermore,
\begin{align}\label{eq: curv.est.of.finger 4}
    \sigma/R < r/R < 2/\sqrt{\rho_0} < \delta,
\end{align}
one hypothesis in Lemma~\ref{lem: quantitative rigidity}, holds true if $\rho_0 \gg 1$.
Using Lemma~\ref{lem: est of loc.G.D.R} with (\ref{eq: curv.est.of.finger 2}), (\ref{eq: curv.est.of.finger 3}), (\ref{eq: curv.est.of.finger 4}), we can compute
\begin{align*}
    \Theta^R(\cG, \bv_0, \sigma) &\geq \int_{\tfrac{1}{\sigma}(\Ga_{t_1} - v_0) \cap B(v_0, \tfrac14 \sqrt{-\log \omega_0})} \Big(1 - \tfrac{\sigma^2}{R^2}(\lvert x \rvert^2 -2)\Big)_+^3 \tfrac{1}{\sqrt{4\pi}}e^{-\frac{\lvert x \rvert^2}{4}} \ d\cH^1(x)\\
    &\geq 2 - \tfrac{C}{\lvert \log \omega_0 \rvert} - \tfrac{C}{\rho_0} - C \sqrt[8]{\omega_0} \lvert \log \omega_0 \rvert - Ce^{-\tfrac{\sqrt{-\log \omega_0}}{16}}\\
    &> 2 -\delta.
\end{align*}
provided $0 <\omega_0 \ll 1$ and $\rho_0 \gg 1$ by taking $t_0 \ll -1$.
On the other hand, by Proposition~\ref{prop: unrescaled GD upper bound}, the upper bound $\Th^R(\cG, X, r) \leq 2 + \delta$ holds if $\rho_0$ sufficiently large.

Choosing a positive integer $N$ such that $2^{N}\sigma \leq 2^{-1}r < 2^{N+1} \sigma \leq r$. 
Then applying Lemma~\ref{lem: quantitative rigidity} recursively on the dyadic scales $\sigma < 2\sigma < 2^2 \sigma < \ldots < 2^N\sigma$ shows that $\Ga_t \cap B(v_0, \sqrt{t_0 - t}/ \vare)$ is $\vare$-close to a line with multiplicity 2 for all $t\in [t_0 - 4^N\sigma^2, t_0 - \sigma^2]$. The curvature estimate is an immediate consequence of this approximation.
\end{proof}

\

\begin{proposition}\label{prop: unique vertex}
    Let $\cG = \bigcup_{t < t_0} \Ga^{\mathcal{G}}_t \times\{t\} \subset \cM$ be a sub-flow of a finger. Given $\omega_0>0$, 
    there exists $T = T(\omega_0) \ll -1$ such that for any $t\leq T$, $\Ga_t^{\cG}$ has at most one sharp vertex that belongs to a vertex path along which $\mathcal{G}$ is $\omega_0$-grim reaper-like and that has curvature $\lvert \kappa \rvert \geq (-4t)^{-\frac{1}{2}}$.
\end{proposition}

\begin{proof}
    Let $0< \vare \ll 1$. We take $T$ as stated in Proposition~\ref{prop: curv.est.of.finger}.
    Suppose that there is some $t_0 \leq T$ and vertices $\mathbf{v}_1(t_0), \mathbf{v}_2(t_0)$ of $\Gamma^{\mathcal{G}}_{t_0}$ that belong to vertex paths $\mathbf{v}_1, \mathbf{v}_2$ along which $\mathcal{G}$ is $\omega_0$-grim reaper-like and has curvature $\lvert \kappa_1 \rvert, \lvert \kappa_2 \rvert \geq (-4t)^{-\frac{1}{2}}$.
    Let $\lambda_0$ denote the regularity scales of $\cG$ at $\mathbf{v}_1(t_0)$, which is less than $2\sqrt{-t_0}$ from curvature assumption. 
    By Proposition~\ref{prop: curv.est.of.finger}, at $t_1 := t_0 + \tfrac14\lambda_0^2 \log \omega_0$, $\Ga^{\mathcal{G}}_{t_1}$ has two knuckles $p, q \in B(\mathbf{v}_1(t_{0}), \sqrt{2(t_0 - t_1)})$ with respect to $v_1(t_{0})$. 
    From $\omega_0$-grim-reaper structure around $\bv_1(t_0)$,
    \begin{align}\label{eq: unique vertex 1}
        \lvert p \rvert, \lvert q \rvert \leq 2\pi \lambda_0 < 4\pi \sqrt{-t_0}.
    \end{align}
    It follows from Lemma \ref{lem: unique vertex} that $v_2(t_1)$ is not in the curve segment of $\Ga_{t_1}$ bounded by $p$ and $q$. 
    We may assume without loss of generality that $\mathbf{v}_2(t_1)$ lies on the curve segment on $\Ga^{\mathcal{G}}_{t_1}$ bounded by $p$ and one end of $\Ga^{\mathcal{G}}_{t_1}$ in the graphical region. 
    Let $\bp(t), \bq(t)$ denote the ancient paths of knuckles $p, q$ with respect to $\mathbf{v}_1(t_0)$ for all $t\leq t_1$.

    Since $\lvert \kappa(\bv_2(t)) \rvert \geq (-4t)^{-\frac{1}{2}}$ for all $t < t_1$ from assumption, $v_2(t)$ needs to stay away from the graphical region $B(0, 2\sqrt{-t}\rho(\tau))$ for all $t < t_1$.
    Whereas from the proof of Proposition~\ref{prop: unrescaled GD upper bound} with $x_0 = v_1(t_0)$, at $t_2 := t_0 - \lvert v_1(t_0) \rvert^2 /(4\rho_0)$, $v_1(t_0)$ lies in $B(0, \sqrt{-t_2}\rho(\tau_2))$, and so does $p(t_2)$ owing to (\ref{eq: unique vertex 1}).
    It follows that $\bv_2(t_2)$ lies in the curved segment of $\Gamma_{t_2}$ bounded by $p(t_2)$ and $q(t_2)$. 
    By the continuity, $\mathbf{p}(t) = \mathbf{v}_2(t)$ for some $t' \in (t_2, t_1)$. 
    However, the $\omega_0$-grim-reaper structure around $\bv_2(t')$ contradicts the $\vare$-two-lines structure around $\bp(t')$ as claimed by Proposition~\ref{prop: curv.est.of.finger}.
\end{proof}

\ 

\begin{proposition}\label{prop: no vertex on tail}
Let $0< \vare < 1$. There exists $T\ll -1$ such that for any $t\leq T$, any tail of $M_t$ does not contain any sharp vertex that belongs to an $\varepsilon$-grim-reaper vertex path. 
\end{proposition}

\begin{proof}
It follows directly from (\ref{eq: curv decay of tail 2}) in Theorem~\ref{thm: curv decay of tail}  and Lemma~\ref{lem: curvature lower bound of vertex}.
\end{proof}

\bigskip

\begin{theorem}[uniqueness of a vertex of a finger]\label{thm: uniqueness of vertex}
    There exists $T = T(\cM) \ll -1$ such that for $t\leq T$, each finger of $M_t$ has a unique sharp vertex, and the curvatures of the sharp vertex and the tip have the same sign. Furthermore, the angle difference\footnote{The angle difference is defined as the integral of curvature by the arc-length measure along the given finger.} between two adjacent knuckles  is $\pm\pi + o(1)$ as $t \to -\infty$ where the sign is the same as the sign of the curvature at the tip.
\end{theorem}

\begin{proof}
    First, we show that each sub-flow of a finger has at least one ancient path of sharp vertices. 
    Let $T$ be the time so that Corollary~\ref{cor: vertex is far} and Proposition~\ref{prop: unique vertex} hold and so that $\rho(\tau) > 10$ for all $t\leq T$. 
    Suppose that a sub-flow of a finger $\cG = \bigcup_{t\leq T} \Ga^{\mathcal{G}}_t\times \{t\}$ did not contain any sharp vertex path. 
    This means that the maximum of $\lvert \kappa \rvert$ on $\Ga_t$ occurs at the knuckles for all $t \leq T$. 
    Thus, the curvature $\lvert \bar{\kappa} \rvert$ of the rescaled flow $\bar{\Ga}^{\mathcal{G}}_\tau$ is uniformly bounded by $\tfrac{1}{100}$ for all $\tau \leq -\log(-T)$. 
    As $\tau \to -\infty$, the rescaled flow $\bar{\Ga}^{\mathcal{G}}_\tau$ merges into $B_1(0)$, a contradiction to rough convergence theorem~\ref{thm: rough convergence}.

    Let $n$ be the number of fingers of $M_t$ as $t \leq t_*$ in Corollary \ref{cor: saturation time}.
    By Corollary~\ref{cor: vertex is far} and Proposition \ref{prop: no vertex on tail}, we know that any sharp vertex path is far away from the graphical region and the tails for $t\leq T$ for some $T\ll -1$. 
    Thus, any sharp vertex path must be trapped in a sub-flow of a finger for $t\leq T$.  
    It follows from the first paragraph and Proposition \ref{prop: unique vertex} that $\cM$ has one and exactly one ancient path of sharp vertices for all $t\leq T$.

    Next, we show that the curvatures at the tip and the sharp vertex of any finger have the same sign for $t\leq T$. 
    Suppose the opposite; by Proposition~\ref{prop: ancient local-min-path/local-max-path} (b) and Lemma~\ref{lem: curvature lower bound of vertex}, the curvatures of the tip and the sharp vertex of the finger remain having opposite signs for all $t < T$. 
    By the uniqueness of the sharp vertex, i.e., maximum of $\lvert\kappa\rvert$, on the finger, the curvature of the tip is bounded between 0 and the curvature of the knuckle, which results in a situation similar to that of the first paragraph.
    As $t\to-\infty$, the tip merges into the graphical region, a contradiction.

    By the rough convergence theorem~\ref{thm: rough convergence}, the angle difference between two adjacent knuckles $\mathbf{p}_1(t), \mathbf{p}_2(t)$ with respect to the origin is $m \pi + O(\rho^{-2})$ for some fixed integer $m$ as $t\ll -1$. 
    Assume that $\cG$ is the sub-flow of fingers bounded by these paths of knuckles, assume that the unique sharp vertex $\bv$ has \emph{positive} curvature with respect to proper orientation such that $\mathbf{p}_1(t) < \mathbf{p}_2(t)$ in the arc-lengh parametrization. 
    
    We claim that $m = 1$.
    Let $0< \vare \ll 1$. By Proposition~\ref{prop: curv.est.of.finger}, if we choose $t_0\ll -1$ and $x_0 = v(t_0)$, then at $ t_1 := t_0 + \tfrac14\lambda_0^2 \log \omega_0$, $\Ga^\cG_{t_1}\cap B(x_0, \sqrt{(t_0 - t_1)}/\vare)$ has two $\vare$-linear components, where $\lambda_0$ is the regularity scale at $\mathbf{v}(t_{0})$ and $\omega_0:= \omega(t_{0})$.
    Thus, there exist exactly two knuckles $r_1, r_2$ of $\Gamma^{\mathcal{G}}_{t_1}$ with respect to $x_0$ with $r_1 < r_2$ in the arc-lengh parametrization.
    Again, their ancient paths are denoted by $\mathbf{r}_1(t), \mathbf{r}_2(t)$.
    Since $r_1 = r(t_1), r_2 = r(t_1)$ are in the $\omega_0$-grim-reaper region around $\bv(t_0)$, we know that 
    \begin{align}\label{eq: angle difference}
        \int_{r_1(t_1)}^{r_2(t_1)} \kappa \ ds = \pi + O(\vare).
    \end{align}
    By Proposition~\ref{prop: curv.est.of.finger}, (\ref{eq: angle difference}) is valid for all $t\in [t_2, t_1]$.
    Note that $r_1(t_2), r_2(t_2)$ are in the graphical region, so 
    \begin{align*}
        \left\lvert \int_{p_1(t_2)}^{p_2(t_2)} \kappa \ ds - \int_{r_1(t_2)}^{r_2(t_2)} \kappa \ ds \right\rvert = O(\rho^{-1}).
    \end{align*}
    This concludes the claim.
\end{proof}

\ 
Recall \emph{finger region} in Definition~\ref{def: finger region}.
\begin{corollary}\label{cor: finger.area}
    The area of the finger region bounded by any finger of rescaled flow $\bar{M}_\tau$ is $\pi + o(1)$ as $\tau\to -\infty$.
\end{corollary}

\begin{proof}
    Let $\bar{\gamma}$ be the arc-length parametrization of a given finger of $\bar{M}_\tau$. Denote the two ancient paths of the knuckles of the finger by $\bp(\tau), \bq(\tau)$ with $p(\tau) < q(\tau)$. 
    For each $\tau \ll -1$, we connect $\bar{\gamma}((\bp(\tau)), \bar{\gamma}(\bq(\tau))$ with a straight line segment $\bar{L}(\tau)$, and let $\bar{\Om}(\tau)$ denote the bounded region enclosed by the finger and $\bar{L}(\tau)$, which has a small area difference $= O(\rho(\tau)^{-4})$ from the one of the finger region in Definition~\ref{def: finger region}.
    
    By divergence theorem, the area of $\bar{\Om}(\tau)$ is given by
    \begin{align}\label{eq: area of finger}
        A(\tau) = \frac12 \int_{\bar{\Om}(\tau)} \Div_{\R^2}(\mathbf{x}) \, d\mathbf{x} = \frac12 \int_{p(\tau)}^{q(\tau)} \langle \bar{\ga}, \nu \rangle \,ds + \frac12 \int_{\bar{L}(\tau) \cap \pa \bar{\Om}(\tau)}  \langle \mathbf{x}, \nu\rangle \,ds
    \end{align}
    where $\nu$ is the outward unit normal of $\pa \bar{\Om}(t)$ and $ds$ is oriented counterclockwise.
    By differentiating $A(\tau)$ in $\tau$, integrating the integral involving $\tfrac{d \mathbf{\nu}}{d\tau}$ by parts, together with fact that the boundary terms and integral over $\bar{L}$ are uniformly bounded by $C\rho^{-4}$ for some constant $C$, and using evolution equation of rescaled flow, we get 
    \begin{align*}
        \tfrac{d}{d\tau} A(\tau) &= - \int_{p(\tau)}^{q(\tau)} \langle \tfrac{d\bar{\ga}}{d\tau}, \nu \rangle \,ds +  O(\rho^{-4}) \\
        &= - \int_{p(\tau)}^{q(\tau)} \bar{\kappa}(s, \tau) \,ds +  \frac12 \int_{p(\tau)}^{q(\tau)} \langle \bar{\ga}, \nu \rangle \,ds + O(\rho^{-4})\\
        &= \big(-\pi +  A(\tau)\big) + O(\rho^{-4}).
    \end{align*}
    The last equality uses the result about the angle difference of two adjacent knuckles of a finger in Theorem \ref{thm: uniqueness of vertex} and (\ref{eq: area of finger}). 
    Let $f(\tau) = A(\tau) - \pi$. Then $f(\tau)$ satisfies the ODE $\tfrac{d}{d\tau} f = f + E(\tau)$ where $\vert E(\tau)\rvert \leq C\rho(\tau)^{-4}$ uniformly for all $\tau \ll -1$. Since $\rho$ is monotone increasing, one can solve $f(\tau) = f(\tau_0) e^{\tau -\tau_0} + O(\rho(\tau_0)^{-4})$ for $\tau \ll \tau_0 \ll -1$. 
    Therefore, after taking $\tau \to -\infty$ and then $\tau_0 \to -\infty$, $A(\tau) = \pi + o(1)$.
\end{proof}

\

We conclude this section by studying the {\it edges}. See Definition~\ref{def: vertex} for edges.

\begin{definition}\label{def: bumpy}
    A complete smooth curve $M$ is called \emph{bumpy} if all the zeros of $\kappa$ and $\kappa_s$ on $M$ are discrete and simple.
\end{definition}

\ 

The following result shows that it suffices to consider bumpy curves in our study of finite-entropy ancient flows.

\begin{corollary}\label{cor: bumpy}
    For $t\leq T = T(\cM)$ as in Theorem \ref{thm: uniqueness of vertex}, $M_t$ is bumpy. 
\end{corollary}

\begin{proof}
    It follows from Theorem \ref{thm: uniqueness of vertex} that the (maximal) number of local extrema of curvature remains constant for all $t \leq T$. 
    Suppose that $\kappa_s(\cdot, t_0)$ has a multiple zero for some $t_0\leq T$. Applying Proposition \ref{prop: zeroset} (b) to equation (\ref{eq: evolution of k_s}), for some $t < t_0$, $\kappa(\cdot, t)$ has more local extrema than $\kappa(\cdot, t_0)$ does, a contradiction. 
    Similarly, if $\kappa(\cdot, t_{0})$ has a multiple zero for some $t_0\leq T$, then by applying Proposition \ref{prop: zeroset} to equation (\ref{eq: evolution of curvature}), for some $t < t_0$, $\kappa(\cdot, t)$ has more nodal domains than $\kappa(\cdot, t_0)$ does. There is one extra sharp vertex of $M_t$ in one of the extra nodal domains of $\kappa(\cdot, t)$, a contradiction.

\end{proof}

\ 

We now prove the uniqueness of inflection points and flat vertices.

\begin{lemma}\label{lem: uniqueness of inflection and flat vertex}
    Let $M$ be a complete smooth bumpy curve. Let $\Sigma\subset M$ be an edge with two end points $v_1,v_2$ which are sharp vertices. If $\kappa(v_1) \kappa(v_2)<0$, then $\kappa$ is strictly monotone on $\Sigma$ and $\Sigma$ has a unique inflection point. If $\kappa(v_1) \kappa(v_2)>0$, then $\Sigma$ is convex and it has a unique flat vertex.
\end{lemma}

\begin{proof}
If $\kappa(v_1) \kappa(v_2)<0$, then the curvature $\kappa$ is strictly monotone. 
Otherwise, since $\Sigma$ is bumpy, $\Sigma$ contains an interior local extremum ($\kappa_s = 0$), a contradiction to the definition of an edge. 
By the intermediate value theorem, $\lvert \kappa \rvert$ has a single local minimum when $\kappa=0$. 

If $\kappa(v_1) \kappa(v_2)> 0$, then $\kappa \neq 0$ on $\Sigma$. Otherwise, since $M$ is bumpy, we can find another vertex in the nodal domain in the interior of $\Sigma$, a contradiction. We may assume that $\lvert \kappa \rvert = \kappa$ with proper orientation. The unique minimum of $\lvert\kappa \rvert$ has $\kappa_s = 0$, which is a flat vertex. There are no other flat vertices; otherwise, we can find one sharp vertex between two flat vertices, a contradiction.
\end{proof}

\begin{proposition}\label{thm: curvature of tail is monotone}
    Suppose $\mathcal{M}$ is a non-compact ancient flow. Then, for $t \leq T(\cM)$ as in Theorem \ref{thm: uniqueness of vertex}, $M_{t}$ has two non-compact edges. 
    Furthermore, for each of the non-compact edges, their curvature $\kappa$ is nonzero and $\lvert \kappa \rvert$ monotonically decreases to $0$ as $\lvert x \rvert\to +\infty$.
\end{proposition}

\begin{proof}
Let $\gamma$ be the arc-length parametrization of $\cM$. 
Denote by $\bv_1$ and $\bv_2$ the ancient paths of sharp vertices on the leftmost and rightmost fingers of $\cM$, respectively.
Curved rays $E_1(t) := \gamma((-\infty, v_1(t))$ and $E_2(t) :=\gamma((v_2(t), \infty))$ do not contain any sharp vertex for any $t \leq T$.
Suppose the claim is false, say without loss of generality $E_1(t_0)$ contains a sharp vertex in addition to the existing ones on all fingers for some $t_0 \leq T$.
By Lemma~\ref{lem: curvature lower bound of vertex}, Theorem~\ref{thm: rough convergence}, and  Theorem~\ref{thm: curv decay of tail} (\ref{eq: curv decay of tail 2}), this extra ancient sharp vertex path from the ray of $E_1(t_0)$ must escape into the leftmost finger as $t \to -\infty$. 
This contradicts Theorem~\ref{thm: uniqueness of vertex}. 
Thus, for $t\leq T$, $E_1(t)$ and $E_2(t)$ are non-compact edges of $M_t$. 
Since $M_t$ is bumpy for all $t\leq T$, by Theorem~\ref{thm: curv decay of tail} (\ref{eq: curv decay of tail 1}) and the non-existence of sharp vertices, $\lvert \kappa \rvert > 0 $ and $\lvert \kappa \rvert$ monotonically decreases to $0$ as $\lvert s\rvert\to +\infty$ on non-compact edges. 
\end{proof}

\

\begin{proof}[Proof of Theorem~\ref{thm:main.geometry}]
    The numbers of fingers and knuckles are obvious by Corollary~\ref{cor: location of knuckles} and Corollary~\ref{cor: number of fingers}; 
    the number of sharp vertices follows from Theorem~\ref{thm: uniqueness of vertex} and Proposition~\ref{prop: no vertex on tail}; the sum of the numbers of flat vertices and inflection points follows from Lemma~\ref{lem: uniqueness of inflection and flat vertex} and Proposition~\ref{thm: curvature of tail is monotone}.
\end{proof}

\bigskip

 \section{Graphical radius lower bound}

  In this section, we will establish lower bounds for the graphical radius of each time slice of rescaled flows, in terms of the profiles' $C^2$-norm on a compact interval. We first define the $\varepsilon$-\textit{trombone time} to state the main theorem.

\begin{definition}\label{def:trombone}
Given $\varepsilon \in (0,10^{-1000})$, we call $\tau_\varepsilon$ the $\varepsilon$-trombone time of an ancient flow $\mathcal{M}$ if for every $\tau \leq \tau_\varepsilon$ the rescaled time slice $\overline{M}_\tau$ satisfies
\begin{enumerate}
    \item $\overline{M}_\tau$ is bumpy,
    \item each finger has a unique sharp vertex.
    \item the curvatures of the sharp vertex and the tip of each finger have the same sign,
    \item the angle difference $\theta$ between two adjacent knuckles satisfies $||\theta|-\pi| <\varepsilon$,
\end{enumerate}
and every path of sharp vertices of $\mathcal{M}$ for $t \leq -e^{-\tau_\varepsilon}$ lies at the tip of a $\frac{\varepsilon}{100}$-grim reaper.
\end{definition}

\begin{remark}\label{rmk:trb.time}
 The existence of the $\varepsilon$-trombone time has been proven by Theorem~\ref{thm: vertex path epsilon grim reaper}, Theorem~\ref{thm: uniqueness of vertex}, and Corollary~\ref{cor: bumpy}.   
\end{remark}

 \

 \begin{theorem}[graphical radius lower bound]\label{thm:grph.radius}
 Let $\mathcal{M}$ be an ancient flow with $\text{Ent}(\mathcal{M})=m \geq 2$. Then, given $\varepsilon \in (0,10^{-1000})$, $\mathcal{M}$ has the $\varepsilon$-trombone time $\tau_\varepsilon$ and there exists some $\delta\leq \varepsilon^8$ with the following significance. Suppose that
for each $\tau \leq \tau_\varepsilon$, there exists a rotation $S(\tau) \in SO(2)$ such that $S(\tau)\overline{M}_{\tau}\cap B_{2}(0)$ consists of the graphs of $m$ functions $u^1(y),\cdots,u^m(y)$ satisfying $ \|u^i\|_{C^2([-2,2])}\leq \delta$ for each $i\in \{1,\cdots,m\}$. Then, in the ball $B_\rho(0)$ of radius $\rho$, defined by
 \begin{equation}\label{graphical radius rho}
       \rho^{-4} :=\max_{1\leq i\leq m}\| u^i\|_{C^2([-1,1])},
 \end{equation}
$S(\tau)\overline{M}_{\tau}\cap B_\rho(0)$ is a union of the graphs of $m$ functions $u^1,\cdots,u^m$ satisfying 
\begin{align*}
    &|u^i(y)|\leq (y+2)^2\rho^{-4},\quad |u^i_y(y)|\leq \varepsilon \quad\mbox{for }\;|y|\leq \rho,\;\mbox{and}\\
    &|u^i_{yy}(y)|\leq 5\varepsilon \rho^{-1}\quad\mbox{for }\;|y|\leq \rho/2,
\end{align*}
for each $i\in \{1,\cdots,m\}$.
 \end{theorem}

\bigskip

\subsection{Setting-up}
We fix $\varepsilon\in (0,10^{-1000})$, $\delta \in (0, \varepsilon^8)$, and a time $\tau'\leq \tau_\varepsilon$ so that for all $i = 1,\ldots, m$, $\|u^i\|_{C^2([-2,2])}\leq \delta$. 
For $i=1, \cdots, m$, denote by $\Sigma^i$ the edge of $S(\tau')\overline{M}_{\tau'}$ that contains $\{(y,u^i(y)): |y|<2\}$, respectively. 
By rearrangement, we may assume that the $(i+1)$-th edge $\Sigma^{i+1}$ is adjacent to the $i$-th edge $\Sigma^i$. 
Note that if the curve $S(\tau')\overline{M}_{\tau'}$ is closed, then $\Sigma^1$ is adjacent to $\Sigma^m$ by the arrangement. 
Furthermore, we assume that the $i$-th edge $\Sigma^{i}$ is on top of the $(i+1)$-th edge $\Sigma^{i+1}$.

For each $i$, there is a maximal closed interval $I_i = [a_i, b_i]$, where $-\infty \leq a_i < -2$ and $2 < b_i \leq +\infty$, such that we can extend the domain of $u^i$ to $I_i$ satisfying
\begin{enumerate}
    \item $(y,u^i(y))\in S(\tau')\overline{M}_{\tau'}$ for $y\in I_i$, 
    \item $u^i\in C^\infty((a_i,b_i))$,
    \item $\lim_{y\downarrow a^i} |u^i_y|=+\infty$ if $a_i:=\inf I_i>-\infty$,
    \item $\lim_{y\uparrow b^i} |u^i_y|=+\infty$ if $b_i:=\sup I_i<+\infty$.
\end{enumerate}
Then, we parametrize $\Sigma^i \cap \{(y,u^i(y)):y\in I_i\}$ by $\gamma^i(y) =(y,u^i(y))$ and we denote by $\bt^i:= \pm\gamma^i_y/|\gamma^i_y|$,  $\bn^i:= J\bt^i$, and $\kappa^i:= \langle \bt^i_y/|\gamma^i_y|,\bn^i\rangle$, respectively, its unit tangent vector, unit normal vector, and curvature that is compatible with the global orientation in Section 2.3. 
In addition, in each $\Sigma^i \cap \{(y,u^i(y)):y\in I_i\}$, we can define the ordering—\emph{left} or \emph{right}—by comparing the ordering of the $y$-value; we can extend this ordering throughout $\Sigma^i$ by choosing a proper arc-length parametrization.
Next, we mark some points with analytic significance.
\begin{enumerate}
    \item Denote by $\gamma^i(k_i)$ the unique\footnote{See Corollary \ref{cor: location of knuckles}} knuckle of $\Sigma^i$. Note that $\lvert k_i\rvert \leq \delta \ll 1$ (Lemma \ref{lem: lim of minimum}).

    \item We define $c_i\in [k_i,b_i]$ as follows: If $\Sigma^i$ contains no inflection point on the right of $\gamma^i(k_i)$, then set $c_i=k_i$; if $\Sigma^i\cap \{(y,u^i(y)):y\in [k_i, b_i)\}$ contains the unique\footnote{See Lemma~\ref{lem: uniqueness of inflection and flat vertex}.} inflection point of $\Sigma^i$, then set $(c_i,u^i(c_i))$ to be the inflection point; otherwise, set set $c_i = b_i$.
    
    \item Let $\bar{d}_i = \sup\{ z \in (2, b_i): |u^i_y(y)|<\tfrac{\vare}{100}, \forall y\in (0, z)\}$. Note $|u^i_y(\bar d_i)| = \tfrac{\vare}{100}$.
    \item Let $d_i = \sup\{ z \in (2, b_i): |u^i_y(y)|<\varepsilon, \forall y\in (0, z)\}$. Note $|u^i_y(d_i)| =\vare$.
\end{enumerate}

\medskip
Now fix $i$. Up to rotation and reflection, we may assume $(10,u^i(10))$ belongs to a \emph{right-pointing} finger $\Gamma^i$ or a \emph{right-going} tail $\Phi^i$. Then, either $\Phi^i\subset \Sigma^i$, or $\Gamma^i\subset (\Sigma^i\cup\Sigma^{i + 1})$. Furthermore, let $\widehat{\Gamma}^i:=\Gamma^i\setminus \cup_{j=i,i+1}\{(y,u^j(y)):y\in I_j\}$ denote the non-graphical part of the finger $\Gamma^i$. See Figure \ref{fig:enter-label}.
 
\begin{figure}[h]
    \centering
    \includegraphics[width=0.7\linewidth]{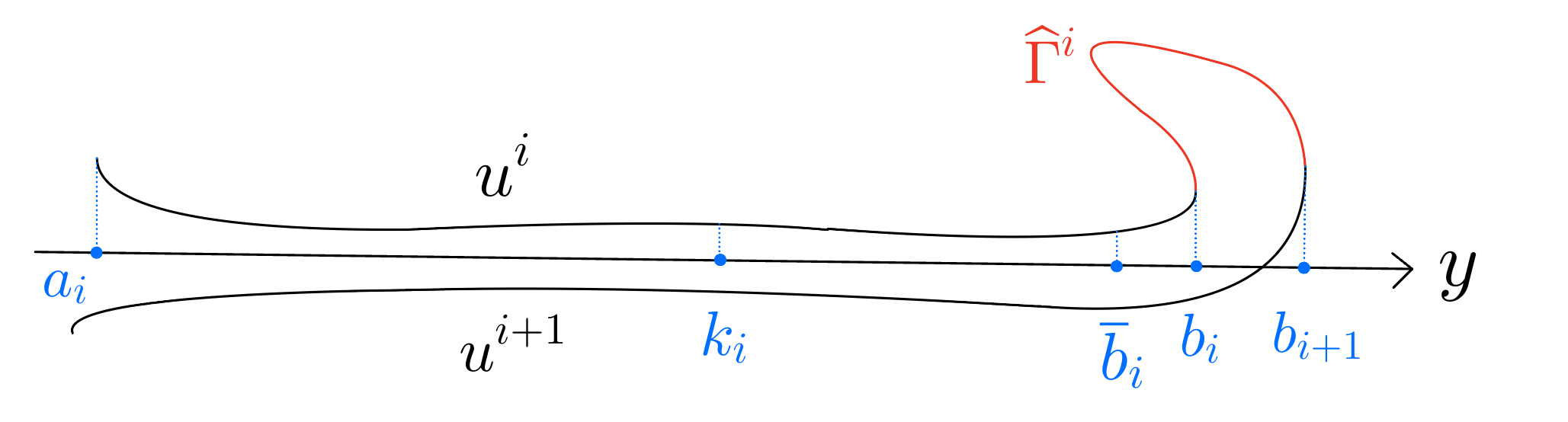}
    \caption{A right-pointing finger $\Gamma^i$ in which $\widehat{\Gamma}^i$ is the non-graphical part.}
    \label{fig:enter-label}
\end{figure}

\subsection{Tail estimate}
Let us consider the first case. We define $\rho>0$ by (\ref{graphical radius rho}).

\begin{lemma}\label{lem:grph.rad.tail}  
Suppose $\Sigma^i$ contains a right-going tail $\Phi^i$. Then, $\bar d_i\geq 15\rho^3$ holds. Namely, $\Phi^i\cap B_\rho(0)$ is the graph of $u^i$ with $|u^i_y|\leq \varepsilon$. 
\end{lemma} 

\begin{proof}
Remembering $|k_{i}|\leq \delta$, $|u^i_y(k_i)|\leq \delta \ll \varepsilon$, and $|\kappa^i|=|u^i_{yy}|(1+|u^i_y|^2)^{-\frac{3}{2}}$. We have
\begin{equation}
    \tfrac{1}{50}\varepsilon \leq |u^i_y(\bar d_i)-u^i_y(k_i)|=\left|\int_{k_i}^{\bar d_i}u^i_{yy}(y)dy\right|\leq 2\int_{k_i}^{\bar d_i}|\kappa^i(y)|dy.
\end{equation}
By Proposition~\ref{thm: curvature of tail is monotone}, $|\kappa^i(z)|\leq |\kappa^i(k_i)|$ for $z\geq k_{i}$. Hence,
\begin{equation}
    \tfrac{1}{50}\varepsilon \leq  2(\bar d_i-k_i)|\kappa^i(k_i)|\leq 3\bar d_i\|u^i\|_{C^2([-1,1])}\leq 3\bar d_i\rho^{-4}.
\end{equation}
The assumption $\rho^{-4}\leq \delta\leq  \varepsilon^8$ implies $\varepsilon\rho\geq \varepsilon^{-1}\geq 10^4$. Hence, we have $\bar d_i \geq 15\rho^3$. 
\end{proof}

\medskip

\subsection{Finger estimates}

Next, we deal with the case that $\Sigma^i\cup \Sigma^{i + 1}$ contains a right-pointing finger $\Gamma^i$.
 
\begin{lemma}\label{lem:gradient.inflection}
  If $c_i\geq \bar d_i$, then $\bar d_i\geq 15\rho^3$ holds. If $c_i\leq \bar d_i$, then $ |u^i_y(y)|\leq  2(y+1)\rho^{-4}$ and $|u^i(y)|\leq (y+2)^2\rho^{-4}$ hold for $y\in [k_{i},c_i]$.
\end{lemma}

\begin{proof}
If $ c_i \geq \bar d_i$, then by Lemma~\ref{lem: uniqueness of inflection and flat vertex}, curvature is monotone, and therefore we have $\bar d_i \geq 15\rho^3$ as the proof of Lemma \ref{lem:grph.rad.tail}. Now, we suppose that $c_i \leq \bar d_i$. Then, as the proof of Lemma \ref{lem:grph.rad.tail}, we have for $y \in [k_i,c_i]$,
\begin{equation}
    |u^i_y(y)|\leq |u^i_y(k_i)| +2(y-k_i)|\kappa^i(k_i)|\leq 2(y+1)\rho^{-4}.
\end{equation}
Then, a direct integration yields
\begin{equation}
    |u^i(y)|\leq |u^i(k_i)| + (y-k_i)^2 \rho^{-4} + 2(y-k_i)\rho^{-4} \leq (y+2)^2\rho^{-4}.
\end{equation}
\end{proof}

\medskip

\begin{lemma}\label{lem:infl.out.rad}
  If $c_i< b_i$, then $b_i > b_{i+1}$ holds unless $\gamma^i(b_i)=\gamma^{i+1}(b_{i+1})$. In addition, $(y,u^i(y))$ is not an inflection point for each $y\in (c_i,b_i]$.
\end{lemma}

 \begin{proof}
We assign the finger $\Gamma^i$ with the orientation as in Figure~\ref{fig:inside} such that the region $\{(y, w): 1 \leq y \leq 2, u^{i+1}(y) \leq w \leq u^{i}(y)\}$ is contained in inside. Since the unit normal points towards inside, the curveture at the tip of $\Gamma^i$ is positive and therefore, by Theorem~\ref{thm: uniqueness of vertex}, so is the curvature at the vertex of $\Gamma^i$ .
Since $c_i < b_i$, the curved segment of $\Sigma^i$ on the right of the unique inflection point $\gamma^i(c_i)$ has positive curvature everywhere. Note that the unit tangent vector points to the left along $\Sigma^i$. Thus, $\theta(\gamma(b_i)) = \pi/2 \mod 2\pi$, where $\theta$ is the angle of the unit tangent vector measured from the $+y$-direction counterclockwise.

Now, we suppose $\gamma^i(b_i)\neq \gamma^{i+1}(b_{i+1})$ and $b_i\leq b_{i+1}$ to get a contradiction. 
Then there is $ p_0\in \Gamma^i$ such that $p_0\neq \gamma^i(b_i)$ and $y(p_0) \geq y(p)$ holds for all $p\in \Gamma^i$. 
Note that the inside region must be in the $(-y)$-direction of $\Gamma^i$ around $p$, so the unit normal at $p_0$ points exactly toward the $(-y)$-direction, $\kappa(p_0) > 0$, and $\theta(p_0) = \pi/2 \mod 2\pi$. 
Consider the curve segment with positively oriented regular parameterization $\hat \gamma:[-2,2]\to \Gamma^i$ such that $\hat \gamma(-2) = p_0$, $\hat \gamma(0)=\gamma^i(b_i)$, $\hat \gamma(1)=\gamma^i(c_i)$, and $\hat \gamma(2)=\gamma^i(k_i)$. 
Since $\hat \gamma([-2,1))$ contains at most one inflection point and the curvature is positive at $\hat\gamma(-2)$ and in $\hat\gamma([0, 1))$, $\hat\gamma([-2,1))$ has positive curvature everywhere. 

To understand the rest of the proof, it will be helpful to refer to Figure~\ref{fig: spiral}. 
By $\theta(\hat\gamma(0)) = \pi/2 \mod 2\pi$ and convexity of $\hat\gamma$, there exists $\eta> 0$ such that $y(\hat\gamma(\sigma)) < b_i$ for any $\sigma \in (-\eta, 0)$. 
From the assumption $y(p_0) \geq b_i$, we can find $h = \inf\{ a \in [-2, 0)\::\: y(\hat\gamma(\sigma)) < b_i, \forall \sigma\in (a, 0)\}$. 
By embeddedness and structure in $B_2(0)$, it is clear that the curve segment $\hat\gamma(( h, 0))$ lies in the region $\mathcal{R} :=\{(y, w)\::\: 2< y < b_i, u^{i+1}(y) < w < u^i(y)\}$. 
By reparamatrizing $\hat\gamma$, we may assume that $y(\hat\gamma(\sigma))$ reaches a local minimum $y_* > 2$ at $\sigma = -1$ and $y(\hat\gamma(\sigma))$ is monotone decreasing on $(h, -1)$. 
Let $\mathcal{U}$ be the bounded region enclosed by $\hat\gamma([-1 , 2])$ and $\{y_*\} \times \R$.
By convexity again, $\theta(\hat\gamma(\sigma)) = -\pi/2 \mod 2\pi$ and there exists $\xi > 0$ so that $\hat\gamma(\sigma) \in \mathcal{U}$ for $\sigma \in (\xi, -1)$.
Then, by the embedding and monotonicity of $y(\hat\gamma(\sigma))$ in $(h, -1)$, we have $\hat\gamma((h, -1))\subset \mathcal{U}$. 
But recall that $y(\hat\gamma(h)) = b_i$ and $y\rvert_{\pa \mathcal{U}} < b_i$ except at $ \hat\gamma(0)$. This implies $\hat\gamma(y) = \hat\gamma(0)$, which contradicts the embeddedness.

\begin{figure}[h]
    \centering
    \includegraphics[width=0.8\linewidth]{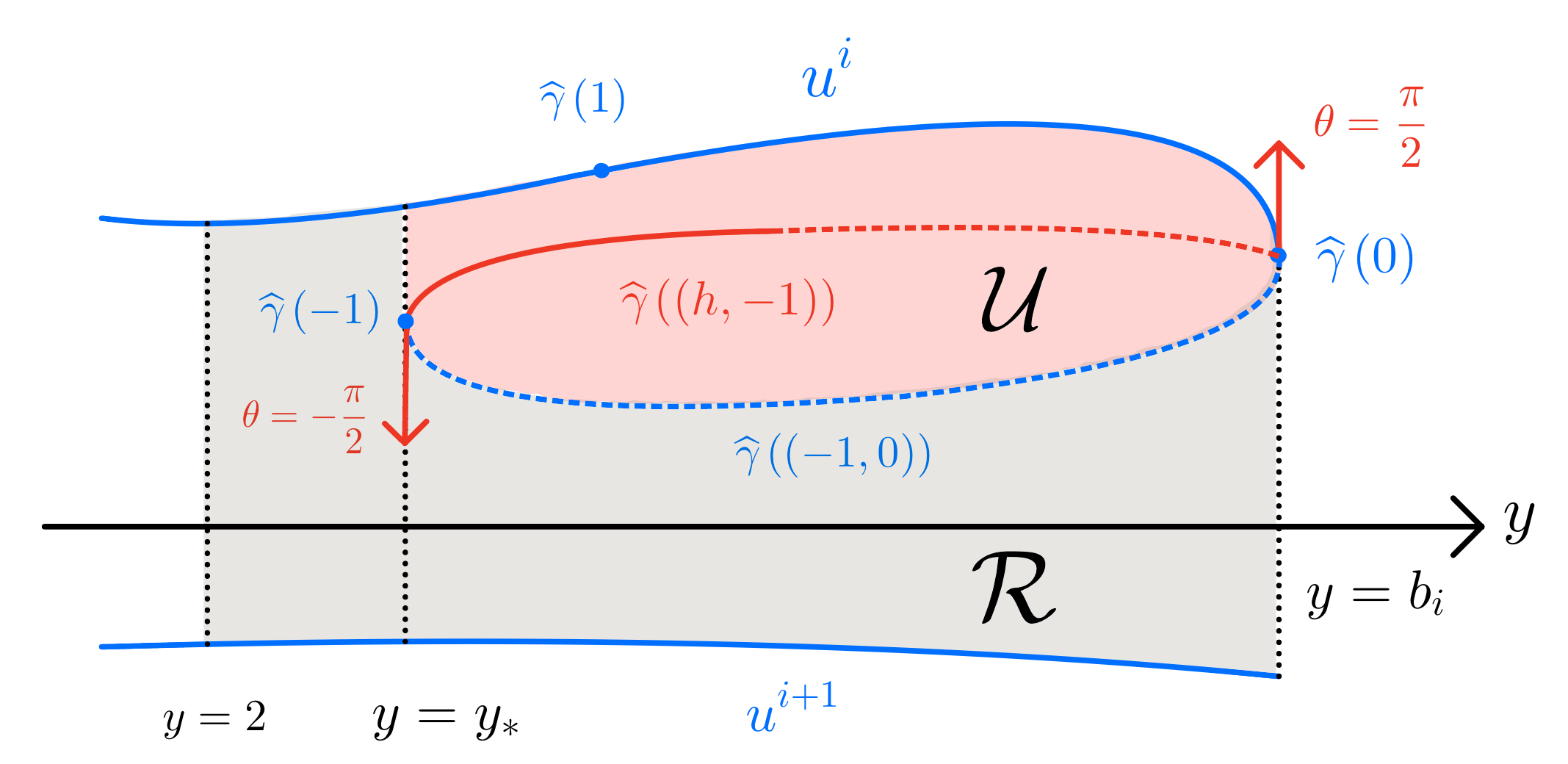}
    \caption{Violation of embedding since $\hat\gamma(h)$ must reach $y = b_i$.}
    \label{fig: spiral}
\end{figure}

\end{proof}
\medskip

\begin{lemma}\label{lem:gradient.flat lower sheet}
  If $c_i \leq \bar d_i$ and $c_{i+1}\geq \bar d_{i+1}$ hold, then $\Gamma^i \cap B_\rho(0)$ consists of the graphs of $u^i,u^{i+1}$ with $|u^i_y|,|u^{i+1}_y|\leq \varepsilon$. 
\end{lemma}

\begin{proof}
We recall that Lemma \ref{lem:gradient.inflection} says
\begin{equation}\label{eq:grph.rad.inflection}
\bar d_{i+1} \geq 15\rho^3.    
\end{equation}
Thus, we only need to exclude the case that $d_i\in (\bar d_i,\rho]$. By Lemma \ref{lem:infl.out.rad} we have $b_i\geq b_{i+1}\geq \bar d_{i+1} \geq 15\rho^3$. 

\medskip

Now, we first consider the case that $\Sigma^i$ has an inflection point. Then, by Lemma~\ref{lem: uniqueness of inflection and flat vertex} $|\kappa^i(y)|$ increases in $\{y \geq c_i\}$, and thus by mean value theorem, $|\kappa^i(y)|\geq \frac{99}{200}\varepsilon(d_i-c_i)^{-1} \geq \frac{1}{3}\varepsilon\rho^{-1}$ holds for $y \geq d_i$. Therefore,
\begin{equation}
\tfrac{1}{3} \varepsilon \rho^{-1}(b_i-d_i) \leq  \int_{d_i}^{b_i} |\kappa^i(z)|dz\leq \int_{d_i}^{b_i} |\kappa^i(z)| \sqrt{1+(u^{i}_{y})^{2}(z)}\:dz\leq \pi.
\end{equation}
Thus, combining with $b_i\geq 15\rho^3$ yields $1 \geq \varepsilon \rho^2$. This contradicts $\varepsilon \rho^2 \geq \varepsilon^{-3}\geq 10^{12}$.
\medskip

Next, we assume that $\Sigma^i$ is convex. Recall that in such a case $k_i = c_i$. From assumption, $u^i>u^{i+1}$ in $I_i\cap I_{i+1}$. Then, $|u^i_y(d_i)| =\varepsilon \gg  \varepsilon^2\geq|u^i_y(c_i)|$ and $d_i > c_i$ implies $u^i_y(d_i)=-\varepsilon$. We consider the tangent line $\ell := \{(y, l(y)=-\varepsilon (y-d_i)+u^i(d_i)): y\in\mathbb{R}\}$ of $u^i$ at $y=d_i$, which satisfies $l(y) \geq u^i(y)\geq u^{i+1}(y)$ for $y\in [d_i,b_{i+1}]$. Also, $|u^i_y(y)|\leq \varepsilon$  for $y\in [k_i,d_i]$ yields $|u^i(d_i)|\leq |u^i(k_i)|+\varepsilon(d_i-k_i)\leq \rho^{-4}+2\varepsilon \rho $.

The tangent line $\ell$ intersects the horizontal line $\mathfrak{h}:=\{(y, -15\rho^{3}\frac{\varepsilon}{100}-\varepsilon^{8}):y\in\mathbb{R}\}$ at $(\bar{y}, -15\rho^{3}\frac{\varepsilon}{100}-\varepsilon^{8})$, with $\bar{y} = d_i + \frac{u^{i}(d_{i})}{\varepsilon}+\frac{15\rho^{3}}{100}+\varepsilon^{7}\leq3\rho + \frac{3}{20}\rho^{3}+2\varepsilon^{7}\leq \rho^{3}\leq\overline{d}_{i+1}$. 
Since $\Sigma^{i}$ is convex and $\Gamma^{i}$ is embedded, $\Sigma^{i}$ lies on the left of $\ell$ and on the top of $\mathfrak{h}$, which implies that
\begin{align*}
    b_{i}\leq\overline{y}\leq\overline{d}_{i+1}<b_{i+1},
\end{align*}
contradicts Lemma~\ref{lem:infl.out.rad}. See Figure~\ref{fig: prob.fig}.
\end{proof}

\begin{figure}[h]
    \centering
    \includegraphics[width = 0.8\textwidth]{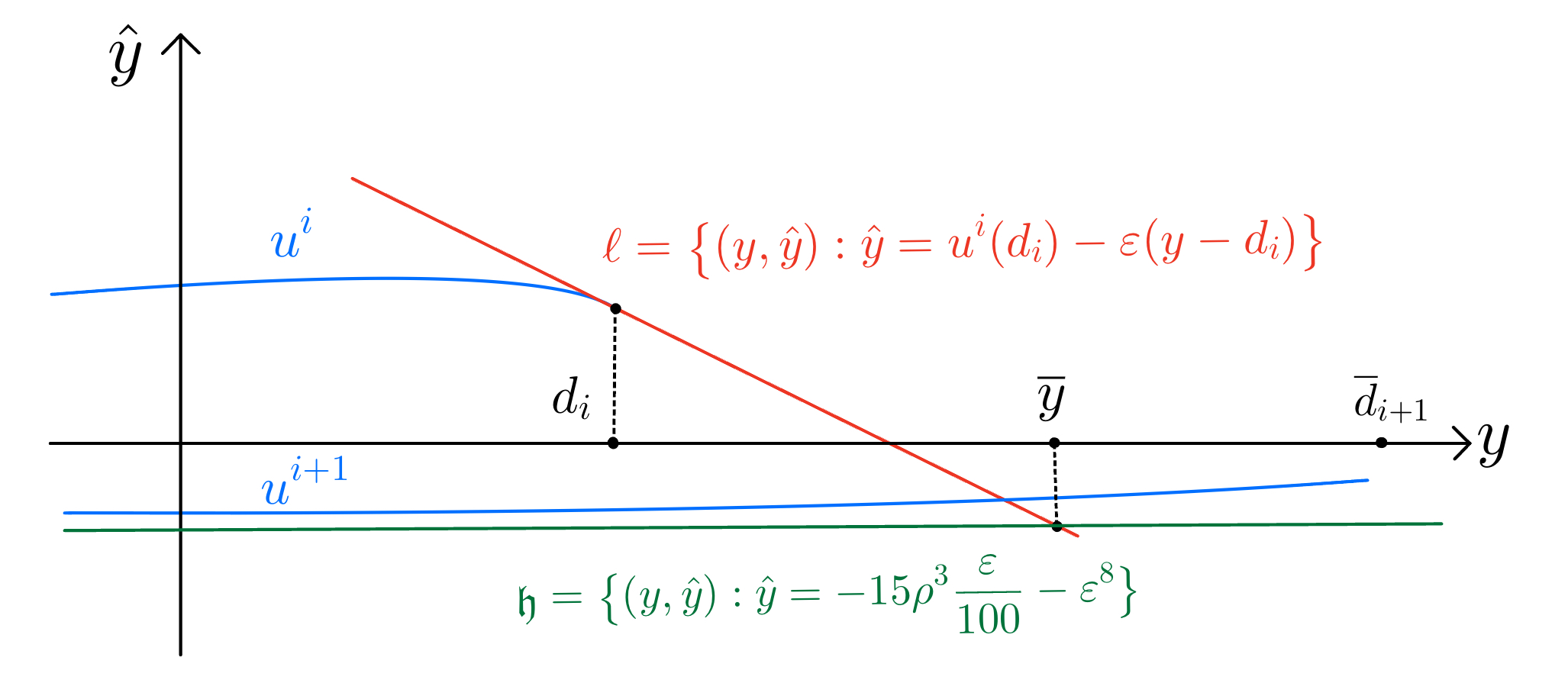}
    \caption{The problematic figure that contradicts Lemma~\ref{lem:infl.out.rad}.}
    \label{fig: prob.fig}
\end{figure}

\begin{lemma}\label{lem:grph.two.inflections}
  If $c_i \leq \bar d_i$ and $c_{i+1}\leq \bar d_{i+1}$ hold, then $\Gamma^i\cap B_\rho(0)$ consists of the graphs of $u^i,u^{i+1}$ with $|u^i_y|,|u^{i+1}_y|\leq \varepsilon$. 
\end{lemma}

\begin{proof}
    Lemma \ref{lem:infl.out.rad} implies that either $b_i>b_{i+1}$ or $\gamma^i(b_i)=\gamma^{i+1}(b_{i+1})$. By the uniqueness of the inflection point (Lemma~\ref{lem: uniqueness of inflection and flat vertex}), $b_i>b_{i+1}$ is impossible. Hence, we are in the situation of $\gamma^i(b_i)=\gamma^{i+1}(b_{i+1})$. 
    We may assume $u^i\geq u^{i+1}$ and denote $b:=b_i=b_{i+1}$. 
    Since $c_i\leq \bar d_i$, $c_{i+1}\leq \bar d_{i+1}$, and $\widetilde{\Gamma}^i:=\cup_{j=i,i+1}\{\gamma^j(y):y\in [c_j,b]\}$ is convex, we have
    \begin{equation}\label{eq:tip.total.curvature}
        \int_{\widetilde{\Gamma}^i} |\kappa|\:ds\leq \pi +\tfrac{3}{100}\varepsilon.
    \end{equation}
    By Theorem~\ref{thm: vertex path epsilon grim reaper} and \ref{thm: uniqueness of vertex}, choosing sufficiently negative $T$ we have a unique vertex $p\in \Gamma^{i}$ of the finger, and $(\Gamma^{i}-p)|\kappa(p)|$ is $\frac{\varepsilon}{100}$-close in $B(0, \tfrac{100}{\varepsilon})$ to a grim reaper curve with unit speed whose tip is the origin. 
    Note that for a unit-speed grim reaper curve $G$ with tip at the origin, we have the estimate that for any $r > 200\pi$
    \begin{align*}
        \int_{G\cap B(0, r)} \kappa \: ds \geq \pi - e^{-\frac{99}{100}r}.
    \end{align*}
    With $r = \log \big(\frac{50}{\vare}\big)$, considering the deviation from the unit-speed grim reaper and scaling invariance of the total curvature, we find
    \begin{align}\label{eq:tip.total.curvature.1}
        \int_{\widetilde{\Gamma}^i\cap B(p, \sigma)} |\kappa|\:ds\geq \pi - \textstyle\frac{\varepsilon}{20},
    \end{align}
    where $\sigma := \log \big(\frac{50}{\vare}\big)\lvert \kappa(p) \rvert^{-1}$.
    Combining \eqref{eq:tip.total.curvature} and \eqref{eq:tip.total.curvature.1} gives
    \begin{align*}
        \int_{\widetilde{\Gamma}^i\setminus B(p, \sigma)}|\kappa|\:ds\leq\textstyle\frac{\varepsilon}{10}.
    \end{align*}
    Note that from \eqref{eq: impro.vert.l.b.0} in Corollary~\ref{prop: impro.vert.l.b} , by taking $T \ll -1$, we can make $\sigma \leq 10^{-1}$. 
    Thus, on the interval $0\leq y\leq b- 2\sigma$
    we have
    \begin{align*}
        |u^{i}_{y}(y)|\leq |u^{i}_{y}(c_{i})| + \int_{c_{i}}^{b-2\sigma}|u^{i}_{yy}(z)|\:dz\leq\textstyle\frac{\varepsilon}{100}+\frac{\varepsilon}{10}\leq\varepsilon.
    \end{align*}
    Similarly, $|u^{i+1}_{y}|\leq\varepsilon$.
    Therefore, it suffices to show that $b \geq \rho + 1$.

\begin{figure}[h]
    \centering
    \includegraphics[width = 0.7\textwidth]{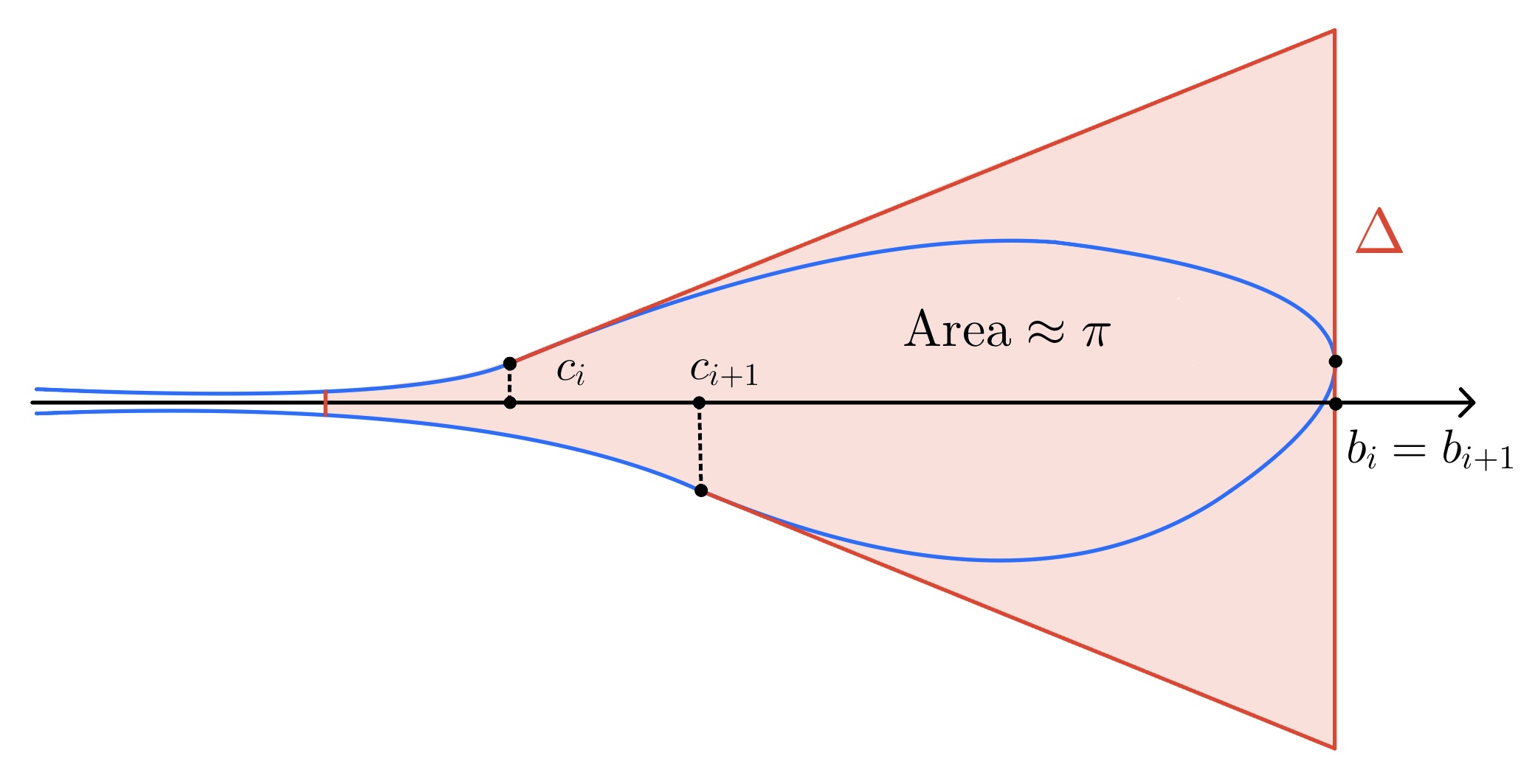}
    \caption{Area of a finger region and the extended triangle $\Delta$. }
    \label{fig: area comparison}
\end{figure}

To prove the claim, we only need to consider the case $c_i,c_{i+1}\leq \rho + 1$, because $b\geq c_i,c_{i+1}$. 
For each $j=i,i+1$ we define a function $v^j$ by $v^j=u^j$ for $y\leq c_j$ and $v^j=u^j_y(c_j)y+u^j(c_j)$. 
Then, we have $u^i\leq v^i$ and $u^{i+1}\geq v^{i+1}$. 
Thus, the region $\Delta :=\{(y,z): 0\leq y \leq b,\: v^{i+1}(y)\leq z \leq v^i(y)\}$ includes the finger region of $\Gamma^i$ (see Figure~\ref{fig: area comparison}). 
By Corollary~\ref{cor: finger.area} we have the area lower bound $\rm{Area}(\Delta) \geq 3$ for $T \ll -1$. Using $|v^i-v^{i+1}|\leq |v_i|+|v_{i+1}|$, we have
\begin{equation}\label{eq: area.l.b}
    3 \leq {\rm{Area}}(\Delta)\leq \int^{b}_0 |v^i(y)|+|v^{i+1}(y)|\:dy.
\end{equation}
On the other hand, by Lemma \ref{lem:gradient.inflection}, we have
\begin{align}\label{eq: area.u.b}
    \rho^4\int^{b}_0 |v^j(y)|dy &\leq \int_{0}^{c_j} (y+2)^2 dy+\int_{c_j}^b 2(c_j+1)y +(c_j+2)^2 dy\leq 10\rho^3+3\rho b^2
\end{align}
for $j=i, i+1$.
Combining \eqref{eq: area.l.b} and \eqref{eq: area.u.b} gives $3\rho^4\leq 20\rho^3+6\rho b^2$.
Therefore, $\rho^3\leq 3b^2$.
This implies that $b > \rho + 1$ as desired. 
\end{proof}
 
 \medskip

\subsection{Proof of Theorem~\ref{thm:grph.radius}}
\begin{proof}[Proof of Theorem \ref{thm:grph.radius}]
As discussed in Remark \ref{rmk:trb.time}, the $\varepsilon$-trombone time does exist for $\mathcal{M}$. Then, for each fixed $\tau'\leq \tau_\varepsilon$, the profiles $u^i$ satisfy the gradient estimate $|u^i_y|\leq \varepsilon$ in $B_{\rho}(0)$ by Lemma~\ref{lem:gradient.inflection}, \ref{lem:gradient.flat lower sheet}, and \ref{lem:grph.two.inflections}. 

Next, we show $|u^i|\leq (y+2)^2\rho^{-4}$. Consider a finger $\Gamma^i\subset (\Sigma^i\cup \Sigma^{i+1})$. We may assume $u^i\geq u^{i+1}$ and recall the profile functions $v^i,v^{i+1}$ for the tangent lines in the proof of Lemma \ref{lem:grph.two.inflections}. Note that for each $j=i,i+1$, $v^j$ is well-defined in the interval $[k_j,\rho]$ even if $c_j\geq \bar b_j$ by Lemma \ref{lem:gradient.inflection}. Thus, we have $v^{i+1}\leq u^{i+1}\leq u^i\leq v^i$. Hence, observing $|v^i|,|v^{i+1}|\leq (y+2)^2\rho^{-4}$, the estimate for $|u^{i}|$ follows. 

Finally, we show $|u^i_{yy}|\leq 5\varepsilon \rho^{-1}$ for $|y|\leq \rho/2$. We have $|u^i_{yy}|\leq 2 |\kappa^i|\leq 2\rho^{-4}$ holds for $y\leq c_i$. Suppose $|u^i_{yy}(y_i)|=5\varepsilon \rho^{-1}$ for some $y_i \in [c_i,\rho/2]$. Then, the monotonicity of the curvature implies $|u^i_{yy}(y_i)|\geq 4\varepsilon \rho^{-1}$ for $y\in  [y_i,\rho]$, and therefore $|u^i_{y}(y_i)-u^i_{y}(\rho)|\geq 2\varepsilon$, a contradiction.
\end{proof}

\bigskip

\section{Uniqueness of tangent flow}

In this section, we show that the rescaled flow converges exponentially fast to a unique tangent flow by using the method of Allard-Almgren. Since we already classified low-entropy flows, we assume in this section that $\mathcal{M}$ is a nontrivial ancient flow with $\text{Ent}(\mathcal{M})=m\geq 2$.

\subsection{Setting-up for spectral analysis}

Suppose that the graph of a function $u:I \times [t_1,t_2] \to \mathbb{R}$ is a rescaled curve-shortening flow. Then, the profile $u$ satisfies
\begin{equation}
    u_\tau=\frac{u_{yy}}{1+u_y^2}-\frac{y}{2}u_y+\frac{u}{2}=\mathcal{L}u -\frac{u_y^2u_{yy}}{1+u_y^2},
\end{equation}
namely
\begin{equation}
    \mathcal{L}:=\partial_{y}^2-\tfrac{1}{2}y\partial_y+\tfrac{1}{2}.
\end{equation}
We recall the inner product
\begin{equation}
    \langle f,g\rangle_\mathcal{H}:=\int_{\mathbb{R}}  \frac{fg}{\sqrt{4\pi}} e^{-\frac{y^2}{4}}dy,
\end{equation}
which defines the Gaussian $L^2$-norm $\|\cdot\|_{\mathcal{H}}$. Since $\mathcal{L}-\frac{1}{2}$ is the Ornstein–Uhlenbeck operator, $\mathcal{L}$ has an eigendecomposition of the Gaussian $L^2$-space. In particular, $\varphi_1(y):=1$ is the only unstable eigenfunction of $\mathcal{L}$, and $\varphi_2(y):=2^{-\frac{1}{2}}y$ is the only neutral eigenfunction of $\mathcal{L}$. Moreover, we can directly compute
\begin{align}
&\|\varphi_1\|_{\mathcal{H}}=\|\varphi_2\|_{\mathcal{H}}=1, && \mathcal{L}\varphi_1-\tfrac{1}{2}\varphi_1=0, && \mathcal{L}\varphi_2=0.
\end{align}
Now, we define the projections $P_+,P_0,P_-$ to the unstable, neutral, stable spaces, respectively. Namely,
\begin{align}
   & P_+f:=\langle \varphi_1,f\rangle_{\mathcal{H}}\varphi_1, & P_0f:=\langle \varphi_2,f\rangle_{\mathcal{H}}\varphi_2, && P_-f:=f-P_+f-P_0f.
\end{align}
Note that $\varphi_3(y):=2^{-\frac{3}{2}}(y^2-2)$ is the stable eigenfunction with the least positive eigenvalue $\lambda_3=\frac{1}{2}$. Thus,
\begin{equation}
    \langle P_- f,\mathcal{L} P_-f\rangle_{\mathcal{H}}\leq -\tfrac{1}{2}\|P_-f\|_{\mathcal{H}}^2.
\end{equation}

\bigskip

On the other hand, the profile of a nontrivial rescaled ancient flow is not an entire function. Hence, we consider a cut-off function $0\leq \eta \in C^\infty(\mathbb{R})$ such that $\eta(s)\equiv 1$ for $|s|\leq 1$, $\eta(s)\equiv 0$ for $|s|\geq 2$, and $|\eta'|,|\eta''|\leq 2$ in $\mathbb{R}$. Then,  given $r\geq 1$ we define 
\begin{equation}
    \hat u(y,\tau):=u(y,\tau)\eta(y/r),
\end{equation}
and
\begin{equation}
   E(u,r):= \hat u_\tau- \mathcal{L}\hat u.
\end{equation}
For simplicity, we denote by $\mathcal{P}^\tau_{r,s}$ the spacetime region in $\mathbb{R}\times (-\infty,0]$ that 
\begin{equation}
    \mathcal{P}^\tau_{r,s}:=(-r,r)\times (\tau-s,\tau].
\end{equation}

\begin{definition}
    We say that a rescaled ancient flow $\overline{M}_\tau$ is $\delta$-linear with multiplicity $m$ at time $\tau'$ if there exist some functions $u^i:\mathcal{P}^{\tau'}_{4,2}\to \mathbb{R}$ for $i=1,\cdots,m$ such that for each $\tau \in [\tau'-2,\tau']$, $\overline{M}_\tau\cap B_{4}(0)$  consists of the graphs of profiles $u^1(\cdot,\tau),\cdots,u^m(\cdot,\tau)$ and $\|u^i\|_{C^2(\mathcal{P}^{\tau'}_{4,2})}\leq \delta$ holds for every $i=1,\cdots,m$.
\end{definition}

To deal with $u^1,\cdots,u^m$ together, we use the notation
\begin{align}
  \mathbf{u}  =( u^1,\cdots, u^m).
\end{align}
If $u^i$ are profiles of a flow, then we call $\mathbf{u}$ a profile bundle of the flow. Also, we denote 
\begin{align}
&    a^i(\tau):=\langle  \hat u^i,\varphi_1\rangle_{\mathcal{H}}, && b^i(\tau):=\langle  \hat u^i,\varphi_2\rangle_{\mathcal{H}}.
\end{align}

\bigskip

 \subsection{Improvement of flatness in Gaussian $L^2$-space}  Let us show that the profiles become closer to linear functions in the norm $\|\cdot\|_{\mathcal{H}}$ as $\tau \to -\infty$. We begin by gluing the domains $\mathcal{P}^\tau_{4,2}$ of the profiles $u^i$ in order to extend the domains with a fixed rotation.

\begin{proposition}\label{prop:glued.domain}
 There are numeric constants $C_0> 1$ and $\sigma_0 \in (0,1)$ with the following significance. Suppose that for some $\tau_1 \ll -1$ and $L\in \mathbb{N}$ with $L \geq 2$, every $\tau'\leq \tau_1$ has a rotation $S_{\tau'}\in SO(2)$ such that  $S_{\tau'}\overline{M}_\tau$ is $\delta$-linear with multiplicity $m$ at $\tau'$ where $\delta  \leq \sigma_0/L $. Then,  there is a profile bundle $ \mathbf{u}:\mathcal{P}^{\tau'}_{3,4L}\to \mathbb{R}^m$ of $S_{\tau'}\overline{M}_{\tau'}$ such that
 \begin{equation}
     S_{\tau'}\overline{M}_\tau\cap B_3(0)\subset 
    \{(y,\mathbf{u}(y,\tau)):|y|\leq 3\},
 \end{equation}
and
\begin{equation}
\|\mathbf{u}\|_{C^2(\mathcal{P}^{\tau
}_{3,4L})}\leq C_0L\delta,
\end{equation}
hold for $\tau \in [\tau'-4L,\tau']$.
\end{proposition}

\begin{proof}
There are $S_{\tau'},S_{\tau'-1},\cdots, S_{\tau'-4L+2}\in SO(2)$ such that for each $j=0,\cdots,4L-2$, $S_{\tau'-j}\overline{M}_\tau\cap B_4(0)$ consists of the graph of a profile bundle $\mathbf{u}^j:\mathcal{P}^{\tau'-j}_{4,2}\to \mathbb{R}^m$ satisfying $\|\mathbf{u}^j\|_{C^2(\mathcal{P}^{\tau'-j}_{4,2})}\leq \delta$. Let $\mathbf{u}^j:=(u^{1,j},\cdots,u^{m,j})$ and denote by $\Gamma^{1,j}$ the graph of $u^{1,j}$. Then, we may even assume  $u^{1,j}>\cdots>u^{m,j}$ and $S_{\tau'-j}^{-1}\Gamma^{1,j}_\tau=S_{\tau'-j+1}^{-1}\Gamma^{1,j}_\tau$ in the overlapping interval $[\tau'-j-1,\tau'-j]$. Thanks to the overlapping interval, we have $|S_{\tau'-j}-S_{\tau'-j+1}|\leq C\delta$. Therefore, remembering $\delta L\ll 1$, we have $|S_{\tau'}-S_{\tau'-j}|\leq CL\delta$ for all $j=0,\cdots,4L-2$. This completes the proof.
\end{proof}

\bigskip

 \begin{lemma}\label{lem:rough.weighted.C2}
Given $\varepsilon \in (0,10^{-1000})$, there is some constant $\sigma_1(\mathcal{M},\varepsilon) \leq C_0\sigma_0$ with the following significance. 
Let $\tau_\varepsilon(\mathcal{M})$ denote the $\varepsilon$-trombone time. Suppose that 
 \begin{equation}
\|\mathbf{u}\|_{C^2(\mathcal{P}^{\tau'}_{3,4L})}\leq (4r)^{-4}\leq \sigma_1,
\end{equation}
holds for some $\tau'\leq \tau_\varepsilon$, $r \geq 100$, and $L\geq 2$. Then, for each $i=1,\cdots,m$,
\begin{align}\label{eq:rough.esti_profile}
&|u^i|\leq (4r)^{-4}(|y|+2)^2, && |u^i_y|\leq \varepsilon, && |u^i_{yy}|\leq \tfrac{5\varepsilon}{4r},
\end{align}
 hold in $\mathcal{P}^{\tau'}_{2r,4L}$. In particular, each function $\hat u^i:=\eta(y/r)u^i$ satisfies
\begin{equation}\label{eq:rough.Gaussian.L2}
    \|\hat u^i\|_{\mathcal{H}}\leq Cr^{-4}
\end{equation}
in $\tau \in [\tau'-4L,\tau']$, where $C$ is  some numeric constant.
 \end{lemma}

\begin{proof}
 The graphical radius lower bound theorem \ref{thm:grph.radius} implies \eqref{eq:rough.esti_profile}. Then, the first inequality in \eqref{eq:rough.esti_profile} yields \eqref{eq:rough.Gaussian.L2}. This completes the proof.
\end{proof}

 \bigskip

\begin{lemma}\label{lem:error.proj}
We recall $\mathbf{u},\delta ,L,\tau_1,C_0$ from Proposition \ref{prop:glued.domain} and assume $\tau_1 \leq \tau_\varepsilon$ for some $\varepsilon\in (0,10^{-1000})$ and $r:=\frac{1}{4}(C_0L\delta)^{-\frac{1}{4}}\geq 100$ by choosing a small enough $\delta$. Then, the profile bundle $\mathbf{u}$ of $S_{\tau'}\mathcal{M}_\tau$ is well-defined on $\mathcal{P}^{\tau'}_{2r,4L}$ for every $\tau' \leq \tau_1$. Moreover, each $E^i:=\hat u^i_\tau-\mathcal{L}\hat u^i$ satisfies
\begin{equation}\label{eq:error.proj.finite.space}
    \|P_+E^i\|_{\mathcal{H}} + \|P_0 E^i\|_{\mathcal{H}} \leq C\varepsilon \|  \hat u^i_y\|_{\mathcal{H}}^2+Cr^{-\frac{3}{2}}e^{-\frac{r^2}{8}},
\end{equation}
 and
 \begin{equation}\label{eq:error.proj.stable}
      \|P_-E^i\|_{\mathcal{H}} \leq C\varepsilon   r^{-1}\|  \hat u^i_y\|_{\mathcal{H}}+Cr^{-\frac{3}{2}}e^{-\frac{r^2}{8}},
 \end{equation}
for $\tau \in [\tau'-4L,\tau']$, where $C$ is some numeric constant.
\end{lemma}

\begin{proof}
    For the sake of brevity, we drop the index $i$ in this proof. We observe
    \begin{equation}
        E(u,r)=\hat u_\tau-\mathcal{L}\hat u=E_1+E_2,
    \end{equation}
where
\begin{equation}
    E_1:=-\frac{\hat u_y^2u_{yy}}{1+u_y^2},
\end{equation}
and
\begin{align}
  E_2:=-\frac{\eta'' u}{r^2}-\frac{2\eta' u_y}{r} +\frac{y\eta' u}{2r} + \frac{u_{yy}}{1+u_y^2}\left(\eta(\eta-1) u_y^2+\frac{2\eta \eta'uu_y}{r} +\frac{|\eta'|^2u^2}{r^2}  \right).
\end{align}
In the set $J:=\{r\leq |y| \leq 2r\}$, remembering $r\geq 100$ and $\varepsilon \leq 10^{-1000}$, Lemma \ref{lem:rough.weighted.C2} yields
 \begin{equation}
     |E_2|\leq r^{-1}\chi_J,
 \end{equation}
 where $\chi_J(y) =1$ for  $y\in J$ and $\chi_J(y) =0$ for $y \not \in J$. Hence,
 \begin{equation}\label{eq:Error2.estimates}
     \|E_2\|_{\mathcal{H}}^2\leq \int_{r}^{2r}  r^{-2}(4\pi)^{-\frac{1}{2}} e^{-\frac{s^2}{4}}ds\leq Cr^{-3}e^{-\frac{r^2}{4}},
 \end{equation}
 for some numeric constant $C$.   

\bigskip

 Next, using Lemma \ref{lem:rough.weighted.C2}, we have
 \begin{equation}
 |\langle E_1,y\rangle_\mathcal{H}| \leq  \int |yu_{yy}|\hat u_y^2 e^{-\frac{y^2}{4}}\leq C\varepsilon \|\hat u_y\|_{\mathcal{H}}^2.  
 \end{equation}
 Similarly, we can obtain $|\langle E_1,1\rangle_\mathcal{H}|\leq C\varepsilon \|\hat u_y\|_{\mathcal{H}}^2$. Thus, combining with \eqref{eq:Error2.estimates} yields \eqref{eq:error.proj.finite.space}. 

 \bigskip
To show \eqref{eq:error.proj.stable}, by using Lemma \ref{lem:rough.weighted.C2}, we directly compute
\begin{equation}
     \int E_1^2e^{-\frac{y^2}{4}}\leq \int  u_{yy}^2\hat u_y^4 e^{-\frac{y^2}{4}} \leq  \frac{C \varepsilon^2(\varepsilon+r^{-3})^2}{r^2}\int   \hat u_y^2 e^{-\frac{y^2}{4}}.
\end{equation}
Since $\|P_-E\|_{\mathcal{H}}^2\leq \|E\|_{\mathcal{H}}^2 \leq \|E_1\|_{\mathcal{H}}^2+\|E_2\|_{\mathcal{H}}^2$, combining with \eqref{eq:Error2.estimates} completes the proof.
\end{proof}

\bigskip
\begin{lemma}\label{lem:ODEs}
There is some numeric constant $C$ such that each $\hat u^i$ satisfies
\begin{align}
    \big|  \tfrac{1}{2}a^i-\tfrac{d}{d\tau}a^i \big| &\leq C \varepsilon\|  \hat u^i_y\|_{\mathcal{H}}^2+Cr^{-\frac{3}{2}}e^{-\frac{r^2}{8}},\\
    \big|\tfrac{d}{d\tau} b^i\big| &\leq C\varepsilon \|  \hat u^i_y\|_{\mathcal{H}}^2+Cr^{-\frac{3}{2}}e^{-\frac{r^2}{8}}, \label{eqn:2}\\
    \tfrac{d}{d\tau}\|P_-\hat u^i\|_{\mathcal{H}}^2 & \leq -\tfrac{1}{2} \|(P_-\hat u^i)_y\|_{\mathcal{H}}^2-\tfrac{1}{4}\|P_-\hat u^i\|_{\mathcal{H}}^2+C\varepsilon^2 r^{-2}\|\hat u^i_y\|_{\mathcal{H}}^2+Cr^{-3} e^{-\frac{r^2}{4}}, \label{eqn:1}
\end{align}
in $\tau \in [\tau'-4L,\tau']$.
\end{lemma}

\begin{proof}
We drop the index $i$ for the sake of brevity. We begin by observing
\begin{equation}
    a'= \langle \hat u_\tau,\varphi_1 \rangle_{\mathcal{H}}=\langle \hat u ,\mathcal{L} \varphi_1 \rangle_{\mathcal{H}}+\langle E,\varphi_1 \rangle_{\mathcal{H}}=\tfrac{1}{2}a+\langle E,\varphi_1 \rangle_{\mathcal{H}}.
\end{equation}
Thus, Lemma \ref{lem:error.proj} implies the first inequality. Similarly, by deriving $b'=\langle E,\varphi_2\rangle_{\mathcal{H}}$, we can obtain the second inequality. Lastly,
\begin{align}
       \tfrac{d}{d\tau}\|P_-\hat u\|_{\mathcal{H}}^2=2\langle \hat u_\tau ,P_-\hat u\rangle_{\mathcal{H}}\leq 2\langle \mathcal{L} P_-\hat u ,P_-\hat u\rangle_{\mathcal{H}} +2\|P_-E\|_{\mathcal{H}}\|P_-\hat u\|_{\mathcal{H}}.
\end{align}
Since $\lambda_k \geq\frac{1}{2}$ for $k\geq 3$, we have
\begin{equation}
    \tfrac{3}{2}\langle \mathcal{L} P_-\hat u ,P_-\hat u\rangle_{\mathcal{H}}\leq -\tfrac{3}{4}\|P_-\hat u\|_{\mathcal{H}}^2.
\end{equation}
Also, by integration by parts, we have
\begin{equation}
    \tfrac{1}{2}\langle \mathcal{L} P_-\hat u ,P_-\hat u\rangle_{\mathcal{H}}=\tfrac{1}{4}\|P_-\hat u\|_{\mathcal{H}}^2- \tfrac{1}{2}\|(P_-\hat u)_y\|_{\mathcal{H}}^2.
\end{equation}
Therefore, combining with $2\|P_-E\|_{\mathcal{H}}\|P_-\hat u\|_{\mathcal{H}} \leq 4\|P_-E\|_{\mathcal{H}}^2+\frac{1}{4}\|P_-\hat u\|_{\mathcal{H}}^2$ yields
\begin{equation}
    \tfrac{d}{d\tau}\|P_-\hat u\|_{\mathcal{H}}^2 \leq -\tfrac{1}{2} \|(P_-\hat u)_y\|^2-\tfrac{1}{4} \|P_-\hat u\|_{\mathcal{H}}^2+4\|P_-E\|_{\mathcal{H}}^2.
\end{equation}
Thus, applying Lemma \ref{lem:error.proj} completes the proof.
\end{proof}

\bigskip

\begin{lemma}\label{lem:suit.subtraction}
There are some numeric constants $\varepsilon \in (0,10^{-1000})$ and $C>0$ such that each $\hat u^i$ satisfies
\begin{equation}
    \|\hat u^i(y,\tau)-2^{-\frac{1}{2}}b^i(\tau'-L)y\|_{\mathcal{H}}^2 \leq Cr^{-10}+Ce^{-\frac{1}{4}L}r^{-8}
\end{equation}
in $\tau \in [\tau'-L-4,\tau'-L]$.
\end{lemma}

\begin{proof}
We drop the index $i$ for simplicity. Since $\hat u =a+2^{-\frac{1}{2}}yb+P_-\hat u$, we have
\begin{equation}\label{eq:proj.u_y}
    \|\hat u_y\|_{\mathcal{H}}^2=\tfrac{1}{2} b^2+    \|(P_-\hat u)_y\|_{\mathcal{H}}^2.
\end{equation}
We consider $\alpha(\tau):=\|P_-\hat u\|_{\mathcal{H}}^2-r^{-2}b^2$. Combining Lemma \ref{lem:ODEs}, \eqref{eq:rough.Gaussian.L2}, (\ref{eqn:1}), (\ref{eqn:2}),  and \eqref{eq:proj.u_y} yields
\begin{equation}
    \alpha'\leq -\tfrac{1}{4}\|P_-\hat u\|_{\mathcal{H}}^2+C\varepsilon r^{-2}b^2+Cr^{-\frac{11}{2}}e^{-\frac{r^2}{8}}\leq -\tfrac{1}{4}\alpha+Cr^{-10}, \label{eqn:3}
\end{equation}
for small enough $\varepsilon$. Since the above inequality implies $\frac{d}{d\tau}e^{\frac{1}{4}\tau}\alpha \leq Cr^{-10}e^{\frac{1}{4}\tau}$,
\begin{equation}
e^{\frac{1}{4}\tau}\alpha(\tau) - e^{\frac{1}{4}(\tau'-4L)}\alpha(\tau'-4L) \leq Cr^{-10}e^{\frac{1}{4}\tau} 
\end{equation}
holds for $\tau \in [\tau'-4L,\tau']$. Since \eqref{eq:rough.Gaussian.L2} guarantees $\alpha(\tau'-4L) \leq Cr^{-8}$, we have
\begin{equation}\label{eq:stable.esti}
    \|P_-\hat u\|_{\mathcal{H}}^2 \leq r^{-2}b^2+ C r^{-8}e^{-\frac{1}{4}L}  + Cr^{-10}
\end{equation}
for $\tau \in [\tau'-3L,\tau']$, where $C$ is some numeric constant.

\bigskip

Next, we define $\beta(\tau):=a^2-r^{-4}(b^2+\|P_-\hat u\|_{\mathcal{H}}^2)$. We again combine Lemma \ref{lem:ODEs}, \eqref{eq:rough.Gaussian.L2}, and \eqref{eq:proj.u_y}, and choose small enough $\varepsilon$ so that we obtain
\begin{equation}
    \beta' \geq a^2   -C\varepsilon r^{-4}b^2-Cr^{-\frac{11}{2}}e^{-\frac{r^2}{8}}\geq \beta -Cr^{-\frac{11}{2}}e^{-\frac{r^2}{8}}.
\end{equation}
Hence, considering $\frac{d}{d\tau}e^\tau \beta$, we have
\begin{equation}
    e^{\tau'}\beta(\tau')-e^{\tau}\beta(\tau) \geq -Ce^{\tau}r^{-\frac{11}{2}}e^{-\frac{r^2}{8}}.
\end{equation}
Thus, if $\tau \in [\tau'-4L,\tau'-L]$ then
\begin{equation}\label{eq:unstable.esti}
    a^2\leq r^{-4}(b^2+\|P_-\hat u\|_{\mathcal{H}}^2)+Cr^{-8}e^{-L}+Cr^{-\frac{11}{2}}e^{-\frac{r^2}{8}}.
\end{equation}
Hence, adding \eqref{eq:stable.esti} and \eqref{eq:unstable.esti} yields
\begin{equation}
    \|\hat u-P_0\hat u\|_{\mathcal{H}}^2 \leq 2r^{-2}\|\hat u\|_{\mathcal{H}}^2+Cr^{-8}e^{-\frac{1}{4}L}+Cr^{-10}.
\end{equation}
Therefore, by remembering \eqref{eq:rough.Gaussian.L2}, we have
\begin{equation}\label{eq:minor.mode.esti}
    \|\hat u-P_0\hat u\|_{\mathcal{H}}^2 \leq Cr^{-10}+Cr^{-8}e^{-\frac{1}{4}L}
\end{equation}
in $\tau \in [\tau'-3L,\tau'-L]$. 
\bigskip

Now, we consider $\gamma(\tau):=b(\tau)-b(\tau'-L)$. Then, for sufficiently small $\varepsilon$, we have
\begin{equation}
    |\gamma'| + \tfrac{d}{d\tau}\|P_-\hat u\|_{\mathcal{H}}^2\leq C\varepsilon b^2+Cr^{-\frac{3}{2}}e^{-\frac{r^2}{8}}\leq Cr^{-8}
\end{equation}
by using Lemma \ref{lem:ODEs}, \eqref{eq:rough.Gaussian.L2}, and \eqref{eq:proj.u_y}. Since $\gamma(\tau'-L)=0$, 
 \begin{equation}
     |\gamma(\tau)|\leq \int_{\tau}^{\tau'-L}|\gamma'(s)|ds \leq 4r^{-8}-\|P_-\hat u(\cdot,\tau'-L)\|_{\mathcal{H}}^2+\|P_-\hat u(\cdot,\tau)\|_{\mathcal{H}}^2
 \end{equation}
 holds for $\tau \in [\tau'-L-4,\tau'-L]$. Thus, \eqref{eq:minor.mode.esti} implies
 \begin{equation}
     \|P_0\hat u(y,\tau)-2^{-\frac{1}{2}}b(\tau'-L)y\|_{\mathcal{H}}=|\gamma(\tau)|\leq Cr^{-8}.
 \end{equation}
 Hence, combining \eqref{eq:rough.Gaussian.L2} and \eqref{eq:minor.mode.esti} completes the proof.
\end{proof}

 \bigskip

\begin{theorem}\label{thm:improve.flat.L2}
There exist some numeric constants $\varepsilon \in (0,10^{-1000})$, $C_1> 1$ and $\sigma_2 \in (0,1)$ with the following significance. Suppose that there is $\tau_1 \leq  \tau_\varepsilon$ such that for each $\tau'\leq \tau_1$, $S_{\tau'}\overline{M}_\tau$ is $\delta$-linear with multiplicity $m$ at each $\tau'$ for some rotation $S_{\tau'}$, where $\delta\leq \sigma_2/L$ for some $L\geq 2$. Then, there exists a small angle $\theta$ with $|\theta |\leq C_1L\delta$ such that the profile bundle $\mathbf{u}$ of $S_{\tau'}\overline{M}_\tau$ is well-defined in $P^{\tau'}_{8,4L}$ and
\begin{equation}
    \max_{i=1,\cdots,m}\|u^i(y,\tau)-(\tan \theta) y\|_{L^2([-8,8])}   \leq C_1 (L^{\frac{5}{4}}\delta ^{\frac{1}{4}}+ L^{-1})\delta 
\end{equation}
holds for $\tau \in [\tau'-L-4,\tau'-L]$, where  $\mathbf{u}:=(u^1,\cdots,u^m)$.
\end{theorem}

\begin{proof}
By Lemma \ref{lem:error.proj}, $\mathbf{u}$ is well-defined in $P^{\tau'}_{8,4L}$, and we may assume $u^1 > u^i$ for every $i\geq 2$ without loss of generality. Let $\beta_i$ denote $2^{-\frac{1}{2}}b^i(\tau'-L)$ for each $i=1,\cdots,m$. If $\beta_1 \geq \beta_2$, then combining $u^1-u^2>0$ and $-(\beta_1-\beta_2)y\geq 0$ for $y\leq 0$ yields
\begin{equation}
    |\beta_1-\beta_2|\leq \|(\beta_1-\beta_2)y\|_{L^2([-2,0])} \leq  \| (u^1-u^2)-(\beta_1-\beta_2)y\|_{L^2([-2,0])}
\end{equation}
Thus, Lemma \ref{lem:suit.subtraction} implies
\begin{equation}\label{eq:angle.gap}
    |\beta_1-\beta_2| \leq Cr^{-10}+Ce^{-\frac{1}{4}L}r^{-8}.
\end{equation}
when $\beta_1 \geq \beta_2$. If $\beta_1\leq \beta_2$, then we can repeat the same process for $L^2([0,2])$ so that we obtain \eqref{eq:angle.gap}. Of course, we can obtain the same estimate for every $|\beta_1-\beta_i|$. Hence, recalling $L\geq 2$, Lemma \ref{lem:suit.subtraction} leads us to
\begin{equation}
    \|\hat u^i - \beta_1 y\|_{\mathcal{H}}^2 \leq Cr^{-10}+Ce^{-\frac{1}{4}L}r^{-8}\leq Cr^{-10}+CL^{-4}r^{-8}.
\end{equation}
Since $4r=(C_0L\delta)^{-\frac{1}{4}}$ as in Lemma \ref{lem:error.proj}, choosing $\theta$ by $\tan \theta:=\beta_1$ with $|\theta| \leq \|\hat u^1(\cdot,\tau'-L)\|_{\mathcal{H}}\ll 1$ completes the proof.
\end{proof}

\bigskip

 \subsection{Fast convergence}
In this final subsection, we show that the rescaled flow converges to a unique tangent flow at backward infinity at an exponential rate. 

\bigskip

We first recall the standard rotation matrix to state the following proposition.
\begin{equation}
    R_\theta:=\begin{pmatrix}
        \cos\theta & -\sin\theta \\
        \sin\theta & \cos\theta
    \end{pmatrix}.
\end{equation}

\begin{proposition}\label{prop:interior.esti}
There are some numeric constants $C_2>1$ and $\varepsilon_0>0$ with the following significance. Suppose that the graph $\overline{M}_\tau$ of a function $u \in C^{\infty}(P^0_{8,4})$ with $\|u\|_{C^2(P^0_{8,4})}\leq \varepsilon_0$ is a rescaled curve-shortening flow. Then, given $\theta \in (-\varepsilon_0,\varepsilon_0)$, the profile $v$ of the rotated flow $R_{-\theta}\overline{M}_\tau$ is well-defined in $P^0_{6,4}$ and the following holds
\begin{equation}
    \|v\|_{C^2(P^0_{4,2})}\leq C_2\sup_{ \tau \in [-4,0]}\|u(y,\tau)-(\tan \theta)y\|_{L^2([-8,8])}.
\end{equation}
\end{proposition}

 \begin{proof}
Choosing $\varepsilon_0$ small enough, $v$ is well-defined in $P^0_{6,4}$. Also, the area between $v$ and $\{x_2=0\}$ in a ball $B_r(0)$ is equivalent to the area between $u$ and $\{x_2=(\tan \theta) x_1\}$ in $B_r(0)$. Since the $L^1$-norm represents the area difference, we have
\begin{equation}
    \|v\|_{L^1([-6,6])} \leq \|u-(\tan \theta)y\|_{L^1([-8,8])} \leq 16\|u-(\tan \theta)y\|_{L^2([-8,8])}.
\end{equation}
Note that the second inequality comes from the standard Holder's inequality. On the other hand, the local $L^\infty$ estimate (cf. \cite[Theorem 6.17]{Lie:1996:SOP}) yields
 \begin{equation}
     \|v\|_{L^\infty(P^0_{5,3})}\leq C \sup_{ \tau \in [-4,0]}\|v\|_{L^2([-6,6])}
 \end{equation}
Also, the interior Schauder estimate (cf. \cite[Theorem 4.9]{Lie:1996:SOP}) implies
 \begin{equation}
     \|v\|_{C^{2,\alpha}(P^0_{4,2})}\leq C \|v\|_{L^\infty(P^0_{5,3})},
 \end{equation}
 for any $\alpha \in (0,1)$. Combining the above inequalities completes the proof.
 \end{proof}

 \bigskip

 \begin{lemma}\label{lem:iteration}
     There exist some numeric constants $\varepsilon \in (0,10^{-1000})$, $L\geq 2$, $C_3>1$, and $\delta_0 \in (0,1)$ with the following significance. Suppose that there is $\tau_1\leq \tau_\varepsilon$ such that for each $\tau'\leq \tau_1$, $S_{\tau'}\overline{M}_\tau$ is $\delta$-linear, where $\delta\leq \delta_0$, with multiplicity $m$ at $\tau'$ for some rotation $S_{\tau'}$. Then, for each $\tau'\leq \tau_1$, there exists a rotation $\bar S_{\tau'}$ such that $|\bar S_{\tau'}-S_{\tau'}| \leq C_3\delta$ holds and $\bar S_{\tau'}\overline{M}_\tau$ is $\frac{1}{2}\delta$-linear with multiplicity $m$ at $\tau'-L$. Moreover, the profile bundle $\mathbf{u}$ of $S_{\tau'}\overline{M}_\tau$ satisfies $\|\mathbf{u}\|_{C^2(P^{\tau'}_{3,4L})}\leq C_3\delta$.
 \end{lemma}
 
 \begin{proof}
 We recall $\tau_0,\theta$ in Theorem \ref{thm:improve.flat.L2}, and define $\bar S_{\tau'}=R_{-\theta}S_{\tau'}$. Then, we have $|\bar S_{\tau'}-S_{\tau'}| \leq CL\delta$. Also, Proposition \ref{prop:glued.domain} implies $\|\mathbf{u}\|_{C^2(P^{\tau'}_{3,4L})}\leq CL\delta$.

 \medskip

 Next, we combine Theorem \ref{thm:improve.flat.L2} and Proposition \ref{prop:interior.esti} so that we have
 \begin{equation}
      \|v\|_{C^{2}(P^{\tau'-L}_{4,2}\,)}\leq C_1C_2(L^{\frac{5}{4}}\delta ^{\frac{1}{4}}+ L^{-1})\delta.
 \end{equation}
Now, we fix $L:=4C_1C_2+2$ so that we have
\begin{equation}
      \|v\|_{C^{2}(P^{\tau'-L}_{4,2}\,)}\leq (C_1C_2L^{\frac{5}{4}}\delta_0^{\frac{1}{4}}+\tfrac{1}{4})\delta.
 \end{equation}
 Then, we choose $\delta_0 \leq \sigma_2/L$ to satisfy $C_1C_2L^{\frac{5}{4}}\delta_0^{\frac{1}{4}}\leq  4^{-1}$. This completes the proof.
 \end{proof}

 \bigskip

Now, we prove our main theorems \ref{thm:main.unique} and \ref{thm:main.conv.rate}. Note that Theorem \ref{thm:main.conv.rate} already implies Theorem \ref{thm:main.unique}, so we just prove Theorem \ref{thm:main.conv.rate}.

 \begin{proof}[Proof of Theorem \ref{thm:main.unique} and \ref{thm:main.conv.rate}]
 We recall $\delta_0$ and $C_3$ from Lemma \ref{lem:iteration}. Then, by the rough convergence theorem \ref{thm: rough convergence}, there exists negative enough $\tau_1 \leq \tau_\varepsilon$ such that for each $\tau'\leq \tau_1$, $S^1_{\tau'}\overline{M}_\tau$ is $\delta_0$-linear with multiplicity $m$ at $\tau'$ for some rotation $S^1_{\tau'}$. Then, by Lemma \ref{lem:iteration}, each $\tau'\leq \tau_2:=\tau_1-L$ has a rotation $S^2_{\tau'}$ such that  $S^2_{\tau'}\overline{M}_\tau$ is $(\delta_0/2)$-linear  at $\tau'$, and $|S^1_{\tau_1}-S^2_{\tau_2}|\leq C_3\delta_0$. 
 
 We repeat this process. Then, each $\tau'\leq \tau_k:=\tau_1-(k-1)L$ has a rotation $S^k_{\tau'}$ such that  $S^k_{\tau'}\overline{M}_\tau$ is $2^{1-k}\delta_0$-linear  at $\tau'$, and $|S^{k-1}_{\tau_{k-1}}-S^k_{\tau_k}|\leq 2^{2-k}C_3\delta_0$. Hence, there is the limit rotation $S_\infty \in SO(2)$ such that $|S_\infty -S^k_{\tau_k}|\leq 2^{2-k}C_3\delta_0$. 

\bigskip

 We may assume that $S_\infty=R_0$ is the identity matrix without loss of generality. Since the profile bundle $\mathbf{u}^k$ of $S_{\tau_k}^k\overline{M}_\tau$ satisfies $\|\mathbf{u}^k\|_{C^2(P^{\tau_k}_{3,4L})}\leq 2^{1-k}C_3\delta_0$ by Lemma \ref{lem:iteration}, rememerbing $|S^k_{\tau_k}-R_0|\leq 2^{2-k}C_3\delta_0$, the profile bundle $\mathbf{u}$ of $\overline{M}_\tau$ satisfies $\|\mathbf{u}\|_{C^2(P^{\tau_k}_{3,4L})}\leq 2^{-k}C$ for some $C$. Hence, there exist some $\delta>0$ only depending on $L$ and constant $C$ such that 
 \begin{equation}
     \|\mathbf{u}(\cdot,\tau)\|_{C^2(B_3(0))}\leq Ce^{5\delta\tau}
 \end{equation}
 holds for $\tau \leq \tau_1$. Hence, for any $\varepsilon>0$, the graphical radius lower bound theorem \ref{thm:grph.radius} implies \eqref{eq:main.graph.radius} for negative enough time. Also, given $R>0$, Theorem \ref{thm:grph.radius} yields
 \begin{equation}
     \|\mathbf{u}(\cdot,\tau)\|_{L^\infty(B_{2R}(0))}\leq Ce^{5\delta\tau}
 \end{equation}
 for some $C$ depending on $R$. Thus, the interior Schauder estimate (cf. \cite[Theorem 4.9]{Lie:1996:SOP}) guarantees \eqref{eq:main.fast.decay}. This completes the proof. 
 \end{proof}

\bigskip

\appendix
\section{Estimate of localized Gaussian density ratio}

\begin{lemma}\label{lem: est of loc.G.D.R}
   There exist $C>0$ and $\delta \in (0, 1)$ small with the following significance. 
   Let $a,b \in (0, \delta)$ and let $r > \delta^{-1}$. For any connected $C^2$ curve $\Ga$ in $B_r(0)\subset \R^2$ whose $\pa \Ga\subset \pa B_r(0)$ with curvature bound $\lvert \kappa \rvert \leq b r^{-1}$ and $d_0 = \min\ \{ \lvert x \rvert : x\in \Ga\} < \delta$, there hold
    \begin{align*}
        \frac{1}{\sqrt{4\pi}}\int_{\Ga \cap B_r(0)} \left[1 - a^2 (\lvert x \rvert^2 -2)\right]_+^3 
        e^{-\frac{\lvert x \rvert^2}{4}}\ ds 
        \leq 1 + Ca^2 + Cb\ ,
    \end{align*}
    and 
    \begin{align*}
        \frac{1}{\sqrt{4\pi}}\int_{\Ga \cap B_r(0)} \left[1 - a^2 (\lvert x \rvert^2 -2)\right]_+^3 
        e^{-\frac{\lvert x \rvert^2}{4}}\ ds \geq 
        1 - Cd_0^2 - Ca - Cb -C\exp(-\tfrac{r}{4}).
    \end{align*}
\end{lemma}

\begin{proof}
    Let $\ga(s)$ be an arc-length parametrization of $\Ga$ such that $\ga(0)$ is a closest point to the origin. 
    Define $d(s) = \lvert \ga(s) \rvert$. We first estimate the error of the localization function. For small $a \in (0, \delta)$, there is $C>0$ such that in $B_r(0)$
    \begin{align*}
        \left[1 - a^2 (\lvert x \rvert^2 -2)\right]_+^3 \le 1+Ca^2
    \end{align*}
    and in $B(0, 1/\sqrt{a}) \cap B_r(0)$
    \begin{align*}
        \left[1 - a^2 (\lvert x \rvert^2 -2)\right]_+^3 \geq 1 - Ca.
    \end{align*}
    Recall that $\pa_s d^2(0) = 2\langle \ga, \mathbf{t} \rangle(0) = 0$ and $\pa_{ss} d^2 = 2 \big(1 + \ka \langle \ga, \mathbf{n} \rangle\big)$. Using $|\kappa|\le br^{-1}$ in $B_r(0)$,
    \begin{align*}
        d_0^2 + (1 - b)s^2 \leq & d^2(s) \leq d_0^2 + (1 + b)s^2.
    \end{align*}
    Note that for $b < \delta$ is small, $d^{-1}(t)$ has exactly two points for any $t\in (d_0, r]$. 
    Thus, 
\begin{align}
    &\frac{1}{\sqrt{4\pi}}\int_{\Ga\cap B_r(0)} \left[1 - a^2 (\lvert x \rvert^2 -2)\right]_+^3 \exp(-\tfrac{d^2}{4})\ ds \notag\\
     &\quad \leq  \frac{\exp (-\tfrac{d_0^2}{4})}{\sqrt{4\pi}}\int_{-\infty}^\infty (1 + Ca^2) \exp ( -\tfrac{(1 -b)s^2}{4})\ ds \notag\\
     &\quad = (1 + Ca^2)\frac{\exp (-\tfrac{d_0^2}{4})}{\sqrt{1-b}}\leq 1 + Ca^2 + Cb. \label{eqn1: appendix}
\end{align}
Next we prove the lower bound. Let $m = \min \{r, 1/\sqrt{a}\}$. 
\begin{align}
     &\frac{1}{\sqrt{4\pi}}\int_{\Ga\cap B_r(0)} \left[1 - a^2 (\lvert x \rvert^2 -2)\right]_+^3 \exp(-\tfrac{d^2}{4})\ ds \notag\\
     &\quad\geq \frac{\exp (-\tfrac{d_0^2}{4})}{\sqrt{4\pi}}\int_{-m}^{m}(1-Ca)\exp(-\tfrac{(1+b)s^2}{4})\ ds \notag\\
     &\quad\geq \exp (-\tfrac{d_0^2}{4})(1 - Ca)(1 - Cb)(1 - C\exp(-\tfrac{m}{4}) ) \notag\\
     &\quad \geq  1 - C d_0^2 - Ca - Cb - C\exp(-\tfrac{r}{4}).  \label{eqn2: appendix}
\end{align}
Here we use $\exp(-x^2/4) \leq \exp(-x/4)$ for $x\geq 1$ to estimate the Gauss error function $\int_{\sqrt{1 + b}m}^\infty \exp(-x^2/4)\ dx$ and we absorb $\exp(-m/4)$ by $Ca + Ce^{-r/4}$ provided that $a < \delta$ is small. Then, (\ref{eqn1: appendix}) and (\ref{eqn2: appendix}) imply the desired conclusion.

\end{proof}

\subsection*{Acknowledgments}
KC has been supported by the KIAS Individual Grant MG078902, an Asian Young Scientist Fellowship, and the National Research Foundation(NRF) grant funded by the Korea government(MSIT) (RS-2023-00219980); DS has been supported by the National Research Foundation of Korea (NRF) grant funded by the Korea Government (MSIT) (No.2021R1C1C2005144); WS has been supported by the Taiwan MOST postdoctoral research abroad program 111-2917-I-564-002. He would like to thank University of Warwick and NCTS for providing the wonderful research environments during the preparation of this paper; KZ would like to thank Nick Edelen for the helpful discussion on GMT.

\bibliographystyle{amsplain} 
\bibliography{CSF}

\end{document}